\documentclass[11pt]{amsart}
\usepackage[utf8]{inputenc}
\usepackage{amssymb,amsmath, amsthm, mathabx}
\usepackage{color, enumerate}
\usepackage{hyperref}
\usepackage{caption}
\usepackage{subcaption}
\usepackage{a4wide}
\usepackage{accents}
\usepackage{tikz}
\usepackage{graphicx}
\usepackage{chngcntr}

\usepackage{xifthen}

\def\hor{\mathrm{hor}}
\def\ver{\mathrm{ver}}

\def\Z{\mathrm{Z}}

\def\E{\mathbf{E}}
\def\P{\mathbf{P}}

\def\rw{\omega}

\def\curv{\sigma}

\def\Exp{\mathrm{Exp}}

\def\E{\bfE}
\def\P{\bfP}

\def\M{\mathrm{M}}

\def \G{\mathrm G}

\def\ShpMin{\zeta}

\def\dd{\mathrm d}

\def\shp{\gamma}

\newcommand{\bbE}{\mathbb{E}}

\newcommand{\bbR}{\mathbb{R}}

\newcommand{\bbZ}{\mathbb{Z}}

\newcommand{\bfE}{\mathbf{E}}

\newcommand{\bfP}{\mathbf{P}}
\newcommand{\bfQ}{\mathbf{Q}}

\newcommand{\bfp}{\mathbf{p}}

\newcommand{\bfu}{\mathbf{u}}
\newcommand{\bfv}{\mathbf{v}}

\newcommand{\bfx}{\mathbf{x}}
\newcommand{\bfy}{\mathbf{y}}

\newcommand{\mrL}{\mathrm{L}}

\newcommand{\mrX}{\mathrm{X}}

\newcommand{\sG}{\mathcal{G}}

\newcommand{\Fabs}[4]
{
\ifthenelse{\isempty{#2}}
{
{#1}_{#3}^{#4}
}
{
{#1}_{#3}^{#4}(#2)
}
}

\newcommand{\one}{\mathbf{1}}

\newcommand{\lc}{\lceil}
\newcommand{\rc}{\rceil}
\newcommand{\lf}{\lfloor}
\newcommand{\rf}{\rfloor}
\newcommand{\Var}{\mathbf{Var}}

\def\dd{\mathrm d}

\definecolor{darkblue}{rgb}{.1,.1,.6}

\newcommand{\wt}[1]{\widetilde{#1}}

\newtheorem{thm}{Theorem}[section]
\newtheorem{prop}[thm]{Proposition}
\newtheorem{lem}[thm]{Lemma}

\theoremstyle{remark}
\newtheorem{rem}{Remark}[section]

\counterwithin{figure}{section}
\counterwithin{equation}{section}

\usepackage{todonotes, xargs}
\newcommandx{\note}[2][1=]{\todo[linecolor=yellow,backgroundcolor=yellow!25,bordercolor=yellow,#1]{#2}}

\newcommand{\norm}[1]{\lVert{#1}\rVert}

\renewcommand{\thesubsection}{\thesection.\Alph{subsection}}

\allowdisplaybreaks

\title[Central moment bounds in exponential LPP]
{Coupling derivation of optimal-order central moment bounds in exponential last-passage percolation}

\author[E.~Emrah]{Elnur Emrah}
\address{Elnur Emrah\\ University of Bristol \\ School of Mathematics \\ Bristol \\ United Kingdom}
\email{e.emrah@bristol.ac.uk}
\urladdr{https://sites.google.com/view/elnur-emrah}
\thanks{E.\ Emrah was supported by the grant KAW 2015.0270 from the Knut and Alice Wallenberg Foundation and by the EPSRC grant EP/W032112/1.} 

\author[N.~Georgiou]{Nicos Georgiou}
\address{Nicos Georgiou\\ University of Sussex\\ Department of  Mathematics \\ Falmer Campus\\ Brighton BN1 9QH\\ UK.}
\email{N.Georgiou@sussex.ac.uk}
\urladdr{http://www.sussex.ac.uk/profiles/329373} 
\thanks{N. Georgiou was partially supported by the Dr.\ Perry James Browne Research Centre.} 

\author[J.~Ortmann]{Janosch Ortmann}
\address{Janosch Ortmann\\ Universit\'e du Qu\'ebec \`a Montr\'eal\\ Case postale 8888, succ. Centre-ville\\ Montr\'eal (QC) H3C 3P8\\ Canada}
\email{ortmann.janosch@uqam.ca}
\urladdr{http://crm.umontreal.ca/~ortmann/}
\thanks{}


\begin{document}
\begin{abstract}
We introduce new probabilistic arguments to derive optimal-order central moment bounds in planar directed last-passage percolation. Our technique is based on couplings with the increment-stationary variants of the model, and is presented in the context of i.i.d.\ exponential weights for both zero and near-stationary boundary conditions. A main technical novelty in our approach is a new proof of the left-tail fluctuation upper bound with exponent $3/2$ for the last-passage times. 
\end{abstract}

\maketitle



\section{Introduction}

\subsection{Exponential LPP}
\label{Ss:ExpLPP}

This article studies the \emph{exponential last-passage percolation} (LPP) on the lattice quadrant. The most standard version of the model can be introduced as follows. 
With $\bbZ^2$ equipped with the coordinatewise partial order, for each $\bfu, \bfv \in \bbZ^2$, let $\Pi_{\bfu, \bfv}$ denote the set of all maximally large, totally ordered subsets of the integer grid $\{\bfp \in \bbZ^2: \bfu \le \bfp \le \bfv\}$. Each $\pi \in \Pi_{\bfu, \bfv}$ can be viewed as a nearest-neighbor path from $\bfu$ to $\bfv$ with increments $(1, 0)$ and $(0, 1)$. Given IID $\Exp(1)$-distributed \emph{weights} $\rw = \{\rw_{\bfv}: \bfv \in \bbZ_{>0}^2\}$, the \emph{last-passage times} are defined by 
\begin{align}
\G_{\bfu, \bfv} = \max_{\pi \in \Pi_{\bfu, \bfv}} \left\{\sum_{\bfp \in \pi}\rw_{\bfp}\right\} \quad \text{ for } \bfu, \bfv \in \bbZ_{>0}^2. \label{EBlkLPP}
\end{align}
By definition, $\G_{\bfu, \bfv} = -\infty$ if the inequality $\bfu \le \bfv$ fails.  When $\bfu \le \bfv$, a.s., there exists a unique maximizer $\pi_{\bfu, \bfv} \in \Pi_{\bfu, \bfv}$ in \eqref{EBlkLPP}, which is called the \emph{geodesic} from $\bfu$ to $\bfv$. Typically of interest are statistical properties of last-passage times and geodesics, particularly on large scales. 

The study of the exponential LPP has been both motivated and enabled by its well-known connections to several other prominent models, which we briefly mention. Abbreviating 
\begin{align}
\label{E:BlkLPP1}
\G_{\bfv} = \G_{(1, 1), \bfv} \quad \text{ for } \bfv \in \bbZ_{>0}^2, 
\end{align}
consider the random cluster of unit squares 
\begin{align}
\sG_t = \{\bfx \in \bbR_{>0}^2: \G_{\lc \bfx \rc} \le t\} \label{E:CGM}
\end{align} 
evolving over time $t \ge 0$. The process $(\sG_t)_{t \ge 0}$ is Markovian and coincides in distribution with the growth process of the \emph{corner growth model}, which first appeared in the seminal article \cite{Rost_81} as a geometric interpretation of the \emph{totally asymmetric simple exclusion process} (TASEP) with the step initial condition. The latter is a prototypical interacting particle system where particles labeled with $\bbZ_{>0}$ perform independent, unit-step, rightward jumps at rate $1$ on sites of $\bbZ$, subject to the exclusion rule that at most one particle is allowed per site at any time. The initial configuration places particle $j$ at site $-j$. Then $\G_{(i, j)}$ is equal in distribution to the time when particle $j$ executes its $i$th jump. Another equivalent model is a tandem of M/M/1 queues \cite{Glyn_Whit_91}, namely, a series of single-server queues with exponential service times and first-in first-out discipline, attending a stream of customers with exponential inter-arrival times. In the queueing context, $\G_{(i, j)}$ corresponds to the time when customer $j$, having just been serviced at queue $i$, joins the next queue $i+1$. The preceding connections are relatively straightforward consequences of expressing formula \eqref{EBlkLPP} in a recursive form. There is also the following deeper link to the random matrix theory. Let $\mrX_{m, n}$ be an $m \times n$ random matrix filled with independent complex Gaussian entries of mean zero and variance one. The (law of the) 
$n \times n$ random Hermitian matrix $\mrL_{m, n} = \mrX^*_{m, n}\mrX_{m, n}$ where * denotes the conjugate transpose is called the \emph{Laguerre {\rm(\emph{Wishart})} unitary ensemble} (LUE) with variance one (and parameter $m$). Then $\G_{(m,n)}$ coincides in distribution with the largest eigenvalue of $\mrL_{m, n}$ \cite{Joha_00}. 

\subsection{Stochastic growth and KPZ universality}
\label{Ss:KPZ}

Through its interpretation in \eqref{E:CGM} as a growth process, the exponential LPP also ties to a broader stream of 
research on stochastic growth, which has been a major topic of probability since the introduction of Eden's model \cite{Eden_61} and the first-passage percolation (FPP) \cite{Hamm_Wels_65} around sixty years ago. The sustained interest is partly driven by the sweeping predictions of the Kardar-Parisi-Zhang (KPZ) universality \cite{Kard_Pari_Zhan_86} for the asymptotic growth statistics. In the best understood case of 
dimension two in particular, 
under fairly broad conditions, 
the long-time fluctuations and spatial correlations of a growth process are expected to scale with exponents $1/3$ and $2/3$, respectively. Furthermore, canonical limit objects such as the Tracy-Widom distributions \cite{Trac_Wido_94},  Airy processes \cite{Prah_Spoh_02},  KPZ fixed point \cite{Mate_Quas_Reme_21} and directed landscape \cite{Dauv_Ortm_Vira_22} are expected to universally arise in the appropriate scaling limits. These predictions are believed to hold in wide generality in many other contexts as well, for the suitable analogues of a growth process, such as the free energy of a planar random polymer and the cumulative particle current of a one-dimensional exclusion process. Further discussion can be found in the review articles \cite{Corw_12, Quas_12, Quas_Spoh_15}. 

The rigorous study of the KPZ universality has progressed predominantly through a small collection of integrable (exactly solvable) settings where explicit calculations are possible.  Some main examples of integrable growth models come from planar LPP with special weight distributions \cite{Baik_Rain_01a, Joha_00b} with the most studied cases being the exponential, geometric and Poisson LPP.  It is of utmost interest but presently out of reach to prove any element of the KPZ universality for a large class of growth processes, for example, for the LPP with general IID weights subject to mild conditions. The current state of the art for accessing the KPZ limit objects is rooted in the pioneering works \cite{Baik_Deif_Joha_99, Joha_00},  and entails a sophisticated and expanding suite of integrable techniques drawing from combinatorics, representation theory, random matrices and asymptotic analysis; see the survey articles \cite{Boro_Gori_16, Boro_Petr_14, Joha_Lec_06, Zygo_22}. These techniques have yielded an impressively detailed picture of limit fluctuations for various integrable LPP to the point of identifying the limit of the entire LPP process \cite{Dauv_Nica_Vira_23, Dauv_Ortm_Vira_22,  Mate_Quas_Reme_21}. 

\subsection{Coupling approach to KPZ fluctuation exponents}
\label{Ss:CoupKPZ}

Some aspects of the KPZ universality such as the scaling exponents (although not any of the KPZ limits at the moment) can also be captured via more 
probabilistic arguments based on couplings and invariant measures. This coupling approach was first developed for the Poisson LPP in the seminal works \cite{Cato_Groe_05, Cato_Groe_06} of Cator and Groeneboom. It was subsequently adapted in strenghtened form to the exponential LPP by Bal\'{a}zs, Cator and Sepp\"{a}l\"{a}inen \cite{Bala_Cato_Sepp_06}. 
Since the present work refines the coupling method in the context of the exponential LPP, in what follows, we first state some results from \cite{Bala_Cato_Sepp_06} and briefly explain the underlying approach. More detailed expositions of the coupling method can be found in the lecture notes \cite{Cato_Pime_11, Sepp_18_CGM}.

We begin with recalling the stationary version of the exponential LPP introduced in \cite{Bala_Cato_Sepp_06}. Given a (boundary) parameter $z \in (0, 1)$, consider independent weights $\rw^{z} = \{\rw^{z}_\bfv: \bfv \in \bbZ^2_{\ge 0}\}$ that extend the \emph{bulk} weights $\rw$ defined on $\bbZ_{>0}^2$ to the nonnegative axes. Thus,    
\begin{align}
\label{E:Coup}
\rw_{\bfv}^z = \rw_{\bfv} \sim \Exp(1) \quad \text{ for } \bfv \in \bbZ^2_{>0}. 
\end{align}
Assume that the marginals of the boundary weights are given by  
\begin{align}
\label{E:wBd}
\begin{split}
\rw_{(i, 0)}^z \sim \Exp(z) \ \text{ for } i \in \bbZ_{>0} \quad \text{ and } \quad \rw_{(0, j)}^z \sim \Exp(1-z) \ \text{ for } j \in \bbZ_{>0},  
\end{split}
\end{align}
and set $\rw_{(0, 0)}^z = 0$. As in \eqref{EBlkLPP}, define the last-passage times of the $\rw^z$-weights by 
\begin{align}
\label{E:LPPwBd}
\G^z_{\bfu, \bfv} = \max_{\pi \in \Pi_{\bfu, \bfv}} \left\{\sum_{\bfp \in \pi}\rw_{\bfp}^z\right\} \quad \text{ for } \bfu, \bfv \in \bbZ_{\ge 0}^2. 
\end{align}
Also, in the case $\bfu \le \bfv$, let $\pi_{\bfu, \bfv}^z$ denote the a.s.\ unique geodesic in \eqref{E:LPPwBd}. 

A tractable feature of the preceding LPP model is that the last-passage times with the initial vertex at the origin, 
\begin{align}
\label{E:IncStLPP}
\G_\bfv^z = \G_{(0, 0), \bfv}^z \quad \text{ for } \bfv \in \bbZ_{\ge 0}^2, 
\end{align}
enjoy the Burke property \cite[Lemma 4.2]{Bala_Cato_Sepp_06}. The aspect of this property relevant to the current work is that the distributional identity   
\begin{align}
\label{EIncDis}
\begin{split}
&\{\G^{z}_{\bfu + (i, 0)}-\G^{z}_{\bfu + (i-1, 0)}: i \in \bbZ_{>0}\} \cup \{\G^{z}_{\bfu + (0, j)}-\G^z_{\bfu + (0, j-1)}: j \in \bbZ_{>0}\} \\
&\stackrel{\text{dist}}{=} \{\rw^z_{(i, 0)}: i \in \bbZ_{>0}\} \cup \{\rw^z_{(0, j)}: j \in \bbZ_{>0}\} 
\end{split}
\end{align}
holds for each $\bfu \in \bbZ_{\ge 0}^2$. It follows from \eqref{EIncDis} and a deterministic property of the last-passage times recorded, for example, in \cite[Lemma A.1]{Sepp_18_CGM} that the $\G^z$-process has stationary increments:  
\begin{align}
\label{EIncSt}
\G^{z}_{\bfv + \bfu}-\G^z_{\bfu} \stackrel{\text{dist}}{=} \G^{z}_{\bfv} \quad \text{ for } \bfu, \bfv \in \bbZ_{\ge 0}^2. 
\end{align}
This stationarity will also be the technical basis of our work. 
As a useful consequence of stationarity and the knowledge of the boundary weights, one can compute the last-passage means explicitly as follows. Writing 
\begin{align}
\label{EMeanF}
\M^z_{\bfx} = \frac{x}{z} + \frac{y}{1-z} \quad \text{ for } \bfx = (x, y) \in \bbR_{\ge 0}^2,     
\end{align}
one obtains from \eqref{EIncSt} and \eqref{E:wBd} that 
\begin{align}
\label{EMean}
\begin{split}
\E[\G^z_{(m, n)}] &= \E[\G^z_{(m, n)}-\G^z_{(0, n)}] + \E[\G^z_{(0, n)}] \\ 
&= \E[\G^z_{(m, 0)}] + \E[\G^z_{(0, n)}] = \M^z_{(m, n)} \quad \text{ for } m, n \in \bbZ_{\ge 0}. 
\end{split}
\end{align}

Through the Burke property, the variance of $\G^z_\bfv$ relates to the expected total weight that the geodesic $\pi_\bfv^z = \pi_{(0, 0), \bfv}^z$ collects on either boundary \cite[Lemma 4.6]{Bala_Cato_Sepp_06}: For $m, n \in \bbZ_{\ge 0}$, 
\begin{align}
\label{E:VarId}
\begin{split}
\Var[\G^z_{(m, n)}] 
&= -\frac{m}{z^2}+\frac{n}{(1-z)^2}+\frac{2}{z}\E\bigg[\sum_{i = 0}^m \rw^{z}_{(i, 0)} \cdot \one_{\{(i, 0) \in \pi_{(m, n)}^z\}}\bigg] \\ 
&= \frac{m}{z^2}-\frac{n}{(1-z)^2} + \frac{2}{1-z}\E\bigg[\sum_{j=0}^n \rw_{(0, j)}^z\cdot \one_{\{(0, j) \in \pi_{(m, n)}^z\}}\bigg].    
\end{split}
\end{align}
These identities are the analogues of the variance identity recorded in \cite[Theorem 2.1]{Cato_Groe_06} for the Poisson LPP. From \eqref{E:VarId}, it is possible to extract the order of the variance with $(m, n) = (m_k, n_k)$ growing to infinity in a given direction as $k \to \infty$. In particular, if $m_k = \lf z^2 k \rf$ and $n_k = \lf (1-z)^2 k \rf$ for $k \in \bbZ_{>0}$ then 
\begin{align}
\label{E:VarBd}
c_0 k^{2/3} \le \Var[\G^z_{(m_k, n_k)}] \le C_0 k^{2/3} \quad \text{ for } k \ge k_0
\end{align}
for some $z$-dependent positive constants $C_0$, $c_0$ and $k_0$ \cite[Theorem 2.1]{Bala_Cato_Sepp_06}. Thus, the fluctuations of $\G^z_{(m_k, n_k)}$ are of order $k^{1/3}$ as in the KPZ class. For the Poisson LPP, bounds analogous to \eqref{E:VarBd} previously appeared in \cite[Corollaries 4.3 and 6.2]{Cato_Groe_06}. As it relates to the topic of the current work, we also remark here that the lower bound in \eqref{E:VarBd} together with Jensen's inequality readily implies correct-order lower bounds for higher central moments: For each $p \in [2, \infty)$,     
\begin{align}
\label{E:StCMLB}
\E[|\G^z_{(m_k, n_k)}-\M_{(m_k, n_k)}^z|^p] \ge c_0^p k^{p/3} \quad \text{ for } k \ge k_0.  
\end{align}

The direction (with $L^1$ normalization) $\xi^z$ of the special vector $(z^2, (1-z)^2)$ chosen above is called the \emph{characteristic direction} in reference to the closely related characteristic curves of Burgers' equation 
\begin{align}
\label{E:Burgers}
\partial_t \rho(t, x) + (1-\rho(t, x)) \partial_x \rho(t, x) = 0 \quad \text{ for } t \ge 0 \text{ and } x \in \bbR 
\end{align}
with the initial condition 
\begin{align}
\label{E:Burgers-IC}
\rho(0, x) = 1-z \quad \text{ for } x \in \bbR. 
\end{align}
The system \eqref{E:Burgers}-\eqref{E:Burgers-IC} describes the evolution of the macroscopic particle density $\rho$ of the TASEP started from the initial configuration where each site of $\bbZ$ is independently empty with probability $z$; see \cite{Bala_Cato_Sepp_06} for details. In terms of \eqref{E:VarId}, the characteristic direction is the unique direction along which the first two terms on the right-hand sides essentially cancel out. For any other direction, these terms together contribute linearly in $k$ and, consequently, the fluctuations are of order $k^{1/2}$. Thus, the $\G_\bfv^z$-process exhibits a mix of KPZ and Gaussian behavior, which will also be reflected in our central moments in Section \ref{S:Res}. 

Via the coupling of the weights in \eqref{E:Coup}, the bulk LPP process $\G$ can be studied by comparison with the increment-stationary version $\G^z$ with advantageously chosen boundary parameter. Through an early instance of this coupling approach, \cite[Corollary 7.2]{Bala_Cato_Sepp_06} derived tail bounds for the $\G$-process, which can be equivalently stated as the following central moment\footnote{In \eqref{E:CMBd}, the centering $k$ is different from the mean $\bbE[\G_{(m_k, n_k)}]$ at least for large enough $k$; see Lemma \ref{LIncRTLB} ahead. Thus, the term \emph{central moment} is used in this work in a more general sense than usual.} bound: With $1 \le p < 3/2$ and the sequence $(m_k, n_k)$ as above,     
\begin{align}
\label{E:CMBd}
\E[|\G_{(m_k, n_k)}-k|^p] \le C_0 k^{p/3} \quad \text{ for } k \in \bbZ_{>0}
\end{align}
for some constant $C_0 > 0$ depending on $p$ and $z$. 
Again for the Poisson LPP, the earlier work \cite[p.\ 24]{Cato_Groe_06} observed an analogous bound in the case $p = 1$. 

Further adaptations applied the coupling method in various KPZ-class models including the asymmetric simple exclusion process (ASEP) \cite{Bala_Sepp_09, Bala_Sepp_10}, the O'Connell--Yor polymer with fixed temperature \cite{Sepp_Valk_10} and in the intermediate disorder regime \cite{More_Sepp_Valk_14}, the KPZ equation \cite{Bala_Quas_Sepp_11}, the log-gamma polymer \cite{Sepp_12_corr} and its inhomogeneous 
generalization \cite{Kang_19_PhD}, the beta, inverse-beta and strict-weak  polymers \cite{Chau_Noac_18b} and the Bernoulli LPP with strict-strict paths \cite{Ciec_Geor_19}. In the span of these works, the strength of the central moment bounds (in comparison to \eqref{E:VarBd} and \eqref{E:CMBd}) advanced mainly in the following respects. First, in the bulk setting, article \cite{More_Sepp_Valk_14} discovered a coupling argument for optimally bounding the first absolute central moment from below, which again yields optimal-order lower bounds for all higher central moments via Jensen's inequality. (Proposition \ref{P:BlkCMLB} ahead records these these lower bounds for the exponential LPP). 
Second, the PhD thesis \cite{Kang_19_PhD} extended the upper limit of the power $p$ in the bulk central moments upper bounds from $3/2$ up to $2$ with a logarithmic weakening at $2$. 

\subsection{Recent developments in the coupling approach}
\label{Ss:Coup20+}

Until 2020, the coupling method could not produce good upper bounds for the central moments beyond the variance. Sharp right-tail upper bounds 
were already available due to superadditivity \cite{Sepp_98}; the difficulty was the lack of sufficient control over the left-tail fluctuations. 
A significant advance came in \cite{Noac_Soso_22}. This article developed higher-degree cumulant identities for the O'Connell--Yor polymer, generalizing the variance identity in \cite{Sepp_Valk_10}. From their cumulant identities, the authors then derived nearly optimal upper bounds for all central moments in the increment-stationary case and around the characteristic direction. The follow-up work of the authors \cite{Noac_Soso_22b} tailored their technique to the four basic integrable lattice polymers (log-gamma, beta, inverse-beta, strict-weak) using the framework of \cite{Chau_Noac_18b}. For ease of comparison later, we mention that the counterparts of the central moment bounds from \cite{Noac_Soso_22, Noac_Soso_22b} would (up to minor differences) read as follows for the exponential LPP: Fix $A \ge 0$, $p \in (0, \infty)$, $z \in (0, 1)$ and $\epsilon > 0$, and assume that $|m_k-z^2k|\le Ak^{2/3}$ and $|n_k-(1-z)^2k| \le Ak^{2/3}$ for $k \in \bbZ_{>0}$. Then there exists a constant $C_0$ (depending on $A$, $p$, $z$ and $\epsilon$) such that  
\begin{align}
\label{Eq-NS}
\E[|\G^z_{(m_k, n_k)}-\M^z_{(m_k, n_k)}|^p] \le C_0 k^{p/3+\epsilon} \quad \text{ for } k \in \bbZ_{>0}, 
\end{align}
which misses the correct-order by an (arbitrarily small) $\epsilon$ power. 

More recently, the coupling method has seen rapid improvement with the utilization of certain MGF identities originally due to Rains \cite{Rain_00} in the context of integrable LPP models on the lattice quadrant. To state Rains' identity for the exponential LPP, let $w \in (0, 1)$ be a new boundary parameter and consider independent weights $\rw^{w, z} = \{\rw_\bfv^{w, z}: \bfv \in \bbZ_{\ge 0}^2\}$ that coincide with the $\rw^z$-weights except on the horizontal boundary where 
\begin{align}
\label{E:Coup-2}
\rw^{w, z}_{(i, 0)} \sim \Exp(w) \quad \text{ for } i \in \bbZ_{>0}.  
\end{align}
The corresponding last-passage times are then defined by  
\begin{align}
\label{E:LPPwBd-2}
\G^{w, z}_{\bfu, \bfv} = \max_{\pi \in \Pi_{\bfu, \bfv}} \left\{\sum_{\bfp \in \pi}\rw_{\bfp}^{w, z}\right\} \quad \text{ for } \bfu, \bfv \in \bbZ_{\ge 0}^2. 
\end{align}
As before, for the last-passage times from the origin, we drop the initial vertex and write
\begin{align}
\label{E:Lpp2Bd}
\G_{\bfv}^{w, z} = \G_{(0, 0), \bfv}^{w, z} \quad \text{ for } \bfv \in \bbZ^2_{\ge 0}. 
\end{align}
Recall the definition \eqref{EMeanF}. Then, as a special case of \cite[Corollary 3.4]{Rain_00}, one has the formula    
\begin{align}
\label{E:Rains}
\E[\exp\{(w-z) \G^{w, z}_{(m, n)}\}] = \exp\bigg\{\int_z^w \M^t_{(m, n)} \dd t\bigg\} = \bigg(\frac{w}{z}\bigg)^{m} \bigg(\frac{1-z}{1-w}\bigg)^{n}
\end{align}
for $\bfv = (m, n) \in \bbZ_{\ge 0}^2$. As noted in \cite[Remark 2.4]{Emra_Janj_Sepp_23}, identity \eqref{E:Rains} includes the variance identity \eqref{E:VarId}, versions of which have been the technical centerpiece of the coupling approach since the seminal article \cite{Cato_Groe_06}. 

Identity \eqref{E:Rains} as well as its counterparts for other integrable LPP models was derived in \cite{Rain_00} from the determinantal formulas developed in \cite{Baik_Rain_01a} for the one-point distribution of the last-passage times. In 2020, the preprint \cite{Emra_Janj_Sepp_20}\footnote{This is a joint work in progress of the first author.} reproved \eqref{E:Rains} via a simple change-of-measure argument using the stationarity property \eqref{EIncSt} and the independence of the boundary weights \eqref{E:wBd}. Combining \eqref{E:Rains} with other elements of the coupling method, \cite{Emra_Janj_Sepp_20} obtained sharp moderate deviation bounds for last-passage times and optimal-order fluctuation bounds for geodesics. The geodesic results were later generalized and developed into a separate article \cite{Emra_Janj_Sepp_23}. There are two related key points conveyed in the works \cite{Emra_Janj_Sepp_20, Emra_Janj_Sepp_23}. First, the coupling method can reach strong bounds previously accessible only through deeper and more delicate tools of integrable probability. Second, with the improvements of \cite{Emra_Janj_Sepp_20, Emra_Janj_Sepp_23}, the coupling method remains adaptable to a variety of settings besides the exponential LPP. Indeed, since the initial appearance of this work in 2022, independent groups have derived and employed the analogues of \eqref{E:Rains} to obtain fluctuation bounds for several LPP models \cite{Busa_Sepp_Sore_24, Groa_Janj_Rass_21}, lattice polymers \cite{Land_Soso_23, Xie_22}, the O'Connell--Yor polymer \cite{Land_Soso_23, Land_Soso_24} and its generalization to a system of interacting diffusions \cite{Land_Noac_Soso_23, Land_Soso_23}, the stochastic six vertex model and the asymmetric simple exclusion process (ASEP) \cite{Land_Soso_24b}. The preceding settings with the exception of LPP models are not covered by \cite{Rain_00} where the MGF identities of the form \eqref{E:Rains} were first recorded. A main reason for the amenability of the above models to the coupling approach is the presence of product-form invariant measures, which in our setting translates to the increment-stationary process having independent boundary weights as described in \eqref{E:wBd}. It is worth emphasis that the interacting diffusion model of \cite{Land_Noac_Soso_23} is not expected to possess integrable structure beyond its invariant measures. Therefore, 
this model cannot presently be studied through the finer techniques of integrable probability involving Fredholm determinants. 
   
\subsection{Contribution of the present work}
\label{Ss:Cont}

As discussed in Subsection \ref{Ss:Coup20+}, with the initial impetus from \cite{Emra_Janj_Sepp_20, Emra_Janj_Sepp_23} and a flurry of subsequent works, the coupling approach to KPZ fluctuations has recently been reinvigorated through the incorporation of Rains-type MGF identities. The current article contributes to this thread of research by demonstrating that the coupling method can also access optimal-order central moment bounds. Although prior works \cite{Emra_Janj_Sepp_20, Emra_Janj_Sepp_23} also relied strongly on identity \eqref{E:Rains} to strengthen the coupling approach, their focus was on right-tail bounds and geodesic fluctuations whereas our motivation is to achieve strong left-tail bounds. It is worth remarking here that fluctuation bounds with sufficient decay are often useful for probing deeper properties of last-passage times, for example, tightness \cite{Chhi_Ferr_Spoh_18} and time correlations \cite{Basu_Gang_21,  Ferr_Occe_19}. Indeed, in one broad approach to LPP, one typically inputs tail bounds at the outset and then proceeds with robust geometric arguments. See the articles \cite{Basu_Gang_21, Basu_Gang_Hamm_Hegd_22} for some representative of this geometric approach and the thesis \cite{Hegd_21} for a nice introduction. An upshot of our work is that the aforementioned inputs can be derived via the coupling method, which offers a soft probabilistic alternative to the earlier arguably more intricate proofs briefly covered in Section \ref{S:Res}.  

We develop our argument for the exponential LPP, one of the simplest yet representative models in the KPZ universality class. Our choice of the setting enables us to import from \cite{Emra_Janj_Sepp_23} certain bounds for geodesic fluctuations (Lemmas \ref{LExit} and \ref{LExit2} ahead) and thereby leads to a more concise presentation. Similarly to \cite{Emra_Janj_Sepp_20, Emra_Janj_Sepp_23}, we believe that our improvement of the coupling method can be implemented in a broader range of models but leave the careful exploration of these extensions to the future. Working with the exponential LPP also means that, at a technical level, our contribution can be seen as a direct refinement of the influential article \cite{Bala_Cato_Sepp_06} reviewed in Subsection \ref{Ss:CoupKPZ}. 

We continue with a summary of our specific results leaving a more detailed account to Section \ref{S:Res}. Concerning the bulk LPP, Theorem \ref{T:Blk} shows that the upper bound in \eqref{E:CMBd} extends optimally and uniformly to all powers and all directions away from the axes. The new aspect of this result lies in its proof, which applies the coupling method to rederive a known left-tail bound with exponent $3/2$, recorded ahead as Theorem \ref{T:BlkLT}. Before the present work, the coupling approach could not produce exponentially decaying left-tail bounds; see Subsections \ref{Ss:CoupKPZ} and \ref{Ss:Coup20+} for a review of prior literature. Thanks to the generic bootstrapping arguments of \cite{Gang_Hegd_23}, the exponent $3/2$ in Theorem \ref{T:BlkLT} readily boosts to its optimal value $3$ as observed in Proposition \ref{P:BlkLT-HG}. For the KPZ-class models in general, bounding the left tail optimally is typically a difficult task. Proposition \ref{P:BlkLT-HG} specifically was obtained before through a suitable tridiagonalization of LUE \cite{Ledo_Ride_10}. The two-stage argument here that combines the coupling method with the results of \cite{Gang_Hegd_23} exemplifies a new simple approach to the left-tail bounds in LPP. Interestingly, this approach can also be implemented for directed polymers where the connection to random matrices is unavailable. Indeed, independently but several months after our initial preprint, the article \cite{Land_Soso_24} utilized the same strategy to establish optimal-order tail bounds for the O'Connell--Yor polymer. 

Our remaining results are optimal-order fluctuation bounds for the LPP with boundary introduced in \eqref{E:LPPwBd-2}. Strictly speaking, these bounds are all new although the main novelty is still in the proof technique. As indicated below and explained further in Section \ref{S:Res}, some of our bounds have appeared before in special cases or in slightly weaker forms,  
or can be extracted from bounds obtained by other means. As argued previously, a strength of our derivation via the coupling method is that the reliance on the integrable aspects of the exponential LPP is limited to the stationarity property \eqref{EIncSt}. Before mentioning our bounds, we briefly return to definition \eqref{E:LPPwBd-2} and first note that the restriction $w, z \in (0, 1)$ there can be relaxed to $w > 0$ and $z < 1$. Since these parameters are only relevant to the axis weights, one can unambiguously abbreviate by   
\begin{align}
\label{E:LppwBd-3}
\begin{split}
\G_{\bfu, \bfv}^w = \G_{\bfu, \bfv}^{w, z} \quad \text{ and } \quad \omega_{\bfu}^w = \omega_{\bfu}^{w, z} \quad \text{ when } \bfu \cdot (1, 0) > 0, \\ 
\G_{\bfu, \bfv}^z = \G_{\bfu, \bfv}^{w, z} \quad \text{ and } \quad \omega_{\bfu}^z = \omega_{\bfu}^{w, z} \quad \text{ when } \bfu \cdot (0, 1) > 0, 
\end{split}
\end{align}
which is also 
consistent with \eqref{E:wBd} and \eqref{E:LPPwBd}. We also distinguish the two special cases $\bfu = (1, 0)$ and $\bfu = (0, 1)$ of the initial vertex with the notation    
\begin{align}
\label{E:Lpp1Bd}
\G_{\bfv}^{w, \hor} = \G_{(1, 0), \bfv}^w \quad \text{ and } \quad \G_{\bfv}^{z, \ver} = \G_{(0, 1), \bfv}^z. 
\end{align}
By definition \eqref{E:Lpp2Bd} and because $\rw^{w, z}_{(0, 0)} = 0$, 
\begin{align}
\label{Emax}
\G_{\bfv}^{w, z} = \max \{\G_{\bfv}^{w, \hor}, \G_{\bfv}^{z, \ver}\} \quad \text{ for } \bfv \in \bbZ_{\ge 0}^2 \smallsetminus \{(0, 0)\}. 
\end{align}
Writing $\wt{\bfx} \in \bbR^2$ for the vector obtained by interchanging the coordinates of a given vector $\bfx \in \bbR^2$, one can also see from definition \eqref{E:LPPwBd-2} that 
\begin{align}
\label{ELPPTrans}
\G^{z, \ver}_{\bfv} \stackrel{\text{dist.}}{=} \G^{1-z, \hor}_{\wt{\bfv}} \quad \text{ for } \bfv \in \bbZ_{>0}^2. 
\end{align}
On account of \eqref{Emax} and \eqref{ELPPTrans}, for the purposes of central moment bounds, it turns out sufficient to focus on the process $\G^{w, \hor}$ with horizontal boundary. This model also recovers the bulk LPP in two ways as follows. 
\begin{align}
\label{E:1Bd2Blk}
\G_\bfv \stackrel{\rm{dist.}}{=} \G^{1, \hor}_{\bfv-(1, 0)} \quad \text{ and } \quad \G_\bfv \stackrel{\rm{dist.}}{=} \lim_{w \to \infty} \G^{w, \hor}_\bfv \quad \text{ for } \bfv \in \bbZ_{>0}^2. 
\end{align}
However, our argument completes the treatment of the bulk model first.  Hence, the separation of the two settings throughout despite one being a special case.

Among our results for the LPP with boundary, Theorems \ref{TIncRTUB} and \ref{TIncRTLB} provide matching-order upper and lower bounds, respectively, for the right tail of the last-passage increment $\G_{\bfv}^{w, \hor}-\G_{\bfv}$. Theorem \ref{TIncRTUB} in particular plays a key role in our proof of the left-tail bound in Theorem \ref{T:BlkLT}. Next recall from Subsection that \ref{Ss:CoupKPZ} the characteristic direction $\xi^w$ corresponding to the boundary parameter $w \in (0, 1)$ is the vector $(w^2, (1-w)^2)$ normalized by the $L^1$ norm. It is known that the limit fluctuations of the $\G_{\bfv}^{w, \hor}$-process are Gaussian below the direction $\xi^w$, and are of KPZ type along and above the direction $\xi^w$. This phase transition is a special case of the Baik--Ben Arous--P\'{e}ch\'{e} (BBP) transition established in \cite{Baik_BenA_Pech_05}. Our last group of results are optimal-order central moment bounds for the $\G_{\bfv}^{w, \hor}$-process. These bounds capture the transition in the fluctuation behavior. Theorems \ref{TBdUB} and \ref{TBdLB} record the upper and lower bounds, respectively for the KPZ regime. The corresponding results for the Gaussian regime are stated in Theorems \ref{TGausUB} and \ref{TGausLB}, respectively. 

Within the coupling approach employed here, the LPP with boundary enters naturally as an auxiliary gadget. However, this model is also of independent significance and has received considerable attention on its own. Similarly to the link between the bulk LPP and the TASEP with the step initial condition recalled in Subsection \ref{Ss:ExpLPP}, one can interpret $\G^{w, z}_{(i, j)}$ as the time of $i$th jump of particle $j$ in the TASEP with the following two-sided Bernoulli initial condition \cite{Bala_Cato_Sepp_06, BenA_Corw_11}. Particles are now labeled with $\bbZ$ from right to left. At time zero, particle $0$ is at site $1$ while the origin is empty. Also, independently, each site of $\bbZ_{>1}$ is empty with probability $w \in (0, 1]$ and each site of $\bbZ_{<0}$ is empty with probability $z \in [0, 1)$. The LPP with one-sided boundary is also connected to the rank one perturbation of the LUE defined as follows. With $w > 0$ again, consider the $m \times (n+1)$ random matrix $\mrX_{m, n}^w$ obtained from the matrix $\mrX_{m, n}$ in Subsection \ref{Ss:ExpLPP} by adding a $0$th column of independent complex Gaussian entries with mean zero and variance $1/w$. Then the $n \times n$ random Hermitian matrix  
\begin{align}
\label{E:LUE_1RP}
\mrL_{m, n}^{w} = (\mrX_{m, n}^w)^*\mrX_{m, n}^w = \mrL_{m, n} + [\overline{\mrX_{m, n}^w}(0, i)\mrX_{m, n}^w(0, i)]_{i, j \in [n]}
\end{align}
is a sum of the $n \times n$ LUE and a matrix of rank one. The largest eigenvalue of $\mrL_{m, n}^w$ has the same distribution as the last-passage time $\G_{(m, n)}^{w, \hor}$ \cite{Baik_BenA_Pech_05}. The limit fluctuations of the $\G^{w, z}$-process exhibit a phase transition (known as the Baik--Rains transition) depending on the limit direction and the strength of the boundary parameters. This was conjectured by Pr\"{a}hofer and Spohn \cite{Prah_Spoh_02b} based in part on the prior work \cite{Baik_Rain_00} of Baik and Rains where the analogous transitions for the geometric and Poisson LPP were established along the diagonal direction. The Pr\"{a}hofer--Spohn conjecture was fully settled in \cite{BenA_Corw_11} relying crucially on the partial progress in \cite{Baik_BenA_Pech_05, Ferr_Spoh_06}.

\subsection*{Outline}
The remainder of this work is structured as follows.  Section \ref{S:Res} states the main results.  The proofs are presented in Sections \ref{S:PfLppInc},  \ref{S:PfBlk}, \ref{SPfBdUB} and \ref{SPfBdLB}.  Appendix \ref{AppA} collects and develops some auxiliary results.  

\subsection*{Notation and conventions}

$\bbR$ and $\bbZ$ stand for the sets of reals and integers, respectively. Let $A_{\square x} = \{a \in A: a \square x\}$ for each subset $A \subset \bbR$ and relation $\square \in \{\ge, >, \le, <\}$. For example, $\bbR_{\ge 0}$ denotes the set of nonnegative reals. 

The ceiling, floor, positive part and negative part functions are given respectively by $\lc x \rc = \min \bbZ_{\ge x}$, $\lf x \rf = \max \bbZ_{\le x}$, $x_+ = \max \{x, 0\}$ and $x_- = (-x)_+$ for $x \in \bbR$. $[n] = \{1, 2, \dotsc, n\}$ for $n \in \bbZ_{>0}$, and $[0] = \emptyset$ where $\emptyset$ denotes the empty set. As usual, $\min \emptyset = \inf \emptyset = \infty$ and $\max \emptyset = \sup \emptyset = -\infty$. 

The points in $\bbR^2$ are denoted with boldface letters e.g.\ $\bfx$. The dot product on $\bbR^2$ is given by $\bfx \cdot \bfy = x_1 y_1 + x_2 y_2$ for $\bfx = (x_1, x_2) \in \bbR^2$ and $\bfy = (y_1, y_2) \in \bbR^2$. $\norm{\cdot}$ refers to the $L^1$ norm on $\bbR^2$: $\norm{\bfx} = |x_1| + |x_2|$. The coordinatewise partial order $\le$ on $\bbR^2$ is defined as follows: $\bfx \le \bfy$ if $x_1 \le y_1$ and $x_2 \le y_2$. 

For an arbitrary set $S$ and any subset $E \subset S$, the indicator function $\one_E: S \to \{0, 1\}$ (also denoted $\one E$) evaluates to $1$ on $E$, and $0$ on the complement $S \smallsetminus E$. More generally, for any function $f: E \to \bbR \cup \{\infty, -\infty\}$, the product $\one_E f: S \to \bbR$ (also denoted $\one E \cdot f$) is defined as $f$ on $E$, and as $0$ on $S \smallsetminus E$. 

For $\lambda > 0$, $\Exp(\lambda)$ refers to the rate $\lambda$ exponential distribution, which has the density $t \mapsto \one\{t \ge 0\} \cdot \lambda e^{-\lambda t}$. The MGF of $\Exp(\lambda)$ is the function $t \mapsto \one\{t < \lambda\} \cdot \lambda(\lambda-t)^{-1} + \one\{t \ge \lambda\} \cdot \infty$ for $t \in \bbR$. The notation $X \sim \Exp(\lambda)$ means that the distribution of $X$ is $\Exp(\lambda)$. If $X \sim \Exp(\lambda)$ then $c^{-1}X \sim \Exp(c\lambda)$ for any $c > 0$. 

\subsection*{Acknowledgement}

E.\ Emrah is grateful to Christopher Janjigian, Timo Sepp\"{a}l\"{a}inen and Xiao Shen for helpful discussions. The authors would like to thank Shirshendu Ganguly and Milind Hegde for mentioning their related work \cite{Gang_Hegd_23}, which led to Proposition \ref{P:BlkLT-HG}. The authors would also like to thank the anonymous referees for their comments, particularly for pointing out the proof of \cite[Lemma 4.9]{Basu_Gang_Zhan_21} in connection with the right-tail lower bound in \eqref{EBlkRTLB}.

\section{Central moment bounds for exponential LPP}
\label{S:Res}

We now begin to present our main results in precise form. As mentioned in Subsection \ref{Ss:Cont} and further explained below, some of these results are either in the literature or can be deduced as corollaries of known bounds. In these cases, our contribution is to offer new proofs via the coupling method. For comparison, we also briefly mention the techniques behind the earlier proofs.

\subsection{Bounds for the bulk LPP}
\label{Ss:BlkBd}

The first set of results concern fluctuation bounds for the bulk LPP process given by \eqref{E:BlkLPP1}. Our first theorem below provides a central moment upper bound that extends \eqref{E:CMBd} to all powers $p \ge 1$. The centering in our bound is in terms of Rost's shape function given by    
\begin{align}
\shp_{\bfx} = (\sqrt{x}+\sqrt{y})^2 \quad \text{ for } \bfx = (x, y) \in \bbR_{>0}^2.   \label{EShp}
\end{align}
This function describes the \emph{limit shape} of the cluster $\sG_t$ in \eqref{E:CGM}, that is, the a.s.\ limit (in the Hausdorff metric) of the rescaled cluster $t^{-1}\sG_t$ as $t \to \infty$ \cite{Rost_81}. The same function also appears as the centering in \eqref{E:CMBd} because $\shp_{(z^2k, (1-z)^2k)} = k$. The bound in our result holds with the usual centering by the mean $\E[\G_\bfv]$ as well 
but we preferred the explicit centering by $\shp_\bfv$. Compared to \eqref{E:CMBd}, one difference in our formulation is that the moment bound is stated uniformly over the cone 
\begin{align}
S_\delta = \{(x, y) \in \bbR^2_{>0}: x \ge \delta y \text{ and } y \ge \delta x\}\label{ECone}
\end{align}
for any fixed $\delta > 0$ rather than along a given sequence. Our bound is also uniform in the power $p$, which in particular can grow with the vertex $\bfv$. 
\begin{thm}
\label{T:Blk}
Let $\delta > 0$. There exists a constant $C_0 = C_0(\delta) > 0$ such that 
\begin{align*}
\E[|\G_{\bfv}-\shp_{\bfv}|^p] \le C_0^p p^{2p/3} (\max \{p, \norm {\bfv}\})^{p/3}
\end{align*}
for $\bfv \in S_\delta \cap \bbZ_{>0}^2$ and $p \ge 1$. 
\end{thm}

Of course, one can readily switch between central moment and tail bounds in general. 
We presented Theorem \ref{T:Blk} as a central moment bound following the predecessor articles \cite{Noac_Soso_22, Noac_Soso_22b}, which provided the initial inspiration for the current work. See \eqref{Eq-NS} for a translation of their polymer results to the exponential LPP.
Combining Theorem \ref{T:Blk} with Markov's inequality and optimizing over $p$, one can produce the following tail bound: For some constants $c_0 = c_0(\delta) > 0$ and $s_0 = s_0(\delta) > 0$, 
\begin{align}
\label{E:BlkTaBd}
\bfP\{|\G_{\bfv}-\shp_\bfv| \ge s \norm{\bfv}^{1/3}\} \le \exp\{-c_0 \min\{s^{3/2}, s \norm{\bfv}^{1/3}\}\}
\end{align}
for $\bfv \in S_\delta \cap \bbZ^2_{>0}$ and $s \ge s_0$. 
Conversely, bound \eqref{E:BlkTaBd} implies Theorem \ref{T:Blk} through an integration, which is the direction our proof takes. 

Deriving \eqref{E:BlkTaBd} naturally breaks into bounding from above the right (upper) and left (lower) tail probabilities. As previously indicated, a coupling approach to sharp right-tail bounds was already known since the article \cite{Sepp_98}. The approach there relies on the superadditivity of the last-passage times to bound the tail probabilities by the right-tail rate function for large deviations, which can be computed via the coupling method. The preprint \cite{Emra_Janj_Sepp_20} recently discovered a more direct argument based on identity \eqref{E:Rains} and the exponential Markov inequality. The following bound and its proof is essentially from \cite{Emra_Janj_Sepp_20}, and is included here for completeness. 

\begin{prop}
\label{PBlkRTUB}
Let $\delta > 0$. There exist constants $c_0 = c_0(\delta) > 0$ and $N_0 = N_0(\delta) > 0$ such that 
\begin{align*}
\P\{\G_{\bfv} \ge \shp_{\bfv} + s \norm{\bfv}^{1/3}\} \le \exp\{-c_0 \min \{s^{3/2}, s \norm{\bfv}^{1/3}\}\}
\end{align*}
for $\bfv \in S_\delta \cap \bbZ_{\ge N_0}^2$ and $s \ge 0$. 
\end{prop} 

As an illustration of multiple approaches to the tail bounds in the KPZ class, we mention that Proposition \ref{PBlkRTUB} in some form has also been derived through different techniques; see \cite{Ledo_07} for a nice survey. For example, one can begin with examining the large deviations of the LUE spectrum via its Coulomb gas interpretation and then derive the right-tail rate function for the largest eigenvalue \cite[Remark 2.3]{Joha_00}. Another argument relies on the spectrum of LUE forming a determinantal point process with an explicit correlation kernel that can be represented as a contour integral. Then a suitable decay of the kernel, which can be established through steepest-descent analysis, leads to Proposition \ref{PBlkRTUB}. The proof of \cite[Lemma A.1]{Corw_Liu_Wang_16} illustrates the preceding argument for the geometric LPP. A third approach develops moment recurrences for the mean spectral measure of LUE and then extracts right-tail upper bound for the largest eigenvalue as a consequence \cite[Proposition 5.4]{Ledo_04}. Yet another method utilizes the Rayleigh representation of the largest eigenvalue of a tridiagonal matrix with the same spectrum as LUE \cite[Theorem 2]{Ledo_Ride_10}.

With Proposition \ref{PBlkRTUB} in place, to reach \eqref{E:BlkTaBd}, it remains to derive an analogous upper bound for the left tail. The main technical contribution with the present work is the attainment of such a bound via the coupling method for the first time. Prior to our proof, the best left-tail bounds achievable through coupling arguments were subpolynomial, and covered the increment-stationary versions of the O'Connell--Yor polymer \cite{Noac_Soso_22} and the four basic integrable lattice polymers \cite{Noac_Soso_22b}. These earlier bounds can be deduced from the central moment bounds in \cite{Noac_Soso_22, Noac_Soso_22b} (analogous to \eqref{Eq-NS}) and Markov's inequality. Before these recent results, the coupling method could only produce up to quadratic bounds on the left tail, which directly reflects the strength of the corresponding central moment bounds such as \eqref{E:VarBd} and \eqref{E:CMBd} for the exponential LPP. 

The next theorem states our left-tail upper bound. In contrast with Proposition \ref{PBlkRTUB}, there is no transition in the exponent now because $\G_{\bfv} \ge 0$ and, consequently, the probability below is zero when $s \ge \shp_\bfv \cdot \norm{\bfv}^{-1/3}$ (where the right-hand side is at most $2\norm{\bfv}^{2/3}$ by \eqref{EShp}). 
\begin{thm}
\label{T:BlkLT}
Let $\delta > 0$. There exist constants $c_0 = c_0(\delta) > 0$, $N_0 = N_0(\delta) > 0$ and $s_0 = s_0(\delta) > 0$ such that 
\begin{align*}
\P\{\G_{\bfv} \le \shp_{\bfv} - s\norm{\bfv}^{1/3}\} \le \exp\{-c_0s^{3/2}\}
\end{align*}
for $\bfv \in S_\delta \cap \bbZ_{\ge N_0}^2$ and $s \ge s_0$. 
\end{thm}

Alternative proofs of Theorem \ref{T:BlkLT} are available in the literature. For example, the article \cite[Proposition 4.3]{Ferr_Nejj_15} records a version of Theorem \ref{T:BlkLT} 
as a corollary of a trace lower bound \cite[Proposition 3]{Baik_Ferr_Pech_14} for the LUE kernel. Establishing the trace bound involves steepest-descent analysis in \cite{Baik_Ferr_Pech_14}. This line of argument is known as Widom's trick and was first used in the context of uniform permutations and the Bernoulli LPP with strict-weak paths \cite{Wido_02}. 
As observed by Ledoux \cite[Section 5.3]{Ledo_07}, the aforementioned trace bound can also be obtained more quickly from the LUE moment recurrences. Further proofs of Theorem \ref{T:BlkLT} are mentioned after Proposition \ref{P:BlkLT-HG} below. 

In a recent work \cite{Gang_Hegd_23}, Ganguly and Hegde derived optimal-order tail bounds (up to a logarithmic weakening in the case of right-tail upper bound) for LPP with general nonnegative i.i.d.\ weights, starting from weaker a priori tail bounds and a curvature bound on the limit shape. When $\bfv$ is diagonally directed, \cite[Theorem 3]{Gang_Hegd_23} improves the bound in Theorem \ref{T:BlkLT} to the optimal-order as follows. (Recall from \eqref{EShp} that $\shp_{(1, 1)} = 4$).   
\begin{prop}
\label{P:BlkLT-HG}
There exist positive absolute constants $c_0$, $N_0$ and $s_0$ such that 
\begin{align*}
\P\{\G_{(n, n)} \le 4n-sn^{1/3}\} \le \exp\{-c_0 s^3\}
\end{align*}
for $n \in \bbZ_{\ge N_0}$ and $s \ge s_0$. 
\end{prop}
The restriction to the diagonal above is because \cite[Theorem 3]{Gang_Hegd_23} covers only that direction. However, the authors explain in \cite[Remark 1.8, Section 4.3]{Gang_Hegd_23} that their result can be extended to the off-diagonal directions. Combining this extension with Theorem \ref{T:BlkLT}, one can likely generalize Proposition \ref{P:BlkLT-HG} to all directions. As mentioned in Subsection \ref{Ss:Cont}, our proof of Proposition \ref{P:BlkLT-HG} is a new argument for optimal-order left-tail deviations. The same approach has also been developed subsequently and independently for the O'Connell--Yor polymer \cite{Land_Soso_24}. (In fact, the counterpart of Proposition \ref{P:BlkLT-HG} is proved in \cite[(1.6)]{Land_Soso_24} for all directions). The previous proof of Proposition \ref{P:BlkLT-HG} (for all directions) can be found in \cite[Theorem 2]{Ledo_Ride_10} where the approach rests on tridiagonal matrix models as mentioned before. It is also widely expected that the Riemann--Hilbert analysis carried out in \cite{Baik_Deif_McLa_Mill_Zhou_01} for the geometric LPP can be adapted to produce sharp estimates for the left-tail probability in Proposition \ref{P:BlkLT-HG}. 

Although we do not prove it here, the upper bound in Theorem \ref{T:Blk} is in fact optimal in order in the sense that the following matching-order lower bound holds: Given $\delta > 0$ and $T > 0$, there exist positive constants $N_0 = N_0(\delta, T)$ and $c_0 = c_0(\delta, T)$ such that 
\begin{align}
\label{EBlkLB}
\E[|\G_{\bfv}-\shp_{\bfv}|^p] \ge c_0^p p^{2p/3} \norm{\bfv}^{p/3}
\end{align}
for $\bfv \in S_\delta \cap \bbZ_{\ge N_0}^2$ and $1 \le p \le T\norm{\bfv}$. In fact, the stronger inequality  
\begin{align}
\E[(\G_{\bfv}-\shp_{\bfv})_+^p] \ge c_0^p p^{2p/3} \norm{\bfv}^{p/3} \label{EBlkLB2}
\end{align}
holds for the same range of $\bfv$ and $p$. Through a simpler version of the argument proving Proposition \ref{PBdMomLB} ahead, one can derive \eqref{EBlkLB2} from the next tail bound: Given $\delta > 0$, there exist positive constants $N_0 = N_0(\delta)$, $C_0 = C_0(\delta)$, $\epsilon_0 = \epsilon_0(\delta)$ and $s_0 = s_0(\delta)$ such that 
\begin{align}
\label{EBlkRTLB}
\P\{\G_{\bfv} \ge \shp_{\bfv} + s\norm{\bfv}^{1/3}\} \ge \exp\{-C_0 s^{3/2}\}
\end{align} 
for $\bfv \in S_\delta \cap \bbZ_{\ge N_0}^2$ and $s \in [s_0, \epsilon_0 \norm{\bfv}^{2/3}]$. It is possible to derive \eqref{EBlkRTLB} via the coupling method but this seems to require more careful tracking of the constants in various bounds than is pursued here. A coupling proof of \eqref{EBlkRTLB} will be reported in the forthcoming update to the preprint \cite{Emra_Janj_Sepp_20}.

To our knowledge, at the time of our first preprint, the bound in \eqref{EBlkRTLB} had not been fully proved although it was widely considered accessible to both the Riemann--Hilbert approach of \cite{Baik_Deif_McLa_Mill_Zhou_01} and the tridiagonalization approach of \cite{Ledo_Ride_10}. The bound itself was recorded, for example, in \cite[Proposition 3.3]{Bhat_20} as an input. In the case of diagonally directed $\bfv$, \eqref{EBlkRTLB} is immediate from \cite[Lemma 4.9]{Basu_Gang_Zhan_21}. The elegant proof of this lemma relies on the superadditivity of the $\G$-process and its convergence to a limit distribution (the Tracy-Widom GUE \cite{Joha_00}) with the support $\bbR$. In fact, it suffices that $\bbR_{>0}$ has positive mass; see \cite[Remark 1.3, Theorem 2]{Gang_Hegd_23} where a similar argument appears. Although the proofs in \cite{Basu_Gang_Zhan_21, Gang_Hegd_23} can be adapted to any direction, \eqref{EBlkRTLB} does not seem immediate in this way because one would need directional uniformity in the distributional convergence result from \cite{Joha_00}. Very recently, the preprint \cite{Basl_Basu_Bhat_Kris_24} established \eqref{EBlkRTLB} using tridiagonal matrix models.

A neat coupling argument due to Moreno Flores, Sepp\"{a}l\"{a}inen and Valk\'{o} \cite{More_Sepp_Valk_14}, which was originally presented in the context of the O'Connell--Yor polymer, provides the following lower bound that matches the order in Theorem \ref{T:Blk} in the case of bounded $p$. Analogous lower bounds were established in \cite[Section 4.6]{Kang_19_PhD} for the inhomogeneous log-gamma polymer also adapting \cite{More_Sepp_Valk_14}.   
\begin{prop}
\label{P:BlkCMLB}
Let $\delta > 0$. There exist positive constants $c_0 = c_0(\delta)$ and $N_0 = N_0(\delta)$ such that 
\begin{align*}
\E[|\G_\bfv-\gamma_\bfv|^p] \ge c_0^p \norm{\bfv}^{p/3}
\end{align*}
for $\bfv \in S_\delta \cap \bbZ_{\ge N_0}^2$ and $p \ge 1$. 
\end{prop}

\subsection{Bounds for LPP increments}
\label{Ss:IncTaBd}

We next prepare to state our tail bounds for certain increments of the LPP with one-sided boundary given by \eqref{E:Lpp1Bd}. Define  
\begin{align}
\ShpMin_{\bfx} = \frac{\sqrt{x}}{\sqrt{x}+\sqrt{y}} \quad \text{ for } \bfx = (x, y) \in \bbR_{>0}^2, \label{EShpMin}
\end{align}
which can also be seen as an increasing continuous function of the ratio $x/y$ from $\bbR_{>0}$ onto $(0, 1)$. Since $\xi^{\ShpMin_{\bfx}} = \bfx\norm{\bfx}^{-1}$ for $\bfx \in \bbR_{>0}^2$ and $\ShpMin_{\xi^z} = z$ for $z \in (0, 1)$,  the characteristic direction map $z \mapsto \xi^z$ is the inverse of the function $\ShpMin$ restricted to the vectors of unit $L^1$ norm. It will be convenient to express our subsequent bounds in terms of $\zeta$ although one can also rewrite them more conventionally in terms of $\xi$ through the preceding correspondence. 

The following result allows for a useful comparison between the bulk LPP and the models with one-sided boundary defined in \eqref{E:Lpp1Bd}. This will enable us to derive the left-tail bound in Theorem \ref{T:BlkLT} from the left-tail bounds for the LPP with boundary. The latter bounds are recorded in Section \ref{S:PfLppInc} as relatively straightforward consequences of identity \eqref{E:Rains}. 
\begin{thm}
\label{TIncRTUB}
Let $\delta > 0$. There exist positive constants $c_0 = c_0(\delta), \epsilon_0 = \epsilon_0(\delta)$, $N_0 = N_0(\delta)$ and $s_0 = s_0(\delta)$ such that the following statements hold for $\bfv \in S_\delta \cap \bbZ_{\ge N_0}^2$, $s \ge s_0$, $w > 0$ and $z < 1$. 
\begin{enumerate}[\normalfont (a)]
\item If $w \ge \ShpMin_{\bfv}-\epsilon_0 \min \{s^{1/2} \norm{\bfv}^{-1/3}, 1\}$ then 
\begin{align*}
\P\{\G^{w, \hor}_{\bfv} - \G_{\bfv} \ge s\norm{\bfv}^{1/3}\} \le \exp\{-c_0\min \{s^{3/2}, s\norm{\bfv}^{1/3}\}\}. 
\end{align*}
\item If $z \le \ShpMin_{\bfv} + \epsilon_0 \min \{s^{1/2} \norm{\bfv}^{-1/3}, 1\}$ then 
\begin{align*}
\P\{\G^{z, \ver}_{\bfv} - \G_{\bfv} \ge s\norm{\bfv}^{1/3}\} \le \exp\{-c_0\min \{s^{3/2}, s\norm{\bfv}^{1/3}\}\}. 
\end{align*}
\end{enumerate}
\end{thm}
Parts (a) and (b) of Theorem \ref{TIncRTUB} are, in fact, equivalent. The same remark applies to all bounds for the $\G^{w, \hor}$ and $\G^{z, \ver}$-processes throughout the article. These equivalences can be readily checked using \eqref{ELPPTrans} and the identities 
\begin{align}
\label{ETrans}
\shp_{\wt{\bfx}} = \shp_{\bfx} \quad \text{ and } \quad \ShpMin_{\wt{\bfx}} = 1 - \ShpMin_{\bfx} \quad \text{ for } \bfx = (x, y) \in \bbR_{>0}^2  
\end{align}
where $\wt{\bfx} = (y, x)$. Hence, we will only focus on parts (a) in the proofs without further mention. 

Our proof of Theorem \ref{TIncRTUB} is a strengthening of the argument in \cite[Lemma 7.1]{Bala_Cato_Sepp_06}, which also applies the coupling method and obtains a decay of the form $s^{3/2-\epsilon}$ for arbitrarily small $\epsilon > 0$. The key aspect of our proof still involves coupling with a suitable pair of increment-stationary LPP and controlling the difference of their increments and their geodesic fluctuations. Such an argument was first devised in \cite{Cato_Groe_06} for the Poisson LPP. Our main innovation here is to utilize the geodesic fluctuation bounds recently obtained in \cite{Emra_Janj_Sepp_23} via the coupling method. Compared to \cite{Bala_Cato_Sepp_06}, our proof is also more direct in the sense that the notions of reversed LPP and competition interface are not used. Another difference is that Theorem \ref{TIncRTUB} is uniform for a broader range of directions while \cite[Lemma 7.1]{Bala_Cato_Sepp_06} covers only the characteristic direction (essentially the case $w = \zeta_\bfv$ in part (a)). For these reasons, our proof is not a line-by-line adaptation. 

Using the containment  
\begin{align}
\label{E:100}
\{\G^{w, \hor}_{\bfv} - \G_{\bfv} \ge s\norm{\bfv}^{1/3}\} \subset \bigg\{\G_\bfv \le \shp_\bfv - \frac{1}{2}s\norm{\bfv}^{1/3}\bigg\} \cup \bigg\{\G_{\bfv}^{w, \hor} \ge \shp_\bfv + \frac{1}{2}s\norm{\bfv}^{1/3}\bigg\}, 
\end{align}
one can alternatively obtain Theorem \ref{TIncRTUB}(a) from Theorem \ref{T:BlkLT} and a suitable right-tail bound for the $\G^{w, \hor}$-process. As mentioned, Theorem \ref{T:BlkLT} is already available by other means. The needed right-tail bound can also be derived, for example, from a sufficiently strong decay bound on the correlation kernel of the perturbed LUE in \eqref{E:LUE_1RP}; see the proof of \cite[Lemma 3.3]{Ferr_Occe_18}. In our argument, however, Theorem \ref{TIncRTUB} being an input to the proof of Theorem \ref{T:BlkLT} must be established first. 

By appropriately modifying the proof strategy for Theorem \ref{TIncRTUB}, we also obtain the following matching-order lower bounds in the moderate deviation regime (for small $s\norm{\bfv}^{-2/3}$). 
\begin{thm}
\label{TIncRTLB}
Let $\delta > 0$ and $K \ge 0$. There exist positive constants $C_0 = C_0(\delta, K)$, $\epsilon_0(\delta, K)$, $N_0 = N_0(\delta, K)$ and $s_0 = s_0(\delta)$ such that the following statements hold for $\bfv \in S_\delta \cap \bbZ_{\ge N_0}^2$, $s \in [s_0, \epsilon_0 \norm{\bfv}^{2/3}]$, $w > 0$ and $z < 1$.
\begin{enumerate}[\normalfont (a)]
\item If $w \le \ShpMin_{\bfv} + Ks^{1/2}\norm{\bfv}^{-1/3}$ then 
$
\P\{\G^{w, \hor}_{\bfv}-\G_{\bfv} \ge s\norm{\bfv}^{1/3}\} \ge \exp\{-C_0 s^{3/2}\}. 
$
\item If $z \le \ShpMin_{\bfv}-K s^{1/2}\norm{\bfv}^{-1/3}$ then 
$
\P\{\G^{z, \ver}_{\bfv}-\G_{\bfv} \ge s\norm{\bfv}^{1/3}\} \ge \exp\{-C_0 s^{3/2}\}. 
$
\end{enumerate}
\end{thm}
Unlike Theorem \ref{TIncRTUB}, Theorem \ref{TIncRTLB} does not seem a direct consequence of known bounds. It also appears to us that a (hypothetical) alternative proof of Theorem \ref{TIncRTLB} based on the LUE connection would be significantly more involved. 

\subsection{Bounds for LPP with boundary}
\label{Ss:LppBdBds}

Our last set of results concerns central moment bounds for the LPP with boundary. As previously noted, one can restrict to the case of one-sided boundary here because the two-sided boundary can then be handled using \eqref{Emax}. These corollaries as well as various (minor) extensions and variations of the following bounds can be found in Section 3 of our first preprint but are omitted from the current article. 

Theorem \ref{TBdUB} below states upper bounds in the KPZ regime where the order of the central moments are as in the bulk case. We capture this regime with precise conditions on the differences between the minimizer $\ShpMin$ defined at \eqref{EShpMin}, and the boundary parameters. As in Theorem \ref{T:Blk}, one can again alter the centerings in the theorem to the means. But since the means do not have a simple explicit form, we preferred the formulation with the shape function $\shp$. 

\begin{thm}
\label{TBdUB}
Let $\delta > 0$. There exist positive constants $C_0 = C_0(\delta)$ and $\epsilon_0 = \epsilon_0(\delta)$ such that the following statements hold for $\bfv \in S_\delta \cap \bbZ_{>0}^2$, $p \ge 1$, $w > 0$ and $z < 1$.  
\begin{enumerate}[\normalfont (a)]
\item If $w \ge \ShpMin_{\bfv}-\epsilon_0 \min \{p^{1/3}\norm{\bfv}^{-1/3}, 1\}$ then $$\E[|\G^{w, \hor}_{\bfv}-\shp_{\bfv}|^p] \le C_0^p p^{2p/3} (\max \{p, \norm{\bfv}\})^{p/3}.$$ 
\item If $z \le \ShpMin_{\bfv}+\epsilon_0 \min \{p^{1/3}\norm{\bfv}^{-1/3}, 1\}$ then $$\E[|\G^{z, \ver}_{\bfv}-\shp_{\bfv}|^p] \le C_0^p p^{2p/3} (\max \{p, \norm{\bfv}\})^{p/3}.$$
\end{enumerate}
\end{thm}

Similarly to \eqref{E:BlkTaBd}, Theorem \ref{TBdUB}(a) can be converted to the equivalent tail bound
\begin{align}
\label{E:BdTaUB}
\bfP\{|\G_{\bfv}^{w, \hor}-\shp_\bfv| \ge s \norm{\bfv}^{1/3}\} \le \exp\{-c_0 \min\{s^{3/2}, s \norm{\bfv}^{1/3}\}\}
\end{align}
for $\bfv \in S_\delta \cap \bbZ^2_{>0}$ and $s \ge s_0$ for some constants $c_0 = c_0(\delta) > 0$ and $s_0 = s_0(\delta) > 0$. Due to the inequality $\G_{\bfv} \le \G_{\bfv}^{w, \hor}$, the left-tail part of \eqref{E:BdTaUB} is immediate from left-tail bound in Theorem \ref{T:BlkLT}, which is available via other methods. The right-tail part can also be obtained in several ways. 
As noted after \eqref{E:100}, one proof would be via kernel analysis along the lines of \cite[Lemma 3.3]{Ferr_Occe_18}. One could also likely obtain the right-tail bound by coarse graining as in the proof of \cite[Theorem 2.6(1)]{Bhat_20} where the left-tail bound in Proposition \ref{P:BlkLT-HG} is taken as an input. In our argument, Theorem \ref{TBdUB} comes as a consequence of Theorem \ref{T:BlkLT} and the increment tail bounds in Theorem \ref{TIncRTUB}, which will be derived earlier.  

We next record central moment lower bounds in the KPZ regime. These bounds match the order of the upper bounds in Theorem \ref{TBdUB}. 
\begin{thm}
\label{TBdLB}
Let $\delta > 0$, $K \ge 0$ and $T > 0$. There exist positive constants $c_0 = c_0(\delta, K, T)$ and $N_0 = N_0(\delta, K, T)$ such that the following statements hold for $\bfv \in S_\delta \cap \bbZ_{\ge N_0}^2$, $p \in [1, T \norm{\bfv}]$, $w > 0$ and $z < 1$.  
\begin{enumerate}[\normalfont (a)]
\item If $w \le \ShpMin_{\bfv}+K\min \{p^{1/3}\norm{\bfv}^{-1/3}, 1\}$ then $$\E[|\G^{w, \hor}_{\bfv}-\shp_{\bfv}|^p] \ge c_0^p p^{2p/3} \norm{\bfv}^{p/3}.$$ 
\item If $z \ge \ShpMin_{\bfv}-K\min \{p^{1/3}\norm{\bfv}^{-1/3}, 1\}$ then $$\E[|\G^{z, \ver}_{\bfv}-\shp_{\bfv}|^p] \ge c_0^p p^{2p/3} \norm{\bfv}^{p/3}.$$
\end{enumerate}
\end{thm}

The conditions in (a) and (b) are technical; the conclusions of the theorem in fact hold for all $w > 0$ and $z < 1$. The stronger statement is an immediate consequence of \eqref{EBlkLB2}, which in turn comes from the right-tail lower bound \eqref{EBlkRTLB}. Because the coupling proof of this bound has not yet appeared in the literature, we derive Theorem \ref{TBdLB} via an alternative argument, which relies on the conditions in (a) and (b). The main effort goes into strengthening a right-tail lower bound from \cite{Sepp_18_CGM} as in Proposition \ref{PBdRTLB} ahead, which is required to treat powers $p$ growing with $\norm{\bfv}$. The core of our argument is still the change-of-measure idea originating from \cite{Bala_Sepp_10}.  

We now switch attention to the central moments in situations complementary to the KPZ regime treated in Theorem \ref{TBdUB}. 

It is well-known that once the distance $|z-\ShpMin_{\bfv}|$ becomes sufficiently large (namely, when the direction of $\bfv$ is sufficiently away from the characteristic direction $\xi^z$),  
$\G^{z}_{\bfv}$ begins to display Gaussian fluctuations. The reason for the Gaussian behavior is that with high probability the geodesic $\pi_{(0, 0), \bfv}^z$ visits order $|\ShpMin_{\bfv}-z| \cdot \norm{\bfv}$ many vertices on the axes. Due to the central limit theorem, the sum of the weights along these vertices fluctuate on scale $|\ShpMin_{\bfv}-z|^{1/2} \cdot \norm{\bfv}^{1/2}$, which dominates order $\norm{\bfv}^{1/3}$ fluctations in the bulk when $|z-\ShpMin_{\bfv}|$ is much larger in order than $\norm{\bfv}^{-1/3}$ as $\norm{\bfv} \to \infty$. 

For the exponential LPP, a precise result showing the convergence of the  suitably rescaled $\G^z_{\bfv}$ to the standard Gaussian distribution along off-characteristic directions can be found in \cite[Corollary 5.2]{Sepp_18_CGM}. Analogous results have also been recorded for the log-gamma polymer \cite[Corollary 2.2]{Sepp_12_corr}, the other three basic integrable lattice polymers \cite[Corollary 1.4]{Chau_Noac_18b}, and the strict-strict Bernoulli LPP \cite[Theorem 2.4]{Ciec_Geor_19}. However, to our knowledge, central moment bounds in the Gaussian regime have not previously appeared in the literature. The next result provides such upper bounds, capturing also the transition from the KPZ to the Gaussian behavior. For example in part (a), roughly speaking, the upper bound interpolates from order $p^{2p/3}\norm{\bfv}^{p/3}$ to order $p^{p/2}\norm{\bfv}^{p/2}$ as the difference $\ShpMin_{\bfv}-w$ increases from order $p^{1/3}\norm{\bfv}^{-1/3}$ to order one. 

\begin{thm}
\label{TGausUB}
Let $\delta > 0$ and $\epsilon > 0$. There exists a constant $C_0 = C_0(\delta, \epsilon) > 0$ such that the following statements hold for $\bfv \in S_\delta \cap \bbZ_{>0}^2$, $p \ge 1$, $w > 0$ and $z < 1$. 
\begin{enumerate}[\normalfont (a)]
\item If $\epsilon \le w \le \ShpMin_{\bfv}-\epsilon \min \{p^{1/3}\norm{\bfv}^{-1/3}, 1\}$ then 
\begin{align*}
\E[|\G^{w, \hor}_{\bfv}-\M^{w}_{\bfv}|^p] \le C_0^p p^{p/2}  \max \{p, (\ShpMin_{\bfv}-w)\norm{\bfv}\}^{p/2}. 
\end{align*}
\item If $\ShpMin_{\bfv}+\epsilon \min \{p^{1/3}\norm{\bfv}^{-1/3}, 1\} \le z \le 1-\epsilon$ then 
\begin{align*}
\E[|\G^{z, \ver}_{\bfv}-\M^{z}_{\bfv}|^p] \le C_0^p p^{p/2} \max \{p, (z-\ShpMin_{\bfv})\norm{\bfv}\}^{p/2}. 
\end{align*}
\end{enumerate} 
\end{thm}
Unlike the earlier results, the conclusions of Theorem \ref{TGausUB} do not hold with the centering with respect to the shape function. For instance for part (a), this is because the difference $|\M^{w}_{\bfv}-\shp_{\bfv}|^p$ is of order $(\ShpMin_{\bfv}-w)^{2p}\norm{\bfv}^{p}$ (see Lemma \ref{LMeanEst} below) which surpasses the upper bound $p^{p/2}(\ShpMin_{\bfv}-w)^{p/2}\norm{\bfv}^{p/2}$ when $\ShpMin_{\bfv}-w$ is much larger in order than $p^{1/3}\norm{\bfv}^{-1/3}$. 

Our final result provides central moment lower bounds for the $\G^{w, z}$-process in the Gaussian regime, complementing the upper bounds in Theorem \ref{TGausUB}. 

\begin{thm}
\label{TGausLB}
Let $\delta > 0$, $\epsilon > 0$ and $T > 0$. There exist positive constants $c_0 = c_0(\delta, \epsilon, T)$, $K_0 = K_0(\delta, \epsilon, T)$ and $N_0 = N_0(\delta, \epsilon, T)$ such that the following statements hold for $\bfv \in S_\delta \cap \bbZ_{\ge N_0}^2$, $p \in [1, T \norm{\bfv}]$, $w > 0$ and $z < 1$.  
\begin{enumerate}[\normalfont (a)]
\item If $\epsilon \le w \le \ShpMin_{\bfv}-\min \{K_0 p^{1/3}\norm{\bfv}^{-1/3}, \epsilon\}$ then $$\E[|\G^{w, \hor}_{\bfv}-\M^w_{\bfv}|^p] \ge c_0^p p^{p/2} (\ShpMin_{\bfv}-w)^{p/2}\norm{\bfv}^{p/2}.$$ 
\item If $\ShpMin_{\bfv}+\min \{K_0p^{1/3}\norm{\bfv}^{-1/3}, \epsilon\} \le z \le 1-\epsilon$ then $$\E[|\G^{z, \ver}_{\bfv}-\M^w_{\bfv}|^p] \ge c_0^p p^{p/2} (z-\ShpMin_{\bfv})^{p/2}\norm{\bfv}^{p/2}.$$
\end{enumerate}
\end{thm}
With $K_0 \ge \epsilon$, the conditions on the boundary parameters in Theorem \ref{TGausLB} are more restrictive than those in Theorem \ref{TGausUB}. Therefore, it appears that we do not have a lower bound complementary to Theorem \ref{TGausUB}(a) in the regime given by $\epsilon p^{1/3}\norm{\bfv}^{-1/3} < \ShpMin_{\bfv}-w \le K_0 p^{1/3}\norm{\bfv}^{-1/3}$ and $p \le \epsilon^3 K_0^{-3}\norm{\bfv}$. But in fact there is no such gap because Theorem \ref{TBdLB}(a) provides the following bound for some positive constants $c_0 = c_0(\delta, T)$ and $N_0 = N_0(\delta, T)$:  \begin{align*}
\E[|\G^{w, \hor}_{\bfv}-\M^{w}_{\bfv}|^p] \ge c_0^p p^{2p/3}\norm{\bfv}^{p/3} \ge c_0^p K_0^{-p/2}p^{p/2}(\ShpMin_{\bfv}-w)^{p/2}\norm{\bfv}^{p/2}
\end{align*}
for $\bfv \in S_\delta \cap \bbZ_{\ge N_0}^2$, $p \in [1, T\norm{\bfv}]$ and $w > 0$ subject to $\ShpMin_{\bfv}-w \le K_0 p^{1/3}\norm{\bfv}^{-1/3}$.

\section{Proofs of the tail bounds for LPP increments}
\label{S:PfLppInc}

We now begin to prove our main results, devoting this section to the tail bounds stated in Subsection \ref{Ss:IncTaBd} for the increments of LPP with one-sided boundary. 

\subsection{Proof of Theorem \ref{TIncRTUB}}

We first collect a few ingredients towards the proof of Theorem \ref{TIncRTUB}. 

Recall the $\G^{w, z}$-process given by \eqref{E:Lpp2Bd} for some boundary parameters $w > 0$ and $z < 1$. Let $\pi_{\bfv}^{w, z} = \pi_{(0, 0), \bfv}^{w, z}$ denote the corresponding a.s.\ unique geodesic from the origin to $\bfv \in \bbZ_{\ge 0}^2$. The \emph{exit points} of $\pi^{w, z}_{\bfv}$ from the horizontal and vertical axes are defined by 
\begin{align}
\label{EExitPt}
\begin{split}
\Z^{w, z, \hor}_{\bfv} &= \max \{k \in \bbZ_{\ge 0}: (k, 0) \in \pi_{\bfv}^{w, z}\} \quad \text{ and } \quad \Z^{w, z, \ver}_{\bfv} = \max \{l \in \bbZ_{\ge 0}: (0, l) \in \pi_{\bfv}^{w, z}\}, 
\end{split}
\end{align}
respectively. 
When $\bfv \neq (0, 0)$, a.s., exactly one of $\Z^{w, z, \hor}_{\bfv}$ and $\Z^{w, z, \ver}_{\bfv}$ equals zero. 

The next pair of lemmas record some tail bounds for the exit points. These bounds were obtained in \cite{Emra_Janj_Sepp_23} also via the coupling approach using stationarity \eqref{EIncSt} and its consequence \eqref{E:Rains}. As such, they can be imported here without departing the coupling framework of this article. 

\begin{lem}[Corollary 3.3 in \cite{Emra_Janj_Sepp_23}]
\label{LExit}
Fix $\delta > 0$. There exist positive constants $c_0 = c_0(\delta)$, $\epsilon_0 = \epsilon_0(\delta)$, $N_0 = N_0(\delta)$ such that the following statements hold whenever $w > 0$, $z < 1$, $\bfv \in S_\delta \cap \bbZ_{\ge N_0}^2$ and $s \ge \norm{\bfv}^{-2/3}$. 
\begin{enumerate}[\normalfont (a)]
\item If $\min \{w, z\} \ge \ShpMin_{\bfv}-\epsilon_0 s \norm{\bfv}^{-1/3}$ then $\P\{\Z^{w, z, \hor} > s\norm{\bfv}^{2/3}\} \le \exp\{-c_0 s^3\}$. 
\item If $\max\{w, z\} \le \ShpMin_{\bfv}+\epsilon_0 s\norm{\bfv}^{-1/3}$ then $\P\{\Z^{w, z, \ver}_{\bfv} > s\norm{\bfv}^{2/3}\} \le \exp\{-c_0 s^3\}$. 
\end{enumerate}
\end{lem}

\begin{lem}[Proposition 3.4 in \cite{Emra_Janj_Sepp_23}]
\label{LExit2}
Fix $\delta > 0$. There exists a constant $c_0 = c_0(\delta) > 0$ such that the following statements hold for all $\bfv \in S_\delta \cap \bbZ_{>0}^2$ and $z \in (0, 1)$. 
\begin{enumerate}[\normalfont (a)]
\item If $z > \ShpMin_\bfv$ then $\bfP\{\Z^{z, \hor}_\bfv > 0\} \le \exp\{-c_0 \norm{\bfv}(z-\ShpMin_{\bfv})^3\}$. 
\item If $z < \ShpMin_\bfv$ then $\bfP\{\Z^{z, \ver}_\bfv > 0\} \le \exp\{-c_0 \norm{\bfv} (\ShpMin_{\bfv}-z)^3\}$. 
\end{enumerate}
\end{lem} 

Another auxiliary device in our proof of Proposition \ref{TIncRTUB} is the increment-stationary LPP with \emph{northeast boundary}, which can be introduced as follows. Let $\bfu = (m, n) \in \bbZ_{\ge 0}^2$ and $z \in (0, 1)$. Consider new independent weights $\wt{\rw}^{\bfu, z} = \{\wt{\rw}^{\bfu, z}_{\bfv}: (1, 1) \le \bfv \le \bfu + (1, 1)\}$ on the integer grid $[(1, 1), \bfu + (1, 1)] = [m+1] \times [n+1]$ such that 
\begin{align}
\label{EwNE}
\begin{split}
\wt{\rw}^{\bfu, z}_{\bfu+(1, 1)} &= 0, \qquad \wt{\rw}^{\bfu, z}_\bfv = \rw_{\bfv} \quad \text{ for } (1, 1) \le \bfv \le \bfu, \quad \text{ and }\\
\wt{\rw}^{\bfu, z}_{(i, n+1)}  &\sim \Exp(z) \quad \text{ and } \quad \wt{\rw}^{\bfu, z}_{(m+1, j)} \sim \Exp(1-z) \quad \text{ for } i \in [m] \text{ and } j \in [n]. 
\end{split}
\end{align}
Then define the associated last-passage times by 
\begin{align}
\label{ELppNE}
\wt{\G}^{\bfu, z}_{\bfv, \bfv'} =  \max_{\pi \in \Pi_{\bfv, \bfv'}} \left\{\sum_{\bfp \in \pi} \wt{\rw}^{\bfu, z}_{\bfp}\right\} \quad \text{ for } \bfv, \bfv' \in [(1, 1), \bfu+(1, 1)]. 
\end{align}
In the case $\bfv \le \bfv'$, write $\wt{\pi}^{\bfu, z}_{\bfv, \bfv'} \in \Pi_{\bfv, \bfv'}$ for the a.s.\ unique geodesic in \eqref{ELppNE}. 

Comparing \eqref{EwNE} with \eqref{E:Coup} and \eqref{E:wBd}, one can check the distributional equality 
\begin{align}
(\wt{\rw}^{\bfu, z}_{\bfv+(1, 1)}: (0, 0) \le \bfv  \le \bfu) \stackrel{\text{dist.}}{=} (\rw_{\bfu-\bfv}^z: (0, 0) \le \bfv \le \bfu). \label{EDistId}
\end{align}
Because definitions \eqref{E:LPPwBd} and \eqref{ELppNE} apply the same map to the given weights, it follows from \eqref{EDistId} that 
\begin{align}
\begin{split}
(\wt{\G}^{\bfu, z}_{\bfv+(1, 1), \bfv' + (1, 1)}: (0, 0) \le \bfv, \bfv' \le \bfu) \stackrel{\text{dist.}}{=} (\G^z_{\bfu-\bfv', \bfu-\bfv}: (0, 0) \le \bfv, \bfv' \le \bfu). 
\end{split}
\label{ELppDistId}
\end{align}
On the grounds of \eqref{EIncSt} and \eqref{ELppDistId}, the process $\{\wt{\G}^{\bfu, z}_{\bfu+(1, 1)-\bfv, \bfu+(1, 1)}: (0, 0) \le \bfv \le \bfu\}$ has stationary increments: For $\bfv, \bfv' \in [(0, 0), \bfu]$ with $\bfv + \bfv' \le \bfu$, 
\begin{align}
\label{EIncStNE}
\begin{split}
\wt{\G}^{\bfu, z}_{\bfu+(1, 1)-\bfv'-\bfv, \bfu+(1, 1)}-\wt{\G}^{\bfu, z}_{\bfu+(1, 1)-\bfv', \bfu+(1, 1)} &\stackrel{\text{dist.}}{=} \G^z_{\bfv' + \bfv} - \G^z_{\bfv'} \stackrel{\text{dist.}}{=} \G^{z}_{\bfv} \\
&\stackrel{\text{dist.}}{=} \wt{\G}^{\bfu, z}_{\bfu+(1, 1)-\bfv, \bfu+(1, 1)}.  
\end{split}
\end{align} 

Identity \eqref{ELppDistId} also allows one to recast Lemma \ref{LExit2} in terms of the LPP process with northeast boundary as follows. 
\begin{lem}
\label{LExit3}
Fix $\delta > 0$. There exists a constant $c_0 = c_0(\delta) > 0$ such that the following statements hold for all $\bfu \in S_\delta \cap \bbZ_{>0}^2$ and $z \in (0, 1)$. 
\begin{enumerate}[\normalfont (a)]
\item If $z > \ShpMin_\bfu$ then $\bfP\left\{\wt{\G}^{\bfu, z}_{(1, 1), \bfu+(0, 1)} > \wt{\G}^{\bfu, z}_{(1, 1), \bfu + (1, 0)}\right\} \le \exp\{-c_0 \norm{\bfu}(z-\ShpMin_{\bfu})^3\}$. 
\item If $z < \ShpMin_\bfu$ then $\bfP\left\{\wt{\G}^{\bfu, z}_{(1, 1), \bfu+(0, 1)} < \wt{\G}^{\bfu, z}_{(1, 1), \bfu + (1, 0)}\right\} \le \exp\{-c_0 \norm{\bfu}(\ShpMin_{\bfu}-z)^3\}$. 
\end{enumerate}
\end{lem}
\begin{proof}
We verify (a) only, the proof of (b) being analogous. Let $\bfu \in S_\delta \cap \bbZ_{>0}^2$ and $z \in (0, 1)$ be such that $z > \ShpMin_\bfu$. Then, by virtue of \eqref{ELppDistId}, definition \eqref{EExitPt} and Lemma \ref{LExit2}(a), 
\begin{align*}
\begin{split}
\bfP\left\{\wt{\G}^{\bfu, z}_{(1, 1), \bfu+(0, 1)} > \wt{\G}^{\bfu, z}_{(1, 1), \bfu + (1, 0)}\right\} &= \bfP\{\G^z_{(1, 0), \bfu} > \G^z_{(0, 1), \bfu}\} = \P\{\Z^{z, \hor}_{\bfu} > 0\} \\
&\le \exp\{-c_0 \norm{\bfu}(z-\ShpMin_\bfu)^3\}
\end{split}
\end{align*}
for some constant $c_0 = c_0(\delta) > 0$. 
\end{proof}

We are now ready to derive Theorem \ref{TIncRTUB} with the strategy in the proof of \cite[Lemma 7.1]{Bala_Cato_Sepp_06}, which was invented in \cite{Cato_Groe_06} for the Poisson LPP. As mentioned after the theorem, our proof proceeds somewhat differently. More specifically, instead of the increment-comparison lemma (Lemma \ref{LComp}), the article \cite{Bala_Cato_Sepp_06} utilizes the monotonicity of the LPP increments with respect to weights \cite[Lemma 4.5]{Bala_Cato_Sepp_06} in a similar role. Also, the article \cite{Bala_Cato_Sepp_06} argues through the competition interface, which does not appear in our proof. While these differences do not seem serious, they still make it cumbersome to prove Theorem \ref{TIncRTUB} by pointing out the needed modifications to \cite{Bala_Cato_Sepp_06}.

Before presenting the complete proof, we describe its relatively simple overall structure. By \eqref{Emax}, it suffices to verify that  
\begin{align}
\P\{\G^{w, z}_{\bfv} - \G_{\bfv} \ge s\norm{\bfv}^{1/3}\} \le \exp\{-c \min \{s^{3/2}, s \norm{\bfv}^{1/3}\}\} \label{eq:NJ0}
\end{align}
for some constant $c = c(\delta) > 0$ under the assumptions of the theorem. Considering the first step of the geodesic $\pi_\bfv^{w, z}$, one can write 
\begin{align}
\begin{split}
\P\{\G^{w, z}_{\bfv} - \G_{\bfv} \ge s\norm{\bfv}^{1/3}\} &= \P\{\Z^{w, z, \hor}_{\bfv} > 0, \G^{w, \hor}_{\bfv} - \G_{\bfv} \ge s\norm{\bfv}^{1/3}\} \\
&+ \P\{\Z^{w, z, \ver}_{\bfv} > 0, \G^{z, \ver}_{\bfv} - \G_{\bfv} \ge s\norm{\bfv}^{1/3}\}. 
\end{split}
\label{E78}
\end{align}
Due to symmetry, it suffices to bound only the first term on the right-hand side. For any $p \in [m]$, definitions \eqref{E:Lpp1Bd} and \eqref{EExitPt} imply that
\begin{align}
\G^{w, \hor}_{\bfv} = \G^{w}_{\bfv} = \max \limits_{k \in [p]} \{\G^w_{(k, 0)} + \G_{(k, 1), \bfv}\} \label{E7}
\end{align}
on the event $\{\Z_\bfv^{w, z, \hor} \in [p]\}$. Therefore, a union bound gives 
\begin{align}
\label{E90}
\begin{split}
&\P\{\Z^{w, z, \hor}_{\bfv} > 0, \G^{w, \hor}_{\bfv} - \G_\bfv \ge s\norm{\bfv}^{1/3}\} \\
&\le \P\{\Z^{w, z, \hor}_{\bfv} > p\} + \P\bigg\{\max \limits_{k \in [p]} \{\G^w_{(k, 0)} + \G_{(k, 1), \bfv}\}-\G_{\bfv} \ge s\norm{\bfv}^{1/3}\bigg\}.
\end{split}
\end{align}

With suitable restrictions on the parameters, Lemma \ref{LExit} ensures that 
\begin{align}
\label{E:101}
\P\{\Z^{w, z, \hor}_{\bfv} > p\} \le \exp\{-c_0p^3 \norm{\bfv}^{-2}\}
\end{align}
for some constant $c_0 = c_0(\delta) > 0$. To bound the last probability in \eqref{E90}, one can first use the coupling in \eqref{EwNE} and apply Lemma \ref{LComp} to obtain 
\begin{align}
\label{E:102}
\G_{(k, 1), \bfv}-\G_{\bfv} = \wt{\G}^{\bfv, r}_{(k, 1), \bfv} - \wt{\G}_{(1, 1), \bfv}^{\bfv, r} \le \wt{\G}^{\bfv, r}_{(k, 1), \bfv+(1, 0)} - \wt{\G}_{(1, 1), \bfv+(1, 0)}^{\bfv, r}
\end{align}
where $r \in (0, 1)$ is a free parameter and $\wt{\G}^{\bfv, r}$ is the LPP with northeast boundary defined by \eqref{ELppNE}. As a consequence of planarity, on the event that the geodesic $\wt{\pi}^{\bfv, r}_{(1, 1), \bfv+(1, 1)}$ visits $\bfv + (1, 0)$, one has 
\begin{align}
\label{E:103}
\wt{\G}^{\bfv, r}_{(k, 1), \bfv+(1, 0)} - \wt{\G}_{(1, 1), \bfv+(1, 0)}^{\bfv, r} = \wt{\G}^{\bfv, r}_{(k, 1), \bfv+(1, 1)} - \wt{\G}_{(1, 1), \bfv+(1, 1)}^{\bfv, r}. 
\end{align}
See Figure \ref{fig}. 

\begin{figure}
	\begin{center}
		\begin{tikzpicture}[x=0.75pt,y=0.75pt,yscale=-1.2,xscale=1.2]
			
			\draw    (100.02,82) -- (101,250) ;
			\draw [shift={(100,79)}, rotate = 89.66] [fill={rgb, 255:red, 0; green, 0; blue, 0 }  ][line width=0.08]  [draw opacity=0] (8.93,-4.29) -- (0,0) -- (8.93,4.29) -- cycle    ;
			\draw    (101,250) -- (383,250) ;
			\draw [shift={(386,250)}, rotate = 180] [fill={rgb, 255:red, 0; green, 0; blue, 0 }  ][line width=0.08]  [draw opacity=0] (8.93,-4.29) -- (0,0) -- (8.93,4.29) -- cycle    ;
			
			\draw [color={rgb, 255:red, 0; green, 0; blue, 0 }  ,draw opacity=1 ][line width=2.25]    (337,115) -- (337,149) ;
			\draw [shift={(337,115)}, rotate = 90] [color={rgb, 255:red, 0; green, 0; blue, 0 }  ,draw opacity=1 ][fill={rgb, 255:red, 0; green, 0; blue, 0 }  ,fill opacity=1 ][line width=2.25]      (0, 0) circle [x radius= 3.22, y radius= 3.22]   ;
			\draw [color={rgb, 255:red, 0; green, 0; blue, 0 }  ,draw opacity=1 ][line width=2.25]    (337,149) -- (268,149) ;
			\draw [color={rgb, 255:red, 0; green, 0; blue, 0 }  ,draw opacity=1 ][line width=2.25]    (198,183) -- (198,217) ;
			\draw [color={rgb, 255:red, 0; green, 0; blue, 0 }  ,draw opacity=1 ][line width=2.25]    (268,183) -- (198,183) ;
			\draw [color={rgb, 255:red, 0; green, 0; blue, 0 }  ,draw opacity=1 ][line width=2.25]    (268,149) -- (268,183) ;
			\draw [color={rgb, 255:red, 0; green, 0; blue, 0 }  ,draw opacity=1 ][line width=2.25]    (198,217) -- (115,216) ;
			\draw [color={rgb, 255:red, 0; green, 0; blue, 0 }  ,draw opacity=1 ][line width=2.25]    (115,216) -- (115,239) ;
			\draw [shift={(115,239)}, rotate = 90] [color={rgb, 255:red, 0; green, 0; blue, 0 }  ,draw opacity=1 ][fill={rgb, 255:red, 0; green, 0; blue, 0 }  ,fill opacity=1 ][line width=2.25]      (0, 0) circle [x radius= 3.22, y radius= 3.22]   ;
			
			\draw [shift={(337,149)}, rotate = 90] [color={gray}  ,draw opacity=1 ][fill={gray}  ,fill opacity=1 ][line width=2.25]      (0, 0) circle [x radius= 2.14, y radius= 2.14]   ;
			\draw [dashed, color={gray}  ,draw opacity=1 ][line width=2.25]    (221,239) -- (138,239) ;
			\draw [shift={(138,239)}, rotate = 180] [color={gray}  ,draw opacity=1 ][fill={gray}  ,fill opacity=1 ][line width=2.25]      (0, 0) circle [x radius= 2.14, y radius= 2.14]   ;
			\draw [dashed, color={gray}  ,draw opacity=1 ][line width=2.25]    (221,205) -- (221,239) ;
			\draw [dashed, color={gray}  ,draw opacity=1 ][line width=2.25]    (299,240) -- (237,240) ;
			\draw [shift={(237,240)}, rotate = 180] [color={gray}  ,draw opacity=1 ][fill={gray} ,fill opacity=1 ][line width=2.25]      (0, 0) circle [x radius= 2.14, y radius= 2.14]   ;
			\draw [dashed, color={gray}  ,draw opacity=1 ][line width=2.25]    (249,183) -- (249,205) ;
			\draw [shift={(249,183)}, rotate = 90] [color={gray}  ,draw opacity=1 ][fill={gray}   ,fill opacity=1 ][line width=2.25]      (0, 0) circle [x radius= 2.14, y radius= 2.14]   ;
			\draw [dashed, color={gray}  ,draw opacity=1 ][line width=2.25]    (249,205) -- (221,205) ;
			\draw [dashed, color={gray}  ,draw opacity=1 ][line width=2.25]    (337,239) -- (310,239) ;
			\draw [shift={(310,239)}, rotate = 180] [color={gray}  ,draw opacity=1 ][fill={gray}  ,fill opacity=1 ][line width=2.25]      (0, 0) circle [x radius= 2.14, y radius= 2.14]   ;
			\draw [dashed, color={gray}  ,draw opacity=1 ][line width=2.25]    (299,240) -- (299,151) ;
			\draw [shift={(299,151)}, rotate = 270] [color={gray}  ,draw opacity=1 ][fill={gray} ,fill opacity=1 ][line width=2.25]      (0, 0) circle [x radius= 2.14, y radius= 2.14]   ;
			\draw [dashed, color={gray}  ,draw opacity=1 ][line width=2.25]    (337,149) -- (337,239) ;
			
			\draw (345,109) node [anchor=north west][inner sep=0.75pt]  [font=\scriptsize]  {${\textstyle (m+1, n+1)}$};
			\draw (345, 142) node [anchor=north west][inner sep=0.75pt]  [font=\scriptsize]  {${\textstyle (m + 1, n)}$};
			\draw (200, 170) node [anchor=north west][inner sep=0.75pt]  {${\textstyle \pi}$};
		\end{tikzpicture}
	\end{center}
	\caption{Illustrates the justification for equation \eqref{E:103} with $\bfv = (m, n)$. When the geodesic $\pi = \wt{\pi}_{(1, 1), \bfv + (1, 1)}^{\bfv, u}$ (black) visits $\bfv+(1, 0) = (m+1, n)$, for any $i \in [m+1]$, the geodesic $\pi' = \wt{\pi}_{(i, 1), \bfv + (1, 1)}^{\bfv, u}$ (any of the dashed gray) must intersect $\pi$ strictly before the endpoint $\bfv + (1, 1) = (m+1, n+1)$. Once they intersect, $\pi$ and $\pi'$ a.s.\ coalesce by the a.s.\ uniqueness of geodesics. In particular, $\pi'$ also visits $(m+1, n)$ a.s., which implies \eqref{E:103}.}
	\label{fig}
\end{figure}

Combining \eqref{E:102} and \eqref{E:103} with a union bound, one finds that the last probability in \eqref{E90}
\begin{align}
\label{E:104}
\begin{split}
&\P\bigg\{\max \limits_{k \in [p]} \{\G^w_{(k, 0)} + \G_{(k, 1), \bfv}-\G_{\bfv}\} \ge s\norm{\bfv}^{1/3}\bigg\} \\ 
&\le \P\bigg\{\max \limits_{k \in [p]} \{\G^w_{(k, 0)} + \wt{\G}^{\bfv, r}_{(k, 1), \bfv+(1, 1)} - \wt{\G}_{(1, 1), \bfv+(1, 1)}^{\bfv, r}\} \ge s\norm{\bfv}^{1/3}\bigg\} \\ 
&+ \P\{\bfv + (0, 1) \in \wt{\pi}^{\bfv, r}_{(1, 1), \bfv+(1, 1)}\}. 
\end{split}
\end{align}
A main point in \eqref{E:104} is that while the increment $\G_{(k, 1), \bfv}-\G_{\bfv}$ has a complicated distribution, the increment $\wt{\G}^{\bfv, r}_{(k, 1), \bfv+(1, 1)} - \wt{\G}_{(1, 1), \bfv+(1, 1)}^{\bfv, r}$ is a sum of IID exponentials and is also independent of $\G^w_{(k, 0)}$, which is another sum of IID exponentials. Consequently, the second probability in \eqref{E:104} can be bounded via Doob's inequality. Also, the last event in \eqref{E:104} is of the type covered in Lemma \ref{LExit3}(a) and, therefore, is at most $\exp\{-c_0 \norm{\bfv} (r-\zeta_\bfv)^3\}$ assuming that $r > \zeta_\bfv$. 

Our argument puts together the preceding two bounds with \eqref{E:101} and selects the free parameters $p$ and $r$ carefully with the aim of achieving the best possible upper bound for the widest range of the boundary parameters $w$ and $z$. This is done in the background, and involves balancing the restrictions of the exit point lemmas and the strengths of the three bounds combined together. 

\begin{proof}[Proof of Theorem \ref{TIncRTUB}]
Let $\epsilon_0 = \epsilon_0(\delta)$, $N_0 = N_0(\delta)$ and $s_0 = s_0(\delta)$ denote positive constants to be determined below. Let $\bfv = (m, n) \in S_\delta \cap \bbZ_{\ge N_0}^2$, $s \ge s_0$, $w = \ShpMin_{\bfv}-\epsilon_0 \min \{s^{1/2}\norm{\bfv}^{-1/3}, 1\}$ and $z = \ShpMin_{\bfv} + \epsilon_0 \min \{s^{1/2}\norm{\bfv}^{-1/3}, 1\}$. By Lemma \ref{LMinEst}(b), $\ShpMin_\bfv \in [a_0, 1-a_0]$ for some constant $a_0 = a_0(\delta) \in (0, 1/2)$. Therefore, $w, z \in [a_0/2, 1-a_0/2]$ upon choosing $\epsilon_0 \le a_0/2$.   

Let $p = \min \{\lf s^{1/2} \norm{\bfv}^{2/3} \rf, m\}$ and $E = \{\Z^{w, z, \hor}_{\bfv} > p\}$. Then the bound in \eqref{E:101} takes the form
\begin{align}
\label{E91}
\begin{split}
	\P\{E\} = \P\{\Z_{\bfv}^{w, z, \hor} > s^{1/2}\norm{\bfv}^{2/3}\} \le \exp\{-c_0s^{3/2}\}
\end{split}
\end{align}
for some constant $c_0 = c_0(\delta) > 0$. In the case $p < m$, the equality in \eqref{E91} holds because the exit points are integer-valued. The same equality also holds in the remaining case $p = m$ because then the both events there are empty. The subsequent inequality in \eqref{E91} holds for sufficiently small $\epsilon_0$ and sufficiently large $N_0 = N_0(s_0)$ by virtue of Lemma \ref{LExit}(a). The hypotheses of the lemma are in place as long as 
$s_0^{1/2} \ge N_0^{-2/3}$ and $\epsilon_0 \le \wt{\epsilon}_0$ where $\wt{\epsilon}_0$ refers to the constant $\epsilon_0$ in the lemma.  

We now turn to bounding the last term in \eqref{E90}. Since this probability is trivially zero when $p = 0$, assume that $p > 0$ from here on. We break the argument into two cases. 

The case $s > 2w^{-1}\norm{\bfv}^{2/3}$ can be handled with a simple argument as follows. One has 
\begin{align}
\label{E87}
\begin{split}
&\P\bigg\{\max \limits_{k \in [p]} \{\G^w_{(k, 0)} + \G_{(k, 1), \bfv}\}-\G_{\bfv} \ge s\norm{\bfv}^{1/3}\bigg\} \le \P\{\G^{w}_{(m, 0)} \ge s\norm{\bfv}^{1/3}\} \\
&\le \P\left\{w\G^w_{(m, 0)} \ge m + \frac{1}{2}ws\norm{\bfv}^{1/3}\right\} \le \exp\left\{-2b_0ws\norm{\bfv}^{1/3}\right\} \le \exp\left\{-a_0 b_0 s\norm{\bfv}^{1/3}\right\}
\end{split}
\end{align}
for some absolute constant $b_0 > 0$. The first inequality above results from dropping the nonpositive terms $\G_{(k, 1), \bfv}-\G_{\bfv}$ for $k \in [p]$, and the monotonicity of the sequence $(\G^{w}_{(k, 0)})_{k \in \bbZ_{>0}}$. To obtain the second inequality, invoke Lemma \ref{LExpTUB}(a) with the sum $w\G^{w}_{(m, 0)} = \sum_{i = 1}^m w\rw^w_{(i, 0)}$ of independent $\Exp(1)$-distributed weights, noting that in the present case  
\begin{align*}
	\frac{1}{2}ws\norm{\bfv}^{1/3} > m = \E[w\G^w_{(m, 0)}]. 
\end{align*}
The last step in \eqref{E87} uses that $w \ge a_0/2$. 

In the remaining, more intricate case $s \le 2w^{-1}\norm{\bfv}^{2/3}$, we carry out the strategy sketched before the proof. Choose the free parameter in \eqref{E:102} as $r= \ShpMin_\bfv + \epsilon_0 s^{1/2}\norm{\bfv}^{-1/3}$, the reason for which will be made more clear along the proof. Since $w \ge a_0/2$, taking $\epsilon_0 \le a_0^{3/2}/4$ yields the upper bound in 
\begin{align}
	a_0 \le \ShpMin_{\bfv} < r \le \ShpMin_{\bfv} + \epsilon_0 2^{1/2}w^{-1/2} \le 1- a_0 + 2\epsilon_0a_0^{-1/2} \le 1- a_0/2  \label{E86} 
\end{align}  
for $r$. In particular, $r \in (0, 1)$ is an admissible parameter for the $\wt{\G}^{\bfv, r}$-process defined at \eqref{ELppNE}. (This also indicates the role of the small $\epsilon_0$ factor in the choice of $r$). 

Continuing from \eqref{E:104} and defining the event 
\begin{align}
	F = \{\wt{\G}_{(1, 1), \bfv + (0, 1)}^{\bfv, r} > \wt{\G}_{(1, 1), \bfv + (1, 0)}^{\bfv, r}\} \stackrel{\text{a.s.}}{=} \{\bfv + (0, 1) \in \wt{\pi}^{\bfv, r}_{(1, 1), \bfv+(1, 1)}\}, \label{E97}
\end{align}
one obtains that 
\begin{align}
\label{E92}
\begin{split}
&\P\bigg\{\max\limits_{k \in [p]} \{\G^w_{(k, 0)} + \G_{(k, 1), \bfv}\}-\G_{\bfv} \ge s\norm{\bfv}^{1/3}\bigg\} \\
&\le \P\{F\} + \P\bigg\{\max \limits_{k \in [p]} \{\G^w_{(k, 0)} + \wt{\G}_{(k, 1), \bfv + (1, 1)}^{\bfv, r}-\wt{\G}_{(1, 1), \bfv + (1, 1)}^{\bfv, r}\} \ge s\norm{\bfv}^{1/3}\bigg\} \\
&\le \P\{F\} + \P\left\{\rw^{w}_{(1, 0)} \ge \frac{1}{2}s\norm{\bfv}^{1/3}\right\} \\
&+\P\bigg\{\max \limits_{k \in [p] \smallsetminus\{1\}} \{\G^w_{(2, 0), (k, 0)} + \wt{\G}_{(k, 1), \bfv + (1, 1)}^{\bfv, r}-\wt{\G}_{(1, 1), \bfv + (1, 1)}^{\bfv, r}\} \ge \frac{1}{2}s\norm{\bfv}^{1/3}\bigg\}.
\end{split}
\end{align}
The last step in \eqref{E92} uses a union bound to separate the term with $k = 1$ within the maximum. 

Decreasing $c_0 = c_0(\delta)$ if necessary, and applying Lemma \ref{LExit3}(a), one obtains that
\begin{align}
\P\{F\} \le \exp\{-c_0\epsilon_0^3 s^{3/2}\}, \label{E98}
\end{align}
(The preceding bound is the motivation for choosing $r$ order $s^{1/2}\norm{\bfv}^{-1/3}$ above $\zeta_\bfv$). Since $\rw^{w}_{(1, 0)} \sim \Exp(w)$, the second-last probability in \eqref{E92} equals  
\begin{align}
	\label{E100}
	\P\left\{\rw^{w}_{(1, 0)} \ge \frac{1}{2}s\norm{\bfv}^{1/3}\right\}=\exp\left\{-\frac{1}{2}w s \norm{\bfv}^{1/3}\right\} \le \exp\left\{-\frac{1}{4}a_0 s\norm{\bfv}^{1/3}\right\}. 
\end{align}
It remains to bound the last probability in \eqref{E92}. This term equals zero if $p = 1$. Assume now that $p > 1$. Introduce the sequence 
\begin{align}
\label{E81}
\begin{split}
M_k &= \G^w_{(2, 0), (k, 0)} + \wt{\G}_{(k, 1), \bfv + (1, 1)}^{\bfv, r}-\wt{\G}_{(1, 1), \bfv + (1, 1)}^{\bfv, r} - \frac{k-1}{w} + \frac{k-1}{r}\quad \text{ for } k \in [p] \smallsetminus \{1\}. 
\end{split}
\end{align}
The contribution from the last two terms can be bounded as follows. 
\begin{align}
	\label{E82}
	\begin{split}
		0 \le \frac{k-1}{w} - \frac{k-1}{r} &\le (p-1) \cdot \frac{r-w}{rw} \le s^{1/2}\norm{\bfv}^{2/3} \cdot \frac{2}{a_0^2} \cdot 2\epsilon_0 s^{1/2} \norm{\bfv}^{-1/3}\\
		&= \frac{4\epsilon_0}{a_0^2} \cdot s \norm{\bfv}^{1/3} \le \frac{1}{4} s \norm{\bfv}^{1/3}. 
	\end{split}
\end{align}
The first line of \eqref{E82} is due to the choices of $w$, $p$ and $r$, and the bounds $r \ge a_0$ and $w \ge a_0/2$. For the last inequality in \eqref{E82}, reduce $\epsilon_0$ further if necessary to have $\epsilon_0 \le a_0^2/16$. (Note that one cannot eliminate the mean term in \eqref{E82} with the choice $r = w$ because $w < \zeta_\bfv$ while one needs $r > \zeta_\bfv$ for \eqref{E98}). From \eqref{E81} and \eqref{E82}, one obtains that 
\begin{align}
	\label{E85}
	\begin{split}
		\P&\bigg\{\max \limits_{k \in [p] \smallsetminus \{1\}} \{\G^w_{(2, 0), (k, 0)} + \wt{\G}_{(k, 1), \bfv + (1, 1)}^{\bfv, r}-\wt{\G}_{(1, 1), \bfv + (1, 1)}^{\bfv, r}\} \ge \frac{1}{2}s\norm{\bfv}^{1/3}\bigg\} \\
		&= \P\bigg\{\max \limits_{k \in [p] \smallsetminus \{1\}} \bigg\{M_k + \frac{k-1}{w} - \frac{k-1}{r}\bigg\}\ge \frac{1}{2}s\norm{\bfv}^{1/3}\bigg\} \\
		&\le \P\bigg\{\max \limits_{k \in [p] \smallsetminus \{1\}} M_k \ge \frac{1}{4}s\norm{\bfv}^{1/3}\bigg\}. 
	\end{split}
\end{align}
Next recall that $\G_{(2, 0), (k, 0)}^{w}$ is a sum of $k-1$ independent $\Exp(w)$-distributed weights for each $k \in \bbZ_{>0}$. Also, on account of \eqref{EIncStNE}, 
\begin{align}
	\wt{\G}_{(1, 1), \bfv + (1, 1)}^{\bfv, r} -\wt{\G}_{(k, 1), \bfv+(1, 1)}^{\bfv, r} \stackrel{\text{dist.}}{=} \G_{(k-1, 0)}^{r} \quad \text{ for each } k \in [m+1],  \label{E9}
\end{align}
where the right-hand side is a sum of $k-1$ independent $\Exp(r)$-distributed weights. Furthermore, the left-hand side of \eqref{E9} as a sequence in $k \in [m+1]$, and the sequence $(\G_{(k, 0)}^w)_{k \in \bbZ_{>0}}$ are independent due to being defined from disjoint collections of weights. Therefore, it follows from Lemma \ref{LExpMaxIneq} combined with the bounds $r, w \in [a_0/2, 1-a_0/2]$ and $1 < p \le s^{1/2}\norm{\bfv}^{2/3}$ that the last probability in \eqref{E85} is at most 
\begin{align}
	\label{E99}
	\begin{split}
		\exp\left\{-d_0 s \norm{\bfv}^{1/3} \min\left\{\frac{s\norm{\bfv}^{1/3}}{p}, 1\right\} \right\} \le \exp\{-d_0 \min\{s^{3/2}, s\norm{\bfv}^{1/3}\}\} 
	\end{split}
\end{align}
for some constant $d_0 = d_0(\delta) > 0$. Putting together \eqref{E92}, \eqref{E98}, \eqref{E100}, \eqref{E85} and \eqref{E99} establishes the bound 
\begin{align}
	&\P\bigg\{\max \limits_{k \in [p]} \{\G^w_{(k, 0)} + \G_{(k, 1), \bfv}\}-\G_{\bfv} \ge s\norm{\bfv}^{1/3}\bigg\} \le 3\exp\{-c \min \{s^{3/2}, s\norm{\bfv}^{1/3}\}\} \label{E101}
\end{align}
for some constant $c = c(\delta) > 0$ in the case $s \le 2w^{-1}\norm{\bfv}^{2/3}$. 

Now combine \eqref{E87} and \eqref{E101} with \eqref{E90} and \eqref{E91} to obtain 
\begin{align*}
	\P\{\Z^{w, z, \hor}_{\bfv} > 0, \G^{w, \hor}_{\bfv} - \G_{\bfv} \ge s\norm{\bfv}^{1/3}\} \le 4\exp \{-c \min \{s^{3/2}, s\norm{\bfv}^{1/3}\}\} 
\end{align*}
after decreasing $c$ if necessary. By choosing $s_0$ sufficiently large and adjusting $c$, the factor of $4$ can be replaced with $1/2$. This finishes the proof since the second term on the right-hand side of \eqref{E78} also obeys a similar bound. 
\end{proof}

\subsection{Proof of Theorem \ref{TIncRTLB}}

We next adapt the approach in the proof of Theorem \ref{TIncRTUB} to establish Theorem \ref{TIncRTLB}. 

We will need some basic estimates on the increments of the function $\zeta$ defined by \eqref{EShpMin}. To obtain these estimates, we first connect $\zeta$ to the mean function $\M$ at \eqref{EMeanF}. For each $\bfx = (x, y) \in \bbR_{>0}^2$, the derivative 
\begin{align}
\partial_z \M^z_{\bfx} = -\frac{x}{z^2} + \frac{y}{(1-z)^2} \quad \text{ for } z \in (0, 1) \label{EMeanDer}
\end{align}
is an increasing function with the unique zero at $\ShpMin_\bfx$ and, therefore, $\zeta_\bfx$ is the unique minimizer of the function $z \mapsto \M_{\bfx}^z$. In fact, the minimum value is given by the shape function introduced in \eqref{EShp}: 
\begin{align}
\label{EShpMean}
\shp_{\bfx} = \M^{\zeta_\bfx}_{\bfx} = \inf_{z \in (0, 1)} \M^z_{\bfx} \quad \text{ for } \bfx \in \bbR_{>0}^2.   
\end{align}
The next lemma records that the minimizer changes approximately linearly with shifts. 

\begin{lem}
\label{LMinShift}
Let $\delta > 0$ and $K \ge 0$. There exist positive constants $C_0 = C_0(\delta, K)$ and $c_0 = c_0(\delta, K)$ such that the following inequalities hold for $\bfx \in S_\delta$ and $0 \le h \le K\norm{\bfx}$. 
\begin{enumerate}[\normalfont (a)]
\item $c_0h\norm{\bfx}^{-1} \le \ShpMin_{\bfx + (h, 0)}-\ShpMin_{\bfx} \le C_0h{\norm{\bfx}}^{-1}$
\item $c_0h\norm{\bfx}^{-1} \le \ShpMin_{\bfx}-\ShpMin_{\bfx + (0, h)} \le C_0h\norm{\bfx}^{-1}$. 
\end{enumerate}
\end{lem}
\begin{proof}[Proof of {\rm(}a{\rm)}]
Let $(x, y) \in S_\delta$ and $0 \le h \le K(x+y)$. Then $x+h \le (1+K) (x+y) \le (1+K)(1+\delta^{-1})y$. This shows that $(x+h, y) \in S_\epsilon$ where $\epsilon =  (1+K)^{-1}(1+\delta)^{-1}$. Abbreviate $\ShpMin = \ShpMin_{x, y}$ and $\wt{\ShpMin} = \ShpMin_{x+h, y}$. Because the derivatives $\partial^z \M^z_{x, y}$ and $\partial^z \M^{z}_{x+h, y}$ vanish at $z = \ShpMin$ and $z = \wt{\ShpMin}$, respectively, one has  
\begin{align}
\label{Eq-46}
\begin{split}
\frac{h}{\wt{\ShpMin}^2} &= \frac{y}{(1-\wt{\ShpMin})^2}-\frac{x}{\wt{\ShpMin}^2} = \frac{y}{(1-\wt{\ShpMin})^2}-\frac{y}{(1-\ShpMin)^2}-\frac{x}{\wt{\ShpMin}^2} + \frac{x}{\ShpMin^2} \\
&= (\wt{\ShpMin}-\ShpMin) \cdot \left(\frac{y(2-\wt{\ShpMin}-\ShpMin)}{(1-\wt{\ShpMin})^2(1-\ShpMin)^2} + \frac{x(\wt{\ShpMin} + \ShpMin)}{\wt{\ShpMin}^2 \ShpMin^2}\right). 
\end{split}
\end{align}
Since $(x, y) \in S_\delta$ and $(x+h, y) \in S_\epsilon$, part (a) follows from \eqref{Eq-46} and Lemma \ref{LMinEst}(b). 
\end{proof}

\begin{proof}[Proof of Theorem \ref{TIncRTLB}{\rm(}a{\rm)}]
Let $\epsilon_0 = \epsilon_0(\delta, K)$, $s_0 = s_0(\delta)$ and $N_0 = N_0(\delta, K) \ge \epsilon_0^{-3/2}s_0^{3/2}$ denote positive constants to be chosen later in the proof. Pick $\bfv = (m, n) \in S_\delta \cap \bbZ_{\ge N_0}^2$ and $s \in [s_0, \epsilon_0 \norm{\bfv}^{2/3}]$. To prove (a), it suffices to check the case $w = \ShpMin_{\bfv} + Ks^{1/2} \norm{\bfv}^{-1/3}$ due to monotonicity. Let $k = \lc s^{1/2}\norm{\bfv}^{2/3}\rc$.  With $\epsilon_0$ chosen sufficiently small, $1 < k \le m$ and $\wt{\bfv} = (m-k+1, n) \in S_{\delta/2}$. Then, by Lemma \ref{LMinEst}(b), $\ShpMin_{\wt{\bfv}} \in [a_0, 1-a_0]$ for some constant $a_0 = a_0(\delta) > 0$. Also, Lemma \ref{LMinShift}(a) implies the existence of constants $b_0 = b_0(\delta) > 0$ and $B_0 = B_0(\delta) > 0$ such that  
\begin{align}
\label{Eq-100}
\ShpMin_{\bfv}-\ShpMin_{\wt{\bfv}} \in [2b_0 (k-1) \norm{\bfv}^{-1}, B_0 (k-1) \norm{\bfv}^{-1}] \subset [b_0 s^{1/2}\norm{\bfv}^{-1/3}, B_0 s^{1/2}\norm{\bfv}^{-1/3}]. 
\end{align}

Introduce $z = \ShpMin_{\wt{\bfv}}-R_0 s^{1/2}\norm{\bfv}^{-1/3}$ where $R_0 = R_0(\delta, K) > 0$ is a constant be tuned below. After decreasing $\epsilon_0 = \epsilon_0(R_0)$ if necessary, $z \ge a_0/2 > 0$. In particular, the northeast LPP process $\wt{\G}^{\bfv, z}$ defined at \eqref{ELppNE} makes sense. Recall also the notation $\wt{\pi}^{\bfv, z}$ for the associated geodesics. The following display uses definition \eqref{E:Lpp1Bd} in the first step, and then appeals to the coupling in  \eqref{EwNE} along with definitions \eqref{EBlkLPP} and \eqref{ELppNE} for the subsequent equality.  The second inequality in \eqref{Eq-101} is due to Lemma \ref{LComp}(a). The final inequality comes from a union bound and planarity.
\begin{align}
\label{Eq-101}
\begin{split}
&\P\{\G^{w, \hor}_{\bfv}-\G_{\bfv} \ge s\norm{\bfv}^{1/3}\} \ge \P\{\G^w_{(k, 0)} + \G_{(k, 1), \bfv}-\G_{\bfv} \ge s \norm{\bfv}^{1/3}\} \\
&=\P\{\G^w_{(k, 0)} + \wt{\G}^{\bfv, z}_{(k, 1), \bfv}-\wt{\G}^{\bfv, z}_{(1, 1), \bfv} \ge s \norm{\bfv}^{1/3}\} \\
&\ge \P\{\G^w_{(k, 0)} + \wt{\G}^{\bfv, z}_{(k, 1), \bfv+(0, 1)}-\wt{\G}^{\bfv, z}_{(1, 1), \bfv+(0, 1)} \ge s \norm{\bfv}^{1/3}\} \\
&\ge \P\{\G^w_{(k, 0)} + \wt{\G}^{\bfv, z}_{(k, 1), \bfv+(1, 1)}-\wt{\G}^{\bfv, z}_{(1, 1), \bfv+(1, 1)} \ge s \norm{\bfv}^{1/3}\}-\P\{\bfv + (1, 0) \in \wt{\pi}^{\bfv, z}_{(k, 1), \bfv + (1, 1)}\}. 
\end{split}
\end{align}
Recall now that $\G^{w}_{(k, 0)}$ is a sum of $k$ independent $\Exp(w)$-distributed weights. Also, \eqref{EIncStNE} shows that $\wt{\G}^{\bfv, z}_{(k, 1), \bfv+(1, 1)}-\wt{\G}^{\bfv, z}_{(1, 1), \bfv+(1, 1)} \stackrel{\text{dist}}{=} \G^z_{(k-1, 0)}$ is a sum of $k-1$ independent $\Exp(z)$-distributed increments. Furthermore, these two sums are independent being defined from disjoint sets of weights. The preceding observations lead to 
\begin{align}
\label{Eq-102}
\begin{split}
&\P\{\G^w_{(k, 0)} + \wt{\G}^{\bfv, z}_{(k, 1), \bfv+(1, 1)}-\wt{\G}^{\bfv, z}_{(1, 1), \bfv+(1, 1)} \ge s \norm{\bfv}^{1/3}\} \\
&\ge \P\left\{\G^w_{(k, 0)} \ge \frac{k}{z} + 2s \norm{\bfv}^{1/3}, \wt{\G}^{\bfv, z}_{(1, 1), \bfv+(1, 1)} - \wt{\G}^{\bfv, z}_{(k, 1), \bfv+(1, 1)} \le \frac{k-1}{z} + s\norm{\bfv}^{1/3}\right\} \\
&= \P\left\{\G^w_{(k, 0)} \ge \frac{k}{z} + 2s \norm{\bfv}^{1/3}\right\} \cdot \P\left\{\G^z_{(k-1, 0)} \le \frac{k-1}{z} + s\norm{\bfv}^{1/3}\right\}. 
\end{split}
\end{align}

We proceed to develop lower bounds for the last two probabilities in \eqref{Eq-102}. From the inequalities $w \ge a_0$ and $z \ge a_0/2$, and the choices of $k$, $w$ and $z$, one obtains that 
\begin{align}
\label{Eq-103}
\begin{split}
0 &\le \frac{k}{z}-\frac{k}{w} = \frac{k(w-z)}{wz} \le 2a_0^{-2} \cdot k \cdot (w-\ShpMin_{\bfv} + \ShpMin_{\bfv}-\ShpMin_{\wt{\bfv}} + \ShpMin_{\wt{\bfv}}-z) \\
&\le C_0 \cdot (K+ B_0 + R_0) \cdot s \norm{\bfv}^{1/3}
\end{split}
\end{align}
for some constant $C_0 = C_0(\delta) > 0$. It follows from \eqref{Eq-103} that 
\begin{align}
\label{Eq-104}
\begin{split}
&\P\left\{\G^w_{(k, 0)} \ge \frac{k}{z} + 2s \norm{\bfv}^{1/3}\right\} \ge \P\left\{w\G^w_{(k, 0)} \ge k + C_0 (K+B_0 + R_0 + 1) s \norm{\bfv}^{1/3}\right\} \\
&\ge \exp\{-D_0 (K^2 + B_0^2 + R_0^2 + 1)s^2 \norm{\bfv}^{2/3} k^{-1}\} \ge \exp\{-D_0 (K^2 + B_0^2 + R_0^2 + 1)s^{3/2}\}
\end{split}
\end{align}
for some absolute constant $D_0 > 0$ and sufficiently large $s_0$ and $N_0$. For the first inequality in \eqref{Eq-104}, use that $w \le 1$ for sufficiently small $\epsilon_0 = \epsilon_0(K)$, and take $C_0 \ge 2$. The second inequality invokes Lemma \ref{LExpTLB}(a) (with $T = 1$), noting from the inequality $\epsilon_0^{-1} \le s^{-1}\norm{\bfv}^{2/3}$ and the choice of $k$ that the deviation 
\begin{align}
\begin{split}
&C_0 (K+B_0 + R_0 + 1) s \norm{\bfv}^{1/3} \le C_0 (K+B_0 + R_0 + 1) \epsilon_0^{1/2} \cdot \epsilon_0^{-1/2}s \norm{\bfv}^{1/3} \\ 
&\le s^{1/2}\norm{\bfv}^{2/3} \le k
\end{split}
\end{align}
for sufficiently small $\epsilon_0 = \epsilon_0(C_0, B_0, K, R_0)$. Next apply Lemma \ref{LExpTUB}(a) to obtain 
\begin{align}
\label{Eq-105}
\begin{split}
&\P\left\{\G^z_{(k-1, 0)} \le \frac{k-1}{z} + s\norm{\bfv}^{1/3}\right\} \ge 1 - \exp\{-2d_0 s^2 \norm{\bfv}^{2/3}k^{-1}\} \\
&\ge 1-\exp\{-d_0 s^{3/2}\} \ge 1/2
\end{split}
\end{align}
for some constant $d_0 = d_0(\delta) > 0$ and sufficiently large $s_0 = s_0(d_0)$. 

Consider now the last probability in \eqref{Eq-101}. By shift invariance and Lemma \ref{LExit3}, 
\begin{align}
\label{Eq-106}
\begin{split}
&\P\left\{\bfv + (1, 0) \in \wt{\pi}^{\bfv, z}_{(k, 1), \bfv + (1, 1)}\right\} = \P\left\{\wt{\bfv} + (1, 0) \in \wt{\pi}^{\wt{\bfv}, z}_{(1, 1), \wt{\bfv}+(1, 1)}\right\} \\
&\le \exp\{-c_0 \norm{\bfv} (\ShpMin_{\wt{\bfv}}-z)^3\} = \exp\{-c_0 R_0^3s^{3/2}\}
\end{split}
\end{align}
for some constant $c_0 = c_0(\delta) > 0$. 

Collecting the bounds from \eqref{Eq-101}, \eqref{Eq-102}, \eqref{Eq-104}, \eqref{Eq-105} and \eqref{Eq-106} results in 
\begin{align}
\label{Eq-107}
\begin{split}
\P\{\G^{w, \hor}_{\bfv}-\G_{\bfv} \ge s\norm{\bfv}^{1/3}\}  &\ge \frac{1}{2}\exp\{-D_0 (K^2 + B_0^2 + R_0^2 + 1)  s^{3/2}\} - \exp\{-c_0 R_0^3 s^{3/2}\} \\
&\ge \frac{1}{4}\exp\{-D_0 (K^2 + B_0^2 + R_0^2 + 1)  s^{3/2}\}
\end{split}
\end{align}
provided that $R_0 = R_0(B_0, c_0, D_0, K)$ is sufficiently large. The claim in (a) then follows upon increasing $s_0 = s_0(D_0)$ if necessary. 
\end{proof}

\section{Proofs of the bounds for bulk LPP}
\label{S:PfBlk}

We continue with the proofs of the bounds in Subsection \ref{Ss:BlkBd} for the bulk LPP. 

\subsection{Proof of Theorem \ref{T:BlkLT}}

Similarly to the argument mentioned after \eqref{E:100}, our proof of Theorem \ref{T:BlkLT} is based on the containment 
\begin{align}
\label{E:105}
\{\G_{\bfv} \le \shp_{\bfv}-s\norm{\bfv}^{1/3}\} &\subset \left\{\G_{\bfv}^{w, z} \le \shp_{\bfv}-\frac{1}{2}s\norm{\bfv}^{1/3}\right\} \cup \left\{\G_{\bfv}^{w, z}-\G_{\bfv} \ge \frac{1}{2}s\norm{\bfv}^{1/3}\right\}
\end{align}
with suitably chosen boundary parameters $w$ and $z$. Then the probability of the last event in \eqref{E:105} can be bounded via Theorem \ref{TIncRTUB}. Hence, it remains to develop sufficient control over the left tail of the $\G^{w, z}$-process, which will be achieved in Lemma \ref{LLppBdLT} below through identity \eqref{E:Rains}. 

We will use a Taylor estimate from \cite{Emra_Janj_Sepp_23} for the mean function in \eqref{EMean}. Define the function 
\begin{align}
\curv_{\bfx} = \frac{1}{2}\partial_{z}^2|_{z = \ShpMin_{\bfx}}\{\M^z_{\bfx}\} = \frac{(\sqrt{x}+\sqrt{y})^{4/3}}{x^{1/6}y^{1/6}} \quad \text{ for } \bfx = (x, y) \in \bbR_{>0}  \label{ECurv}
\end{align}
where $\zeta_\bfx$ is given by \eqref{EShpMin} and is the unique minimizer at \eqref{EShpMean}. 
\begin{lem}[Lemma C.2 in \cite{Emra_Janj_Sepp_23}]
\label{LMeanEst}
Fix $\delta > 0$ and $\epsilon > 0$. There exists a constant $C_0 = C_0(\delta, \epsilon) > 0$ such that 
\begin{align*}
|\M^z_{\bfx}-\shp_{\bfx} - \curv_{\bfx}^3(z-\ShpMin_{\bfx})^2| \le C_0 \norm{\bfx}|z-\ShpMin_{\bfx}|^3 
\end{align*}
whenever $\bfx \in S_\delta$ and $z \in [\epsilon, 1-\epsilon]$. 
\end{lem}
The following lemma estimates the exponent in identity \eqref{E:Rains}. 
\begin{lem}
\label{LMgfEst}
 Fix $\delta > 0$ and $\epsilon > 0$. There exists a constant $C_0 = C_0(\delta, \epsilon) > 0$ such that 
\begin{align*}
&\bigg|\int_z^w \M^t_{\bfx} \dd t-(w-z)\shp_{\bfx} - \dfrac{1}{3} \curv_{\bfx}^3\{(w-\ShpMin_{\bfx})^3-(z-\ShpMin_{\bfx})^3\}\bigg| \le C_0 \norm{\bfx}\{(w-\ShpMin_{\bfx})^4 + (z-\ShpMin_{\bfx})^4\} 
\end{align*}
whenever $\bfx \in S_\delta$ and $w, z \in [\epsilon, 1-\epsilon]$. 
\end{lem}
\begin{proof}
This comes from integrating the estimate in Lemma \ref{LMeanEst}. 
\end{proof}

We combine the preceding lemmas to produce the following left-tail upper bound. 
\begin{lem}
\label{LLppBdLT}
Let $\delta > 0$. There exist positive constants $C_0 = C_0(\delta)$ and $\epsilon_0 = \epsilon_0(\delta)$ such that
\begin{align*}
\log \P\{\G^{w, z}_{\bfv} \le \shp_{\bfv}-\curv_{\bfv}s\} \le -\frac{1}{3}(y^3-x^3)-(y-x)s + \frac{C_0(x^4 + y^4)}{\norm{\bfv}^{1/3}}
\end{align*}
whenever $\bfv \in S_\delta \cap \bbZ_{>0}^2$, $s \in \bbR$, and $w = \ShpMin_{\bfv} + x\curv_{\bfv}^{-1}$ and $z = \ShpMin_{\bfv} + y\curv_{\bfv}^{-1}$ for some $x, y \in [-\epsilon_0\norm{\bfv}^{1/3}, \epsilon_0\norm{\bfv}^{1/3}]$ with $x \le y$. 
\end{lem}
\begin{proof}
Let $\bfv \in S_\delta \cap \bbZ_{>0}^2$, $s \in \bbR$ and $w, z \in (0, 1)$ with $w \le z$. Apply the exponential Markov inequality with the exponent $w-z \le 0$ and then appeal \eqref{E:Rains} to obtain 
\begin{align}
\label{E1}
\begin{split}
\log \P\{\G^{w, z}_{\bfv} \le \shp_{\bfv}-\curv_{\bfv}s\} &\le \log \E[\exp\{(w-z)\G_{\bfv}^{w, z}\}] - (w-z)(\shp_{\bfv}-\curv_{\bfv}s) \\
&= \int_z^w \M^t_{\bfv} \dd t - (w-z)(\shp_{\bfv}-\curv_{\bfv} s).  
\end{split}
\end{align}

Let $\epsilon_0 = \epsilon_0(\delta) > 0$ be a constant to be specified below. Set $w = \ShpMin_{\bfv} + x/\curv_{\bfv}$ and $z = \ShpMin_{\bfv} + y/\curv_{\bfv}$ for some $x, y \in [-\epsilon_0\norm{\bfv}^{1/3}, \epsilon_0\norm{\bfv}^{1/3}]$ with $x \le y$. On account of Lemma \ref{LMinEst}(b)-(c), choosing $\epsilon_0$ sufficiently small ensures that $w, z \in [c_0, 1-c_0]$ for some constant $c_0 = c_0(\delta) > 0$. Then, estimating the integral in \eqref{E1} via Lemma \ref{LMgfEst} leads to 
\begin{align}
\label{E2}
\begin{split}
\log \P\{\G^{w, z}_{\bfv} \le \shp_{\bfv}-\curv_{\bfv}s\} &\le \frac{1}{3}\curv_{\bfv}^3\{(w-\ShpMin_{\bfv})^3 - (z-\ShpMin_{\bfv})^3\} + (w-z)\curv_{\bfv}s \\
&+ C_0 \norm{\bfv}\{(z-\ShpMin_{\bfv})^4 + (w-\ShpMin_{\bfv})^4\}\\
&= \frac{1}{3}(x^3-y^3)+(x-y)s + C_0\norm{\bfv}\curv_{\bfv}^{-4}(x^4 + y^4)
\end{split}
\end{align}
for some constant $C_0 = C_0(\delta) > 0$. Now using the lower bound in Lemma \ref{LMinEst}(c) completes the proof. 
\end{proof}

We now derive Theorem \ref{T:BlkLT} from Theorem \ref{TIncRTUB} and Lemma \ref{LLppBdLT}. 

\begin{proof}[Proof of Theorem \ref{T:BlkLT}]
Let $N_0 = N_0(\delta)$ and $s_0 = s_0(\delta)$ denote positive constants to be chosen in the course of the proof. Let $\bfv \in S_\delta \cap \bbZ_{\ge N_0}^2$ and $s \ge s_0$. It suffices to consider the case $s \le \shp_{\bfv}\norm{\bfv}^{-1/3}$ because the left tail event $\{\G_\bfv \le \gamma_\bfv-s \norm{\bfv}^{1/3}\}$ is empty otherwise. Then $s \le 2\norm{\bfv}^{2/3}$ due to Lemma \ref{LMinEst}(a). Put $w = \ShpMin_\bfv - \epsilon_0 s^{1/2}\curv_{\bfv}^{-1}$ where $\epsilon_0 = \epsilon_0(\delta) > 0$ is another constant to be selected. By Lemma \ref{LMinEst}(b), if $\epsilon_0$ is sufficiently small then $w > 0$ and the process $\G^{w, \ShpMin_\bfv}$ makes sense. 

The next inequality follows from \eqref{E:105} and a union bound. 
\begin{align}
\label{E4}
\begin{split}
\P\{\G_{\bfv} \le \shp_{\bfv}-s\norm{\bfv}^{1/3}\} &\le \P\left\{\G_{\bfv}^{w, \ShpMin_{\bfv}} \le \shp_{\bfv}-\frac{1}{2}s\norm{\bfv}^{1/3}\right\} \\
&+ \P\left\{\G_{\bfv}^{w, \ShpMin_{\bfv}}-\G_{\bfv} \ge \frac{1}{2}s\norm{\bfv}^{1/3}\right\}. 
\end{split}
\end{align}
We bound the two terms on the right-hand side separately. 

Since $x = \curv_{\bfv} \cdot (\ShpMin_{\bfv}-w) = \epsilon_0 s^{1/2} \in (0, 2\epsilon_0 \norm{\bfv}^{1/3}]$, for sufficiently small $\epsilon_0$, one can apply Lemma \ref{LLppBdLT} 
to obtain  
\begin{align}
\label{E3}
\begin{split}
\P\left\{\G_{\bfv}^{w, \ShpMin_{\bfv}} \le \shp_{\bfv}-\frac{1}{2}s\norm{\bfv}^{1/3}\right\} &\le \exp\bigg\{-\frac{x^3}{3}-\frac{xs\norm{\bfv}^{1/3}}{2\curv_{\bfv}} + \frac{C_0x^4}{\norm{\bfv}^{1/3}}\bigg\} \\
&\le \exp\bigg\{-\frac{xs\norm{\bfv}^{1/3}}{2\curv_{\bfv}} \bigg\} \le \exp\{-c_0 s^{3/2}\}
\end{split}
\end{align}
for some positive constants $C_0 = C_0(\delta)$ and $c_0 = c_0(\delta)$. The last inequality in \eqref{E3} relies also on Lemma \ref{LMinEst}(c). 

On the other hand, by virtue of Theorem \ref{TIncRTUB}, Lemma \ref{LMinEst}(b) and the bound $s \le 2\norm{\bfv}^{2/3}$, 
\begin{align}
\label{E5}
\P\left\{\G_{\bfv}^{w, \ShpMin_{\bfv}}-\G_{\bfv} \ge \frac{1}{2}s\norm{\bfv}^{1/3}\right\} \le \exp\{-c_0 s^{3/2}\}
\end{align}
for sufficiently large $N_0$ and $s_0$, and after shrinking $\epsilon_0$ and $c_0$ if necessary. 

Combining \eqref{E4}, \eqref{E3} and \eqref{E5} and adjusting the constants $c_0$ and $s_0$ 
suitably completes the proof. 
\end{proof}

\subsection{Proof of Proposition \ref{PBlkRTUB}}

We next obtain Proposition \ref{PBlkRTUB} via another use of identity \eqref{E:Rains} similarly to the proof \cite[Theorem 2.2]{Emra_Janj_Sepp_20}, which covers moderate deviations only.  

\begin{proof}[Proof of Proposition \ref{PBlkRTUB}]
It follows from \cite[Theorem 2.2]{Emra_Janj_Sepp_20} and Lemma \ref{LMinEst}(c) that 
\begin{align}
\label{Eq-25}
\P\{\G_{\bfv} \ge \shp_{\bfv} + s \norm{\bfv}^{1/3}\} \le \exp\{-c_0s^{3/2}\} \quad \text{ for } s \in [0, \epsilon_0 \norm{\bfv}^{2/3}]
\end{align}
for some constants $c_0 = c_0(\delta) > 0$ and $\epsilon_0 = \epsilon_0(\delta) > 0$. Assume now that $s > \epsilon_0 \norm{\bfv}^{2/3}$. By Lemma \ref{LMinEst}(b), $\ShpMin_{\bfv} \in [a_0, 1-a_0]$ for some constant $a_0 = a_0(\delta) > 0$. Decrease $\epsilon_0$ if necessary to have $\epsilon_0 \le a_0$. Set $w = \ShpMin_{\bfv} + \epsilon_0/2 \in [a_0, 1-a_0/2]$ and $z = \ShpMin_{\bfv} - \epsilon_0/2 \in [a_0/2, 1-a_0]$. Using monotonicity, the exponential Markov inequality, identity \ref{E:Rains}, Lemmas \ref{LMgfEst} and \ref{LMinEst}(c), and the lower bound on $s$, one obtains that 
\begin{align}
\label{Eq-26}
\begin{split}
&\P\{\G_{\bfv} \ge \shp_{\bfv} + s \norm{\bfv}^{1/3}\} \le \P\{\G^{w, z}_{\bfv} \ge \shp_{\bfv} + s \norm{\bfv}\} \\
&\le \E[\exp\{\epsilon_0\G^{w, z}_{\bfv}\}] \exp\{-\epsilon_0\shp_{\bfv}-\epsilon_0s\norm{\bfv}^{1/3}\} = \exp\left\{\int_{z}^w \M^t_{\bfv} \dd t - \epsilon_0 \shp_{\bfv} - \epsilon_0 s \norm{\bfv}^{1/3}\right\} \\
&\le \exp\left\{\frac{\curv^3\epsilon_0^3}{12} + C_0 \epsilon_0^4\norm{\bfv}  - \epsilon_0 s \norm{\bfv}^{1/3}\right\} \le \exp\left\{C_0 \epsilon_0^3 \norm{\bfv} -\frac{1}{2}\epsilon_0^2\norm{\bfv} - \frac{1}{2}\epsilon_0 s \norm{\bfv}^{1/3}\right\} \\
&\le \exp\left\{-\frac{1}{2}\epsilon_0 s \norm{\bfv}^{1/3}\right\} 
\end{split}
\end{align}
for some constant $C_0 = C_0(\delta)$ provided that $\epsilon_0 = \epsilon_0(C_0)$ is made sufficiently small. Combining \eqref{Eq-25} and \eqref{Eq-26}, and renaming $\min \{c_0, \epsilon_0/2\}$ as $c_0$ complete the proof. 
\end{proof}

\subsection{Proof of Theorem \ref{T:Blk}}

With the tail bounds in Theorem \ref{T:BlkLT} and Proposition \ref{PBlkRTUB} in place, the proof of Theorem \ref{T:Blk} now boils down to an integration. 

\begin{proof}[Proof of Theorem \ref{T:Blk}]
Let $N_0 = N_0(\delta) > 0$ denote a constant to be taken sufficiently large below. Let $\bfv \in S_\delta \cap \bbZ_{\ge N_0}^2$ and $p \ge 1$. Applying Fubini's theorem and then a change-of-variables give
\begin{align}
\label{Eq-9}
\begin{split}
\E[|\G_{\bfv}-\shp_{\bfv}|^p] &= \int_{0}^\infty \P\{|\G_{\bfv}-\shp_{\bfv}|^p \ge x\} \dd x \\
&= p\norm{\bfv}^{p/3}\int_0^\infty \P\{|\G_{\bfv}-\shp_{\bfv}| \ge s\norm{\bfv}^{1/3}\}s^{p-1} \dd s. 
\end{split}
\end{align}  

Pick a constant $s_0 = s_0(\delta) > 0$ at least as large as the constant named $s_0$ in Theorem \ref{T:BlkLT}.  
It follows from Proposition \ref{PBlkRTUB} and Theorem \ref{T:BlkLT} that, for some constant $c_0 = c_0(\delta) > 0$, 
\begin{align}
\label{Eq-10}
\begin{split}
&\int_0^\infty \P\{|\G_{\bfv}-\shp_{\bfv}| \ge s\norm{\bfv}^{1/3}\}s^{p-1} \dd s \\
&\le \int_0^{s_0} s^{p-1} \dd s + 2 \int_{s_0}^\infty \exp\{-c_0 \min \{s^{3/2}, s \norm{\bfv}^{1/3}\}\} s^{p-1} \dd s \\
&= p^{-1}s_0^p + 2\int_{s_0}^{\norm{\bfv}^{2/3}}\exp\{-c_0 s^{3/2}\} s^{p-1}\dd s + 2\int_{\norm{\bfv}^{2/3}}^{\infty} \exp\{-c_0 s \norm{\bfv}^{1/3}\}s^{p-1}\dd s \\
&\le p^{-1}s_0^p + 2\int_{0}^{\infty}\exp\{-c_0 s^{3/2}\} s^{p-1}\dd s + 2\int_{0}^{\infty} \exp\{-c_0 s \norm{\bfv}^{1/3}\}s^{p-1}\dd s. 
\end{split}
\end{align}

We bound the last two integrals in \eqref{Eq-10} via changing the variables and then applying Lemma \ref{LIntBd}(a) as follows. For some absolute constant $A_0 > 0$, 
\begin{align}
\label{Eq-11}
\begin{split}
&\int_{0}^{\infty}\exp\{-c_0 s^{3/2}\} s^{p-1}\dd s = c_0^{-2p/3}\int_{0}^\infty \exp\{-t^{3/2}\}  t^{p-1} \dd t \le c_0^{-2p/3} A_0^p p^{2p/3}, \\
&\int_{0}^{\infty} \exp\{-c_0 s \norm{\bfv}^{1/3}\}s^{p-1}\dd s = c_0^{-p}\norm{\bfv}^{-p/3} \int_{0}^{\infty} \exp\{-t\}t^{p-1} \dd t \le c_0^{-p}A_0^p p^p \norm{\bfv}^{-p/3}. 
\end{split}
\end{align}

Putting together \eqref{Eq-9}, \eqref{Eq-10} and \eqref{Eq-11}, one ends up with 
\begin{align*}
\E[|\G_{\bfv}-\shp_{\bfv}|^p] &\le s_0^p\norm{\bfv}^{p/3} + 2pc_0^{-2p/3} A_0^p p^{2p/3} \norm{\bfv}^{p/3} + 2pc_0^{-p}A_0^p p^p \\
&\le C_0^p p^{2p/3} \max \{p, \norm{\bfv}\}^{p/3}  
\end{align*}
for some constant $C_0 = C_0(\delta) > 0$.  

In the last bound,  the restriction $\norm{\bfv} \ge N_0$ can be removed after suitably increasing $C_0$.  This is because if $\norm{\bfv} < N_0$ then 
\begin{align*}
&\E[|\G_{\bfv}-\shp_{\bfv}|^p] \le 2^p\E[\G_{N_0, N_0}^p] + 2^p \shp_{N_0, N_0}^p,  \quad \text{ and } \\
&\E[\G_{N_0, N_0}^p] \le \E\left[\left(\sum_{i=1}^{N_0} \sum_{j =1}^{N_0} \rw_{i, j}\right)^p\right] \le N_0^{2p} \E[\rw_{1, 1}^p] \le 2^{p} N_0^{2p} p! \E\left[\exp\left\{\frac{1}{2}\rw_{1, 1}\right\}\right] \\
&\le 2^{2p}N_0^{2p} p^p.  \qedhere
\end{align*}
\end{proof}

\subsection{Proof of Proposition \ref{P:BlkLT-HG}}

We next obtain Proposition \ref{P:BlkLT-HG} by invoking \cite[Theorem 3]{Gang_Hegd_23}. In verifying the assumptions of that result, the following lemma will be useful. 

\begin{lem}
\label{LIncRTLB}
Let $\delta > 0$. There exist positive constants $c_0 = c_0(\delta)$ and $N_0 = N_0(\delta)$ such that 
\begin{align*}
\E[\G_{\bfv}] \le \shp_{\bfv}-c_0 \norm{\bfv}^{1/3} \qquad \text{ for } \bfv \in S_\delta \cap \bbZ_{\ge N_0}^2. 
\end{align*}
\end{lem}
\begin{proof}
Let $C_0 = C_0(\delta)$, $s_0 = s_0(\delta)$ and $N_0 = N_0(\delta) \ge s_0^{3/2}$ denote constants for which the assertions of Theorem \ref{TIncRTLB} hold with $K = 0$. Using the theorem along with \eqref{EMean} and \eqref{EShpMean} gives
\begin{equation*}
\shp_{\bfv}-\E[\G_{\bfv}] = \E[\G^{\ShpMin_{\bfv}}_\bfv - \G_{\bfv}] \ge s_0 \norm{\bfv}^{1/3} \P\{\G^{\ShpMin_{\bfv}}_{\bfv}-\G_{\bfv} \ge s_0 \norm{\bfv}^{1/3}\} \ge s_0 \exp\{-C_0 s_0^{3/2}\}\norm{\bfv}^{1/3}
\end{equation*}
for $\bfv \in S_\delta \cap \bbZ_{\ge N_0}^2$. 
\end{proof}

\begin{proof}[Proof of Proposition \ref{P:BlkLT-HG}]
Assume that $\bfv = (n, n)$ for some $n \in \bbZ_{>0}$, and abbreviate $\bfv_k = (n-k, n+k)$ for $k \in [n-1]$. Formula \eqref{EShp} shows that $\shp_{\bfv_{k}}-\shp_{\bfv} = 2n \cdot (\sqrt{1-k^2n^{-2}}-1)$. Let $\epsilon \in (0, 1)$, and work with $n \ge N_0 = N_0(\epsilon)$ and $k \le (1-\epsilon) n$ below, taking $N_0 > 0$ sufficiently large to ensure that such $k \in [n-1]$ exists. In particular, $\bfv_{k} \in S_{\eta_0}$ for some constant $\eta_0 = \eta_0(\epsilon) > 0$. The Taylor expansion $\sqrt{1-x} = 1 - x/2 - x^{2}/8 + \cdots$ reveals that
$\shp_{\bfv_k}-\shp_{\bfv} \in [-k^2n^{-1}-A_0 k^4 n^{-3}, -k^2n^{-1}]$ for some constant $A_0 = A_0(\epsilon) > 0$ provided that $\epsilon \le \epsilon_0$ for some absolute constant $\epsilon_0 > 0$. On account of Theorem \ref{T:Blk} and Lemma \ref{LIncRTLB}, one also has $\E[\G_{\bfv_k}] \in [\shp_{\bfv_{k}}-B_0 n^{-1/3}, \shp_{\bfv_k}-b_0 n^{-1/3}]$ for sufficiently large $N_0$ and some constants $B_0 = B_0(\epsilon) > 0$ and $b_0 = b_0(\epsilon) > 0$. Combining this inclusion with the previous one yields  
\begin{align*}
-k^2n^{-1} -A_0k^4n^{-3}-B_0n^{-1/3} \le \E[\G_{\bfv_k}] - \shp_{\bfv} \le -k^2n^{-1} -b_0 n^{-1/3}, 
\end{align*}
which implies \cite[Assumption 2]{Gang_Hegd_23} upon setting $\epsilon = \epsilon_0$, for example. Resuming with an arbitrary $\epsilon \in (0, 1)$, next apply Proposition \ref{PBlkRTUB} and Theorem \ref{T:BlkLT} to obtain  
\begin{align}
\label{Eq-84}
\max \{\P\{\G_{\bfv_k} \ge \shp_{\bfv_k} + s n^{1/3}\}, \P\{\G_{\bfv_k} \le \shp_{\bfv_k}-sn^{1/3}\}\} \le \exp\{-c_0 \min \{s^{3/2}, sn^{1/3}\}\}
\end{align}
for $s \ge s_0$ and some constants $c_0 = c_0(\epsilon) > 0$ and $s_0 = s_0(\epsilon) > 0$, and sufficiently large $N_0$. The inequality $|\E[\G_{\bfv_k}]-\shp_{\bfv_k}| \le B_0n^{-1/3}$ allows us to replace the term $\shp_{\bfv_k}$ with the mean $\E[\G_{\bfv_k}]$ at the cost of adjusting $c_0$ and possibly increasing $s_0$. Hence, \cite[Assumption 3]{Gang_Hegd_23} is also verified (for any power $\alpha < 1$ there). Now \cite[Theorem 3]{Gang_Hegd_23} implies that 
\begin{align*}
\P\{\G_{\bfv} \le \shp_{\bfv}-sn^{1/3}\} \le \exp\{-c_1 s^3\}
\end{align*}
for $s \ge s_1$ and $n \ge N_1$ for some absolute positive constants $c_1$, $N_1$ and $s_1$. 
\end{proof}

\subsection{Proof of Proposition \ref{P:BlkCMLB}}
We conclude this section with a brief coupling proof adapted from \cite[Section 2.6]{More_Sepp_Valk_14}. 

\begin{proof}[Proof of Proposition \ref{P:BlkCMLB}]
Given $\delta > 0$, there exist positive constants $C_0 = C_0(\delta)$, $N_0 = N_0(\delta)$ and $s_0 = s_0(\delta)$ such that the following inequalities hold 
\begin{align}
\label{Eq-41}
\begin{split}
\E[|\G_{\bfv}-\shp_{\bfv}|] &\ge \E[(\G_{\bfv}-\shp_{\bfv})_-] \ge \E[(\G^{\ShpMin_{\bfv}}_{\bfv}-\shp_{\bfv})_-] = \E[(\G^{\ShpMin_{\bfv}}_{\bfv}-\shp_{\bfv})_+] \\ 
&\ge s_0 \norm{\bfv}^{1/3}\P\{\G^{\ShpMin_{\bfv}}_{\bfv} \ge \shp_{\bfv} + s_0 \norm{\bfv}^{1/3}\} \ge s_0 \exp\{-C_0s_0^{3/2}\}\norm{\bfv}^{1/3}
\end{split}
\end{align}
for $\bfv \in S_\delta \cap \bbZ_{\ge N_0}^2$. The second step in \eqref{Eq-41} uses that $x \mapsto (x)_-$ is nonincreasing and $\G \le \G^{\ShpMin_{\bfv}}$. The subsequent equality holds because $\E[\G^{\ShpMin_{\bfv}}_{\bfv}] = \M^{\ShpMin_{\bfv}}_{\bfv} = \shp_{\bfv}$ in view of \eqref{EMean} and \eqref{EShpMean}. The last inequality in \eqref{Eq-41} comes from a suitable right-tail lower bound such as \cite[(5.43)]{Sepp_18_CGM}, or Proposition \ref{PBdRTLB} of this article. Now the result is a consequence of \eqref{Eq-41} and Jensen's inequality.
\end{proof}

\section{Proofs of the upper bounds for the LPP with boundary}
\label{SPfBdUB}

We now set out to prove our central moment upper bounds for the LPP with boundary, treating the KPZ and Gaussian regimes separately. As noted around \eqref{ETrans}, it suffices to prove only part (a) in various bounds from this section onward. 

\subsection{Proofs of the upper bounds in the KPZ regime}

In preparation for the proof of Theorem \ref{TBdUB}, we record suitable tail bounds for the LPP with boundary. These bounds are fairly immediate now given the development in Section \ref{S:PfBlk}. 

\begin{prop}
\label{PBdRTUB}
Let $\delta > 0$. There exist positive constants $c_0 = c_0(\delta)$, $\epsilon_0 = \epsilon_0(\delta)$, $N_0 = N_0(\delta)$ and $s_0 =  s_0(\delta)$ such that the following statements hold for $\bfv \in S_\delta \cap \bbZ_{\ge N_0}^2$, $s \ge s_0$, $w > 0$ and $z < 1$. 
\begin{enumerate}[\normalfont (a)]
\item If $w \ge \ShpMin_{\bfv}-\epsilon_0 \min \{s^{1/2} \norm{\bfv}^{-1/3}, 1\}$ then 
\begin{align*}
\P\{\G^{w, \hor}_{\bfv} \ge \shp_{\bfv} + s \norm{\bfv}^{1/3}\} \le \exp\{-c_0 \min \{s^{3/2}, s\norm{\bfv}^{1/3}\}\}.  
\end{align*}
\item If $z \le \ShpMin_{\bfv} + \epsilon_0 \min \{s^{1/2} \norm{\bfv}^{-1/3}, 1\}$ then 
\begin{align*}
\P\{\G^{z, \ver}_{\bfv} \ge \shp_{\bfv} + s \norm{\bfv}^{1/3}\} \le \exp\{-c_0 \min \{s^{3/2}, s\norm{\bfv}^{1/3}\}\}. 
\end{align*}
\end{enumerate}
\end{prop}
\begin{proof}[Proof of {\rm(}a{\rm)}]
Let $\epsilon_0, N_0$ and $s_0$ denote positive constants depending only on $\delta$. Let $\bfv \in S_\delta \cap \bbZ_{\ge N_0}^2$ and $s \ge s_0$. Pick $w > 0$ such that $w \ge \ShpMin_{\bfv}-\epsilon_0 \min \{s^{1/2} \norm{\bfv}^{-1/3}, 1\}$. By virtue of Proposition \ref{PBlkRTUB} and Theorem \ref{TIncRTUB}(a), and a union bound, 
\begin{align*}
\P\{\G^{w, \hor}_{\bfv} \ge \shp_{\bfv} + s \norm{\bfv}^{1/3}\} &\le \P\{\G_{\bfv} \ge \shp_{\bfv} + \frac{1}{2}s \norm{\bfv}^{1/3}\} + \P\{\G^{w, \hor}_{\bfv}-\G_{\bfv} \ge \frac{1}{2}s\norm{\bfv}^{1/3}\} \\
&\le \exp\{-c_0 \min \{s^{3/2}, s \norm{\bfv}^{1/3}\}\}
\end{align*}
for some constant $c_0 = c_0(\delta) > 0$ provided that $\epsilon_0$ is sufficiently small, and $N_0$ and $s_0$ are sufficiently large. 
\end{proof}

\begin{prop}
\label{PLppBdLT}
Let $\delta > 0$. There exist positive constants $c_0 = c_0(\delta)$, $N_0 = N_0(\delta)$ and $s_0 = s_0(\delta)$ such that
\begin{align*}
\P\{\G^{w, \hor}_{\bfv} \le \shp_{\bfv} - s \norm{\bfv}^{1/3}\} &\le \exp\{-c_0 \min \{s^{3/2}, s\norm{\bfv}^{1/3}\}\}, \\
\P\{\G^{z, \ver}_{\bfv} \le \shp_{\bfv} - s \norm{\bfv}^{1/3}\} &\le \exp\{-c_0 \min \{s^{3/2}, s\norm{\bfv}^{1/3}\}\}
\end{align*}
for $\bfv \in S_\delta \cap \bbZ_{\ge N_0}^2$, $s \ge s_0$, and $w > 0$ and  $z < 1$. 
\end{prop}
\begin{proof}
These bounds are corollaries of Theorem \ref{T:BlkLT} because $\G_{\bfv} \le \min \{\G^{w, \hor}_{\bfv}, \G^{z, \ver}_{\bfv}\}$ for any $\bfv \in \bbZ_{>0}^2$, $w > 0$ and $z < 1$. 
\end{proof}

\begin{proof}[Proof of Theorem \ref{TBdUB}{\rm(}a{\rm)}]
Let $N_0 = N_0(\delta)$ and $s_0 = s_0(\delta)$ denote positive constants to be chosen sufficiently large. Let $\bfv \in S_\delta \cap \bbZ_{\ge N_0}^2$, $p \ge 1$ and $w > 0$ be such that $w \ge \ShpMin_{\bfv}-\epsilon_0 \min \{p^{1/3}\norm{\bfv}^{-1/3}, 1\}$ where $\epsilon_0 = \epsilon_0(\delta)$ denotes the constant from Proposition \ref{PBdRTUB}. 

Assume that $p \ge s_0^{3/2}$ for now. Proceeding as in the proof of Theorem \ref{T:Blk}, one develops  
\begin{align}
\label{Eq-12}
\begin{split}
&\E[|\G^{w, \hor}_{\bfv}-\shp_{\bfv}|^p] = p\norm{\bfv}^{p/3}\int_0^\infty \P\{|\G^{w, \hor}_{\bfv}-\shp_{\bfv}| \ge s\norm{\bfv}^{1/3}\}s^{p-1} \dd s \\
&\le p\norm{\bfv}^{p/3}\left(\int_0^{p^{2/3}} s^{p-1} \dd s + 2 \int_{p^{2/3}}^\infty \exp\{-c_0 \min \{s^{3/2}, s \norm{\bfv}^{1/3}\}\} s^{p-1} \dd s \right)\\
&\le \norm{\bfv}^{p/3} \bigg(p^{2p/3} + 2p\one_{\{p \le \norm{\bfv}\}}\int_{p^{2/3}}^{\norm{\bfv}^{2/3}}\exp\{-c_0 s^{3/2}\} s^{p-1}\dd s \\
&\qquad \qquad \qquad \ \ \ + 2p\int_{\norm{\bfv}^{2/3}}^{\infty} \exp\{-c_0 s \norm{\bfv}^{1/3}\}s^{p-1}\dd s \bigg)\\
&\le \norm{\bfv}^{p/3} \bigg(p^{2p/3} + 2p\int_{0}^{\infty}\exp\{-c_0 s^{3/2}\} s^{p-1}\dd s + 2p\int_{0}^{\infty} \exp\{-c_0 s \norm{\bfv}^{1/3}\}s^{p-1}\dd s \bigg)\\
&\le \norm{\bfv}^{p/3} \left(p^{2p/3} + 2pc_0^{-2p/3} B_0^p p^{2p/3} + 2pc_0^{-p}B_0^p p^p \norm{\bfv}^{-p/3}\right)\\
&\le C_0^p p^{2p/3} \max \{p, \norm{\bfv}\}^{p/3}  
\end{split}
\end{align}
for some absolute constant $B_0 > 0$ and constants $c_0 = c_0(\delta) > 0$ and $C_0 = C_0(\delta) > 0$. The second line in \eqref{Eq-12} comes from Propositions \ref{PBdRTUB} and \ref{PLppBdLT}. To be able to invoke these results, choose $N_0$ and $s_0$ sufficiently large, and observe that $w \ge \ShpMin_{\bfv}-\epsilon_0 \min \{s^{1/2}\norm{\bfv}^{-1/3}, 1\}$ when $s \ge p^{2/3}$. The third inequality in \eqref{Eq-12} is an application of Lemma \ref{LIntBd}(a). 

The remaining case $p \in [1, s_0^{3/2})$ follows from the case $p = s_0^{3/2}$ and Jensen's inequality.  Finally, the restriction $\norm{\bfv} \ge N_0$ can be lifted as in the proof of Theorem \ref{T:Blk}. 
\end{proof}

\subsection{Proofs of the upper bounds in the Gaussian regime}\label{SsGausUB}
Our proof of Theorem \ref{TGausUB} follows the basic idea outlined before the theorem that the LPP fluctuations come from the boundary weights in the Gaussian regime. To build an argument along these lines, we will rely on Theorem \ref{TBdUB} applied with a suitable shift. The next lemma records the properties of this shift (represented with $k$ and $l$ in the lemma) to be used in our proof. The main point in part (a) is that the boundary parameter $w$ is close to the shifted minimizer $\zeta_{(m-k, n)}$, which will permit using Theorem \ref{TBdUB} with the last-passage time $\G_{(m-k, n)}^{w, \hor}$. 

\begin{lem}
\label{LShiftProp}
Let $\delta > 0$ and $\epsilon > 0$. There exist positive constants $c_0 = c_0(\delta, \epsilon)$, $C_0 = C_0(\delta, \epsilon)$, $N_0 = N_0(\delta, \epsilon)$ and $\eta_0 = \eta_0(\delta, \epsilon)$ such that the following statements hold for $\bfv = (m, n) \in S_\delta \cap \bbZ_{\ge N_0}^2$. 
\begin{enumerate}[\normalfont (a)]
\item Assume that $\epsilon \le w \le \ShpMin_{\bfv}$. There exists a necessarily unique $k \in [m-1]$ such that  
\begin{enumerate}[\normalfont (\romannumeral1)]
\item $\ShpMin_{(m-k, n)} < w \le \ShpMin_{(m-k+1, n)}$
\item $(m-k, n) \in S_{\eta_0}$
\item $c_0 (k-i) \norm{\bfv}^{-1} \le \ShpMin_{(m-i, n)}-\ShpMin_{(m-k, n)} \le C_0 (k-i) \norm{\bfv}^{-1}$  for $i \in [k] \cup \{0\}$. 
\end{enumerate}
\item Assume that $\ShpMin_{\bfv} \le z \le 1-\epsilon$. There exists a necessarily unique $l \in [n-1]$ such that  
\begin{enumerate}[\normalfont (\romannumeral1)]
\item $\ShpMin_{(m, n-l+1)} < w \le \ShpMin_{(m, n-l)}$
\item $(m, n-l) \in S_{\eta_0}$
\item $c_0 (l-j) \norm{\bfv}^{-1} \le \ShpMin_{(m, n-l)}-\ShpMin_{(m, n-j)} \le C_0 (l-j) \norm{\bfv}^{-1}$  for $j \in [l] \cup \{0\}$. 
\end{enumerate}
\end{enumerate}
\end{lem}
\begin{proof}[Proof of {\rm(}a{\rm)}]
Let $\bfv = (m, n) \in S_\delta \cap \bbZ_{\ge N_0}^2$ for some constant $N_0 = N_0(\delta, \epsilon) > 0$ to be selected. To obtain (a), first note from \eqref{EMeanDer} that 
\begin{align*}
\ShpMin_{(1, n)} = (1-\ShpMin_{(1, n)}) n^{-1/2} \le N_0^{-1/2} < \epsilon \le w \le \ShpMin_{\bfv} = \ShpMin_{(m, n)}
\end{align*}
for sufficiently large $N_0$. Therefore, one can choose $k \in [m-1]$ minimal such that $\ShpMin_{(m-k, n)} < w$. Due to minimality, $\ShpMin_{(m-k+1, n)} \ge w \ge \epsilon$. In particular, (\romannumeral1) holds. Using \eqref{EMeanDer} once more yields 
\begin{align*}
m-k = \frac{\ShpMin_{(m-k+1, n)}^2 n}{(1-\ShpMin_{(m-k+1, n)})^2} -1 \ge \epsilon^2 n-1 \ge \frac{1}{2}\epsilon^2 n   
\end{align*}
provided that $N_0$ is sufficiently large. Therefore, $(m-k, n) \in S_{\eta_0}$ with $\eta_0 = \min \{\delta, \epsilon^2/2\}$, proving (\romannumeral2). Also, $k < m \le \delta^{-1}n \le \delta^{-1}(m-k+n)$, and $\delta (1+\delta)^{-1}\norm{\bfv} \le n \le m-k+n \le \norm{\bfv} = m+n$. Then, by Lemma \ref{LMinShift}(a) (applied with $\eta_0$ and $\delta^{-1}$ in place of $\delta$ and $K$), 
\begin{align*}
c_0 (k-i) \norm{\bfv}^{-1} \le \ShpMin_{(m-i, n)}-\ShpMin_{(m-k, n)} \le C_0 (k-i) \norm{\bfv}^{-1} 
\end{align*}
for $i \in [k] \cup \{0\}$ and some constants $c_0 = c_0(\delta, \epsilon) > 0$ and $C_0 = C_0(\delta, \epsilon) > 0$. Hence, (\romannumeral3). 
\end{proof}

\begin{proof}[Proof of Theorem \ref{TGausUB}]
Let $\bfv = (m, n) \in S_\delta \cap \bbZ_{\ge N_0}^2$ for some constant $N_0 = N_0(\delta, \epsilon) \ge 1$ to be chosen sufficiently large. Let $p \ge 1$, and assume that $\epsilon \le w \le \ShpMin_{\bfv}-\epsilon \min \{p^{1/3}\norm{\bfv}^{-1/3}, 1\}$. Due to Lemma \ref{LMinEst}(b), the preceding interval for $w$ is nonempty provided that $\epsilon \le a_0(\delta)$ for some sufficiently small constant $a_0(\delta) \in (0, 1]$. 

Pick $k \in [m-1]$ as in Lemma \ref{LShiftProp}(a). Writing $\wt{\bfv} = (m-k, n)$, the triangle inequality and the linearity of $\M^w$ yield
\begin{align}
\label{Eq-44}
\begin{split}
\E[|\G^{w, \hor}_{\bfv}-\M^w_{\bfv}|^p] &\le 3^p \cdot \E[|\G^{w}_{(k, 0)}-kw^{-1}|^p] \\
&+ 3^p \cdot \E[|\G^{w, \hor}_{\bfv}-\G^{w}_{(k, 0)}-\shp_{\wt{\bfv}}|^p] + 3^p \cdot (\M^w_{\wt{\bfv}}-\shp_{\wt{\bfv}})^p
\end{split}
\end{align}
We treat the terms on the right-hand side of \eqref{Eq-44} separately. 

Applying Lemmas \ref{LExpTUB} and \ref{LIntBd}(c)-(d), one obtains that 
\begin{align}
\label{Eq-43}
\begin{split}
&\E[|\G^{w, \hor}_{(k, 0)}-kw^{-1}|^p] = w^{-p}\E[|w\G^{w, \hor}_{(k, 0)}-k|^p]  \le \epsilon^{-p}\E[|w\G^{w, \hor}_{(k, 0)}-k|^p]\\
&= p \epsilon^{-p}\int_0^\infty \P\{|w\G^{w}_{(k, 0)}-k| \ge s\} s^{p-1} \dd s \le 2p \epsilon^{-p}\int_0^\infty \exp\left\{-b_0 s\min \left\{\frac{s}{k}, 1\right\}\right\} s^{p-1} \dd s 
\\
&\le 2p \epsilon^{-p}\int_0^\infty \exp\left\{-b_0s^2k^{-1}\right\}t^{p-1} \dd s + 2p\epsilon^{-p} \int_0^\infty \exp\left\{-b_0s\right\}s^{p-1}\dd s  \\
&\le 2p \epsilon^{-p} b_0^{-p/2}k^{p/2} \int_0^\infty \exp\{-t^2\} t^{p-1} \dd t  + 2p \epsilon^{-p}b_0^{-p}\int_0^\infty \exp\{-t\}t^{p-1} \dd t \\ 
&\le B_0^p \cdot (p^{p/2} k^{p/2} + p^p) \le B_0^p p^{p/2} \max \{p, k\}^{p/2} 
\end{split}
\end{align}
for some constants $b_0 > 0$ and $B_0 = B_0(\epsilon) > 0$. Let $R_0 = R_0(\delta, \epsilon)$ and $r_0 = r_0(\delta, \epsilon)$ denote the constants called $C_0(\delta, \epsilon)$ and $c_0(\delta, \epsilon)$, respectively, in Lemma \ref{LShiftProp}. By properties (\romannumeral1) and (\romannumeral3) in part (a) of the lemma,  
\begin{align}
\label{Eq-51}
\begin{split}
k &\le r_0^{-1}\norm{\bfv} (\ShpMin_{\bfv}-\ShpMin_{\wt{\bfv}}) \le r_0^{-1}\norm{\bfv} (\ShpMin_{\bfv}-w + R_0 \norm{\bfv}^{-1}) \le 2r_0^{-1} \norm{\bfv}(\ShpMin_{\bfv}-w). 
\end{split}
\end{align}
The last step above holds for sufficiently large $N_0 = N_0(R_0, \epsilon)$ because then $R_0 \norm{\bfv}^{-1} \le \epsilon \norm{\bfv}^{-1/3} \le \ShpMin_{\bfv}-w$ due to the assumption on $w$ and the inequality $N_0 \ge 1$. Using \eqref{Eq-51} in \eqref{Eq-43}, and working with $r_0 \le 1$ yield
\begin{align}
\label{Eq-52}
\begin{split}
\E[|\G^{w, \hor}_{(k, 0)}-kw^{-1}|^p] &\le B_0^p 2^p r_0^{-p} p^{p/2}\max \{p, (\ShpMin_{\bfv}-w)\norm{\bfv}\}^{p/2} \\
&\le C_0^p p^{p/2}\max \{p, (\ShpMin_{\bfv}-w)\norm{\bfv}\}^{p/2} 
\end{split}
\end{align}
for some constant $C_0 = C_0(\delta, \epsilon) > 0$. 

Next develop the bound  
\begin{align}
\label{Eq-45}
\begin{split}
&\E[|\G^{w, \hor}_{\bfv}-\G^{w}_{(k, 0)}-\shp_{\wt{\bfv}}|^p] = \E[(\G^{w, \hor}_{\bfv}-\G^{w}_{(k, 0)}-\shp_{\wt{\bfv}})_+^p] + \E[(\G^{w, \hor}_{\bfv}-\G^{w}_{(k, 0)}-\shp_{\wt{\bfv}})_-^p] \\
&\le \E[(\G^{w}_{\bfv}-\G^{w}_{(k, 0)}-\shp_{\wt{\bfv}})_+^p] + \E[(\G^{w}_{(k+1, 0), \bfv}-\shp_{\wt{\bfv}})_-^p] \\
&= \E[(\G^{w}_{\wt{\bfv}}-\shp_{\wt{\bfv}})_+^p] + \E[(\G^{w, \hor}_{\wt{\bfv}}-\shp_{\wt{\bfv}})_-^p] \le \E[|\G^{w}_{\wt{\bfv}}-\shp_{\wt{\bfv}}|^p] + \E[|\G^{w, \hor}_{\wt{\bfv}}-\shp_{\wt{\bfv}}|^p] \\
&\le C_0^p p^{2p/3} \max \{p, \norm{\bfv}\}^{p/3} 
\end{split}
\end{align}
after possibly increasing $C_0$. The first inequality in \eqref{Eq-45} comes from the monotonicity of the functions $(\cdot)_+$ and $(\cdot)_-$, and the inequalities $\G^{w, \hor}_{\bfv} \le \G^{w}_{\bfv}$ and $\G^{w, \hor}_{\bfv} \ge \G^{w}_{(k, 0)} + \G^{w}_{(k+1, 0), \bfv}$. To justify the subsequent step, observe from increment-stationarity \eqref{EIncSt} that $\G^{w}_{\bfv}-\G^{w}_{(k, 0)} \stackrel{\text{dist.}}{=} \G^{w}_{\wt{\bfv}}$, and from definitions \eqref{E:LPPwBd} and \eqref{E:Lpp1Bd} that $\G^{w}_{(k+1, 0), \bfv} \stackrel{\text{dist.}}{=} \G^{w, \hor}_{\wt{\bfv}}$. For the last line in \eqref{Eq-45}, apply Theorem \ref{TBdUB}(a) with the vertex $\wt{\bfv}$. The hypotheses of the theorem are in place owing to Lemma \ref{LShiftProp}(a): There exists a constant $\eta_0 = \eta_0(\delta, \epsilon) > 0$ such that $\wt{\bfv} \in S_{\eta_0}$, and $w > \ShpMin_{\wt{\bfv}}$.  

Turn to the last term in \eqref{Eq-44}. Since $\wt{\bfv} \in S_{\eta_0}$, Lemmas \ref{LMeanEst} and \ref{LMinEst}(c) imply the existence of a constant $A_0 = A_0(\delta, \epsilon) > 0$ such that 
\begin{align}
\label{Eq-47}
\begin{split}
(\M^w_{\wt{\bfv}}-\shp_{\wt{\bfv}})^p &\le A_0^p \{(w-\ShpMin_{\wt{\bfv}})^2\norm{\bfv} + |w-\ShpMin_{\wt{\bfv}}|^3\norm{\bfv}\}^p \\
&\le A_0^p \{(\ShpMin_{m-k+1, n}-\ShpMin_{m-k, n})^2\norm{\bfv} + (\ShpMin_{m-k+1, n}-\ShpMin_{m-k, n})^3\norm{\bfv}\}^p \\
&\le A_0^p \{R_0^2 \norm{\bfv}^{-1} + R_0^3 \norm{\bfv}^{-2}\}^p \le 2^p A_0^p R_0^{2p} \norm{\bfv}^{-p} \le 1
\end{split}
\end{align}
provided that $N_0 = N_0(A_0, R_0)$ is sufficiently large for the last inequality. 

Now gather the bounds \eqref{Eq-44}, \eqref{Eq-52}, \eqref{Eq-45} and \eqref{Eq-47} to obtain 
\begin{align}
\label{Eq-53}
\begin{split}
\E[|\G^{w, \hor}_{\bfv}-\M^{w}_{\bfv}|^p] &\le 3^pC_0^p \cdot \{p^{p/2} \max \{p, (\ShpMin_{\bfv}-w)\norm{\bfv}\}^{p/2} + p^{2p/3} \max \{p, \norm{\bfv}\}^{p/3} + 1\} \\
&\le 3^{p+1}C_0^p p^{p/2} (1+\epsilon^{-p/2}) \max \{p, (\ShpMin_{\bfv}-w)\norm{\bfv}\}^{p/2}
\end{split}
\end{align}
To arrive at the last inequality in \eqref{Eq-53}, observe that 
\begin{align}
\label{Eq-54}
\begin{split}
&p^{p/2} \max \{p, (\ShpMin_{\bfv}-w)\norm{\bfv}\}^{p/2} \ge p^{p/2} \max \{p, \epsilon \min \{p^{1/3}\norm{\bfv}^{-1/3}, 1\}\norm{\bfv}\}^{p/2} \\
&\ge \epsilon^{p/2} p^{p/2} \max \{p, \min \{p^{1/3}\norm{\bfv}^{-1/3}, 1\}\norm{\bfv}\}^{p/2} \\
&= \epsilon^{p/2} p^{p/2} \left(\one_{\{p \le \norm{\bfv}\}}p^{p/6}\norm{\bfv}^{p/3} + \one_{\{p > \norm{\bfv}\}}p^{p/2}\right) = \epsilon^{p/2} p^{2p/3} \max \{p, \norm{\bfv}\}^{p/3}
\end{split}
\end{align}
due to the assumptions that $w \le \ShpMin_{\bfv}-\epsilon \min \{p^{1/3}\norm{\bfv}^{-1/3}, 1\}$ and $\epsilon \le 1$. 

Finally,  the restriction $\norm {\bfv} \ge N_0$ can be removed as in the proof of Theorem \ref{T:Blk}. 
\end{proof}

\section{Proofs of the lower bounds for the LPP with boundary}
\label{SPfBdLB}

We now turn to proving our central moment lower bounds for the LPP with boundary. As in the situation with the upper bounds, we first treat the KPZ regime and then move on to the Gaussian regime. 

\subsection{Proof of the lower bounds in the KPZ regime}

Our main step towards the proof of Theorem \ref{TBdLB} is to develop the following right-tail lower bounds. The result strengthens a recent lower bound \cite[(5.43)]{Sepp_18_CGM}, which was also derived via the coupling approach. In our notation, the earlier bound was stated for the increment-stationary $\G^z$-process defined by \eqref{E:IncStLPP} and covered the case of fixed $s$. To prove Theorem \ref{TBdLB}, however, one needs to permit $s$ to grow with $\norm{\bfv}$. 

\begin{prop}
\label{PBdRTLB}
Let $\delta > 0$ and $K \ge 0$. There exist positive constants $C_0 = C_0(\delta, K)$, $\epsilon_0 = \epsilon_0(\delta, K)$, $N_0 = N_0(\delta, K)$ and $s_0 = s_0(\delta)$ such that the following statements hold for $\bfv \in S_\delta \cap \bbZ_{\ge N_0}^2$, $s \in [s_0, \epsilon_0\norm{\bfv}^{2/3}]$, and $w > 0$ and  $z < 1$. 
\begin{enumerate}[\normalfont (a)]
\item If $w \le \ShpMin_{\bfv} + K s^{1/2}\norm{\bfv}^{-1/3}$ then 
$
\P\{\G^{w}_{(1, 0), \bfv} \ge \shp_{\bfv} + s \norm{\bfv}^{1/3}\} \ge \exp\{-C_0 s^{3/2}\}.  
$
\item If $z \ge \ShpMin_{\bfv} - K s^{1/2}\norm{\bfv}^{-1/3}$ then 
$
\P\{\G^{z}_{(0, 1), \bfv} \ge \shp_{\bfv} + s \norm{\bfv}^{1/3}\} \ge \exp\{-C_0 s^{3/2}\}.
$
\end{enumerate}
\end{prop}
\begin{rem}
\label{RConst}
Our proof shows that one can choose $C_0 = C_1 \max \{K, 1\}^2 $ for some constant $C_1 = C_1(\delta) > 0$. This fact becomes important in the proof of Proposition \ref{PBdMomLB} ahead. Similarly to Theorem \ref{TBdLB}, the restrictions in (a) and (b) are purely technical. Due to the inequality $\G_{\bfv} \le \min \{\G^{w, \hor}_{\bfv}, \G^{z, \ver}_{\bfv}\}$ for $\bfv \in \bbZ_{>0}^2$, the bounds in Proposition \ref{PBdRTLB} in fact hold for all $w > 0$ and $z < 1$ with the constant $C_0 = C_0(\delta) > 0$ chosen as in the lower bound \eqref{EBlkRTLB}.  
\end{rem}

Our aim now is to prove Proposition \ref{PBdRTLB} through a refinement of the change-of-measure argument proving \cite[(5.43)]{Sepp_18_CGM}, which in turn adapts the lower bound proof for the current variance in the asymmetric simple exclusion process (ASEP) \cite{Bala_Sepp_10}. A crucial component in our proof is the following difference inequality for the variance of the increment-stationary LPP. This result, in a more precise form, was observed in \cite[Lemma 4.7]{Bala_Cato_Sepp_06}. The version below is essentially the one from the lecture notes \cite{Sepp_18_CGM}. The proofs in \cite{Bala_Cato_Sepp_06, Sepp_18_CGM} employ the coupling method as in here but they rely on the Burke property \cite[Lemma 4.2]{Bala_Cato_Sepp_06}, specifically the independence of the LPP increments along any down-right path in the increment-stationary case. One can also obtain the lemma using the stationarity property \eqref{EIncSt} by first rederiving the variance identity \cite[Theorem 5.6]{Sepp_18_CGM} from \eqref{E:Rains} as indicated in \cite[Remark 2.4]{Emra_Janj_Sepp_23}. 

\begin{lem}[Lemma 5.7 in \cite{Sepp_18_CGM}]
\label{LVarLip}
Let $\delta > 0$. There exists a constant $C_0 = C_0(\delta) > 0$ such that 
\begin{align*}
\Var [\G_{\bfv}^{w}] \le \Var[\G_{\bfv}^z] + C_0 \norm{\bfv} (z-w)
\end{align*}
for $\bfv \in \bbZ_{\ge 0}^2$ and $w, z \in [\delta, 1-\delta]$ with $w \le z$. 
\end{lem}

\begin{proof}[Proof of Proposition \ref{PBdRTLB}{\rm(}a{\rm)}]
Let $\epsilon_0 = \epsilon_0(\delta, K)$, $N_0 = N_0(\delta, K)$ and $s_0 = s_0(\delta)$ denote positive constants to be chosen within the proof. Pick $\bfv \in S_\delta \cap \bbZ_{\ge N_0}^2$ and $s \in [s_0, \epsilon_0 \norm{\bfv}^{2/3}]$, taking $N_0 = N_0(\epsilon_0, s_0)$ sufficiently large to ensure that such $s$ exists. Due to monotonicity, it suffices to verify (a) for the case $w = \ShpMin_{\bfv}+Ks^{1/2}\norm{\bfv}^{-1/3}$. 

Introduce $k = \lf R_0 s^{1/2} \norm{\bfv}^{2/3}\rf$ and $z = \ShpMin_{\bfv}-r_0s^{1/2}\norm{\bfv}^{-1/3}$ where $R_0 = R_0(\delta) > 0$ and $r_0 > 0$ are constants to be specified. Because $\ShpMin_{\bfv} \in [a_0, 1-a_0]$ for some constant $a_0 = a_0(\delta) > 0$ due to Lemma \ref{LMinEst}(b), the first and last inequalities in $a_0/2 \le \ShpMin_{\bfv}- r_0 \epsilon_0^{1/2} \le z \le \ShpMin_{\bfv} \le w = \ShpMin_{\bfv} + K \epsilon_0^{1/2} \le 1-a_0/2$ hold upon choosing $\epsilon_0 = \epsilon_0(a_0, K, r_0)$ sufficiently small. In particular, the use of $z > 0$ below as the rate on horizontal axis is legitimate. 

With the preliminaries in order, we begin to develop a probability lower bound for the event $\{\G^{w, \hor}_{\bfv} \in I_{\bfv, s}\}$ where $I_{\bfv, s}$ denotes the interval $[\shp_{\bfv}+s\norm{\bfv}^{1/3}, \infty)$. In the next display, the first inequality is immediate from definition \eqref{E:Lpp1Bd}. (Recall here that if $i > m$ then the corresponding term inside the maximum is $-\infty$). For the second step, note that $\G^{w}_{(k, 0)}$ is a sum of $k$ independent $\Exp(w)$-distributed weights, and choose $\epsilon_0 = \epsilon_0(a_0, K, r_0)$ sufficiently small so that the middle inequality in $w-z = (K+r_0)s^{1/2}\norm{\bfv}^{-1/3} \le (K+r_0)\epsilon_0^{1/2} \le a_0/4 \le w/2$ is in force. Subsequently, \eqref{Eq-2} applies the Cauchy-Schwarz inequality. The final step changes the rate of the exponentials on the vertices $\{(i, 0): i \in [k]\}$ from $w$ to $z$. The Radon-Nikodym derivative of this change-of-measure is given by $(w/z)^k \exp\{(z-w)\G^z_{(k, 0)}\}$. 
\begin{align}
\label{Eq-2}
\begin{split}
&\P\{\G^{w, \hor}_{\bfv} \in I_{\bfv, s}\} \ge \P\left\{\max_{i \in [k]} \{\G_{(i, 0)}^{w} + \G_{(i, 1), \bfv} \} \in I_{\bfv, s}\right\}\\
&= \P\left\{\max_{i \in [k]} \{\G_{(i, 0)}^{w} + \G_{(i, 1), \bfv} \} \in I_{\bfv, s}\right\}\E[\exp\{2(w-z)\G^{w}_{(k, 0)}\}] \left(\frac{w}{w-2(w-z)}\right)^{-k} \\
&\ge \E\left[\one\left\{\max_{i \in [k]} \{\G_{(i, 0)}^{w} + \G_{(i, 1), \bfv}\} \in I_{\bfv, s}\right\} \cdot \exp\{(w-z) \G^{w}_{(k, 0)}\}\right]^2 \left(\frac{w}{w-2(w-z)}\right)^{-k} \\
&= \P\left\{\max_{i \in [k]} \{\G_{(i, 0)}^{z} + \G_{(i, 1), \bfv}\} \in I_{\bfv, s}\right\}^2 \left(\frac{w}{z}\right)^{2k} \left(\frac{w}{w-2(w-z)}\right)^{-k}. 
\end{split}
\end{align}

The logarithm of the last two terms in \eqref{Eq-2} can be bounded from below as follows.  
\begin{align}
\label{Eq-3}
\begin{split}
&-2k \log \left(1 - \frac{w-z}{w}\right) + k \log \left(1-\frac{2(w-z)}{w}\right) 
\ge -\frac{k(w-z)^2}{w^2} - \frac{10 k (w-z)^3}{w^3} \\
&\ge -R_0 (K+r_0)^2a_0^{-2}s^{3/2}-10 R_0 (K+r_0)^3a_0^{-3}s^2\norm{\bfv}^{-1/3} \ge - 2R_0(K+r_0)^2 a_0^{-2}s^{3/2}. 
\end{split}
\end{align}
The first bound in \eqref{Eq-3} uses the estimate $|\log (1-t)+t+t^2/2|$ $\le |t|^3$ for $t \in [-1/2, 1/2]$, and that $0 \le (w-z)w^{-1} \le (K+r_0)s^{1/2} \norm{\bfv}^{-1/3} a_0^{-1} \le (K+r_0)\epsilon_0^{1/2} a_0^{-1} \le 1/4$ where the last inequality has already been assumed. The preceding bound on $(w-z)w^{-1}$ along with the bound $k \le R_0 s^{1/2} \norm{\bfv}^{2/3}$ also justifies the second inequality in \eqref{Eq-3}. For the last step, choose $\epsilon_0 = \epsilon_0(a_0, K, r_0)$ sufficiently small.  

It remains to bound the final probability in \eqref{Eq-2} from below. By definition \eqref{EExitPt} of the exit points, and a union bound,  
\begin{align}
\label{Eq-1}
\begin{split}
&\P\left\{\max_{i \in [k]} \{\G_{(i, 0)}^{z} + \G_{(i, 1), \bfv} \} \ge \shp_{\bfv} + s\norm{\bfv}^{1/3}\right\} \ge \P\left\{\G_{\bfv}^z \ge \shp_{\bfv} + s\norm{\bfv}^{1/3}, \Z^{z, \hor}_{\bfv} \in [k]\right\} \\
&\ge \P\left\{\G_{\bfv}^z \ge \shp_{\bfv} + s\norm{\bfv}^{1/3}\right\} - \P\{\Z^{z, \ver}_{\bfv} > 0\} - \P\{\Z^{z, \hor}_{\bfv} > k\}. 
\end{split}
\end{align}
Applying Lemma \ref{LExit2}(b) and recalling the choice of $z \in (0, 1)$, one obtains that 
\begin{align}
\label{Eq-4}
\P\{\Z^{z, \ver}_{\bfv} > 0\} \le \exp\{-c_0 \norm{\bfv} (\ShpMin_{\bfv}-z)^3\} = \exp\{-c_0 r_0^3 s^{3/2}\} \le \exp\{-c_0 r_0^3 s_0^{3/2}\}
\end{align}
for some constant $c_0 = c_0(\delta) > 0$. Next choose $N_0 = N_0(\wt{N}_0, R_0, s_0) \ge \wt{N}_0$ and $R_0 = R_0(\wt{\epsilon}_0, r_0)$ sufficiently large such that $R_0 s_0^{1/2} \ge N_0^{-2/3}$ and $R_0 \wt{\epsilon}_0 \ge r_0$ where $\wt{\epsilon}_0 = \wt{\epsilon}_0(\delta)$ and $\wt{N}_0 = \wt{N}_0(\delta)$ stand for the constants named $\epsilon_0$ and $N_0$, respectively, in Lemma \ref{LExit}. Then it follows from Lemma \ref{LExit}(a) and the choices of $k$ and $z$ that 
\begin{align}
\label{Eq-5}
\P\{\Z^{z, \hor}_{\bfv} > k\} = \P\{\Z^{z, \hor}_{\bfv} > R_0 s^{1/2}\norm{\bfv}^{2/3}\} \le \exp\{-c_0R_0^3 s^{3/2}\} \le \exp\{-c_0 R_0^3 s_0^{3/2}\}
\end{align}
after reducing $c_0$ if necessary. (The equality in \eqref{Eq-5} is because $\Z^{z, \hor}_{\bfv}$ is integer-valued). Finally, observe from Lemmas \ref{LMeanEst} and \ref{LMinEst}(c) that, for some constant $B_0 = B_0(\delta) > 0$,  
\begin{align}
\label{Eq-6}
\begin{split}
\shp_{\bfv} + s\norm{\bfv}^{1/3} &\le \M^z_{\bfv} - \norm{\bfv}(z-\ShpMin_{\bfv})^2 + B_0 \norm{\bfv}(\ShpMin_{\bfv}-z)^3 + s \norm{\bfv}^{1/3} \\
&= \M^z_{\bfv}-(r_0^2-1)s\norm{\bfv}^{1/3} + B_0 r_0^3 s^{3/2} \le \M^z_{\bfv}-(r_0^2-2B_0r_0^3\epsilon_0^{1/2})s\norm{\bfv}^{1/3} \\
&\le \M^z_{\bfv}-s\norm{\bfv}^{1/3}
\end{split}
\end{align}
provided that $r_0 \ge 2$, and $\epsilon_0 = \epsilon_0(B_0, r_0)$ is sufficiently small. Using \eqref{Eq-6}, Chebyshev's inequality, Lemma \ref{LVarLip}, Theorem \ref{TBdUB} (recalling that $\shp_\bfv = \M_\bfv^{\zeta_\bfv}$ by \eqref{EShpMean}) and the fact that $a_0/2 \le z \le \ShpMin_{\bfv} \le 1-a_0$ then yield 
\begin{align}
\label{Eq-7}
\begin{split}
\P\left\{\G_{\bfv}^z \ge \shp_{\bfv} + s\norm{\bfv}^{1/3}\right\} &= 1 - \P\left\{\G_{\bfv}^z < \shp_{\bfv} + s\norm{\bfv}^{1/3}\right\} \\
&\ge 1- \P\left\{\G_{\bfv}^z < \M^z_{\bfv} - s\norm{\bfv}^{1/3}\right\} \ge 1 - \frac{\Var[\G_{\bfv}^z]}{s^2\norm{\bfv}^{2/3}} \\
&\ge 1- \frac{\Var[\G_{\bfv}^{\ShpMin_{\bfv}}]}{s^2 \norm{\bfv}^{2/3}}- \frac{D_0(\ShpMin_{\bfv}-z)\norm{\bfv}}{s^2 \norm{\bfv}^{2/3}} \ge 1-\frac{D_0}{s^2}-\frac{D_0 r_0}{s^{3/2}} \ge \frac{1}{2}
\end{split}
\end{align}
for some constant $D_0 = D_0(\delta) > 0$ after taking $s_0 = s_0(D_0, r_0)$ sufficiently large. Combining \eqref{Eq-1}, \eqref{Eq-4}, \eqref{Eq-5} and \eqref{Eq-7} gives  
\begin{align}
\label{Eq-8}
\begin{split}
\P\left\{\max_{i \in [k]} \{\G_{(i, 0)}^{z} + \G_{(i, 1), \bfv} \} \ge \shp_{\bfv} + s\norm{\bfv}^{1/3}\right\} &\ge \frac{1}{2}-\exp\{-c_0 s^{3/2}\} - \exp\{-c_0 R_0^3 s^{3/2}\} \\
&\ge \frac{1}{4}
\end{split}
\end{align}
after increasing $s_0 = s_0(c_0)$ and $R_0 = R_0(c_0)$ if necessary. 

Returning to \eqref{Eq-2} and appealing to \eqref{Eq-3} and \eqref{Eq-8}, one concludes that 
\begin{align*}
\begin{split}
\P\{\G^{w, \hor}_{\bfv} \in I_{\bfv, s}\} \ge \frac{1}{4}\exp\{-2R_0(K+r_0)^2 a_0^{-2} s^{3/2}\} \ge \exp\{-C_0s^{3/2}\}, 
\end{split}
\end{align*}
for some constant $C_0 = C_0(\delta, K) > 0$. 
%
\end{proof}

With the aid of Proposition \ref{PBdRTLB}, we next establish our central moment lower bounds in the KPZ regime. In fact, we obtain the following stronger proposition than Theorem \ref{TBdLB}. 

\begin{prop}
\label{PBdMomLB}
Let $\delta > 0$, $K \ge 0$ and $T > 0$. There exist positive constants $c_0 = c_0(\delta, K, T)$ and $N_0 = N_0(\delta, K, T)$ such that the following statements hold for $\bfv \in S_\delta \cap \bbZ_{\ge N_0}^2$, $1 \le p \le T\norm{\bfv}$, $w > 0$ and $z < 1$. 
\begin{enumerate}[\normalfont (a)]
\item If $w \le \ShpMin_{\bfv} + K\min \{p^{1/3}\norm{\bfv}^{-1/3}, 1\}$ then $\E[(\G_\bfv^{w, \hor}-\shp_\bfv)_+^p] \ge c_0^p p^{2p/3} \norm{\bfv}^{p/3}.$ 
\item If $z \ge \ShpMin_{\bfv}-K\min \{p^{1/3}\norm{\bfv}^{-1/3}, 1\}$ then $\E[(\G_\bfv^{z, \ver}-\shp_\bfv)_+^p] \ge c_0^p p^{2p/3} \norm{\bfv}^{p/3}.$ 
\end{enumerate}
\end{prop}

\begin{proof}[Proof of Proposition \ref{PBdMomLB}{\rm(}a{\rm)}]
Let $N_0 = N_0(\delta, K, T)$, $p_0 = p_0(\delta, K)$ and $r_0 = r_0(\delta, K)$ denote positive constants to be chosen suitably below. Let $\bfv \in S_\delta \cap \bbZ_{\ge N_0}^2$ and $p \in [1, T\norm{\bfv}]$. By monotonicity, it suffices to prove (a) for the extreme case $w = \ShpMin_{\bfv} + K\min \{p^{1/3}\norm{\bfv}^{-1/3}, 1\}$. 

Consider first the situation $p \in [p_0, r_0 \norm{\bfv}]$, taking $N_0 = N_0(p_0, r_0)$ sufficiently large to avoid a vacuous case. Let $L_0 = L_0(\delta, K) \ge K$ denote a constant to be chosen sufficiently large. Pick constants $\epsilon_0 = \epsilon_0(\delta, L_0) > 0$ and $s_0 = s_0(\delta) > 0$ as in Proposition \ref{PBdRTLB}. Then a computation similar to \eqref{Eq-9} leads to  
\begin{align}
\label{Eq-30}
\begin{split}
\E[(\G_{\bfv}^{w, \hor}-\shp_{\bfv})_+^p] &= p\norm{\bfv}^{p/3} \int_0^\infty \P\{\G_{\bfv}^{w, \hor} \ge \shp_{\bfv} + s\norm{\bfv}^{1/3}\}s^{p-1} \dd s \\
&\ge p \norm{\bfv}^{p/3} \int_{K^2 L_0^{-2}p^{2/3}}^{\epsilon_0 \norm{\bfv}^{2/3}}\P\{\G_{\bfv}^{w, \hor} \ge \shp_{\bfv} + s\norm{\bfv}^{1/3}\}s^{p-1} \dd s  \\
&\ge p \norm{\bfv}^{p/3} \int_{K^2L_0^{-2}p^{2/3}}^{\epsilon_0 \norm{\bfv}^{2/3}} \exp\{-L_0^2 C_0 s^{3/2}\} s^{p-1} \dd s \\
&= pL_0^{-4p/3}C_0^{-2p/3}\norm{\bfv}^{p/3} \int_{K^2L_0^{-2/3}C_0^{2/3}p^{2/3}}^{L_0^{4/3}C_0^{2/3}\epsilon_0\norm{\bfv}^{2/3}} \exp\{-t^{3/2}\} t^{p-1} \dd t
\end{split}
\end{align}
for some constant $C_0 = C_0(\delta) > 0$. In the second line of \eqref{Eq-30}, the lower limit of integration does not exceed the upper limit provided that $r_0 = r_0(\epsilon_0)$ is sufficiently small because then $K^2 L_0^{-2}p^{2/3} \le p^{2/3} \le r_0^{2/3} \norm{\bfv}^{2/3} \le \epsilon_0\norm{\bfv}^{2/3}$. For the third line of \eqref{Eq-30}, invoke Proposition \ref{PBdRTLB} observing that $w = \ShpMin_{\bfv} + Kp^{1/3}\norm{\bfv}^{-1/3} \le \ShpMin_{\bfv} + L_0 s^{1/2}\norm{\bfv}^{-1/3}$ for $s \ge K^2 L_0^{-2}p^{2/3}$ with $r_0 \le 1$, and that $K^2 L_0^{-2}p^{2/3} \ge K^2 L_0^{-2}p_0^{2/3} \ge s_0$ with $p_0 = p_0(K, L_0, s_0)$ chosen sufficiently large. 

Next bound from below the last integral in \eqref{Eq-30} as follows. By virtue of Lemma \ref{LIntBd}(b)-(c), there exist absolute positive constants $a_0$, $B_0$ and $b_0$ such that 
\begin{align}
\label{Eq-31}
\begin{split}
&\int_0^\infty \exp\{-t^{3/2}\} t^{p-1} \dd t \ge b_0^p p^{2p/3}, \quad \text{ and }\\
&\int_{0}^{K^2 L_0^{-2/3}C_0^{2/3}p^{2/3}} \exp\{-t^{3/2}\} t^{p-1} \dd t + \int_{L_0^{4/3}C_0^{2/3}\epsilon_0\norm{\bfv}^{2/3}}^{\infty} \exp\{-t^{3/2}\} t^{p-1} \dd t \\
&\le K^{2p}L_0^{-2p/3}C_0^{2p/3}p^{2p/3} \\
&+ p^{2p/3} B_0^p\exp\left\{-2a_0(L_0^2C_0 \epsilon_0^{3/2}\norm{\bfv}-\lc 2p/3 \rc)_+\min \left\{\frac{L_0^2C_0\epsilon_0^{3/2}\norm{\bfv}}{\lc 2p/3\rc}-1, 1\right\}\right\}\\
&\le \frac{1}{4}b_0^p p^{2p/3} + p^{2p/3} B_0^p \exp\left\{-a_0L_0^2C_0\epsilon_0^{3/2}\norm{\bfv}\right\} \le \frac{1}{2}b_0^p p^{2p/3}. 
\end{split}
\end{align}
For the last two inequalities in \eqref{Eq-31}, choose $L_0 = L_0(b_0, C_0, K)$ sufficiently large and then $r_0 = r_0(B_0, b_0, C_0, \epsilon_0, L_0)$ sufficiently small recalling that $\norm{\bfv} \ge r_0^{-1}p$ in the present case. 
Combining \eqref{Eq-30} and \eqref{Eq-31} yields 
\begin{align}
\label{Eq-32}
\E[(\G^{w, \hor}_{\bfv}-\shp_{\bfv})_+^p] \ge \frac{1}{2}pb_0^p L_0^{-4p/3}C_0^{-2p/3}p^{2p/3}\norm{\bfv}^{p/3} \ge c_0^p p^{2p/3}\norm{\bfv}^{p/3}
\end{align}
for some constant $c_0 = c_0(\delta, K) > 0$. 

Consider the case $p < p_0$ now. Another application of Proposition \ref{PBdRTLB} yields  
\begin{align}
\label{Eq-35}
\begin{split}
\E[(\G^{w, \hor}_{\bfv}-\shp_{\bfv})_+^p] &\ge s_0^p \norm{\bfv}^{p/3}\P\{\G_{\bfv}^{w, \hor} \ge \shp_{\bfv} + s_0 \norm{\bfv}^{1/3}\} \\
&\ge s_0^p \exp\{-L_0^2 C_0 s_0^{3/2}\}\norm{\bfv}^{p/3} \ge c_0^p p_0^{2p_0/3}\norm{\bfv}^{p/3} \ge c_0^p p^{2p/3} \norm{\bfv}^{p/3}
\end{split}
\end{align}
after reducing $c_0$ further if necessary.

Finally, turn to the case $p \in (r_0 \norm{\bfv}, T\norm{\bfv}]$. Then Jensen's inequality and the lower bound \eqref{Eq-32} give 
\begin{align}
\label{Eq-33}
\begin{split}
\E[(\G^{w, \hor}_{\bfv}-\shp_{\bfv})_+^p] &\ge \left\{\E[(\G^{w, \hor}_{\bfv}-\shp_{\bfv})_+^{r_0\norm{\bfv}}]\right\}^{pr_0^{-1}\norm{\bfv}^{-1}} \ge c_0^{p} (r_0 \norm{\bfv})^{2p/3} \norm{\bfv}^{p/3} \\
&\ge c_0^p r_0^{2p/3} T^{-2p/3} p^{2p/3}\norm{\bfv}^{p/3} \ge d_0^p p^{2p/3} \norm{\bfv}^{p/3}
\end{split}
\end{align}
for some constant $d_0 = d_0(\delta, \epsilon, K) > 0$. 

Putting together \eqref{Eq-32}, \eqref{Eq-35} and \eqref{Eq-33} completes the proof (a). 
\end{proof}

\begin{proof}[Proof of Theorem \ref{TBdLB}]
This is now immediate from Proposition \ref{PBdMomLB}. 
\end{proof}

\subsection{Proof of the lower bounds in the Gaussian regime}

Our derivation of central moment lower bounds in the Gaussian regime reuses ideas from the proofs of Theorems \ref{TGausUB} and \ref{TBdLB}. The key point still is that in this regime the main contribution to the fluctuations of the $\G^{w, \hor}$-process comes from $\G^{w}_{(k, 0)}$ for a suitably chosen $k$. 

\begin{proof}[Proof of Theorem \ref{TGausLB}]
Let $K_0 = K_0(\delta, \epsilon, T) > 0$ and $N_0 = N_0(\delta, \epsilon, T) > 0$ denote constants to be chosen large enough. Let $\bfv = (m, n) \in S_\delta \cap \bbZ_{\ge N_0}^2$ and $p \in [1, T\norm{\bfv}]$. Let $\epsilon \le w \le\ShpMin_{\bfv}-\min \{K_0 p^{1/3}\norm{\bfv}^{-1/3}, \epsilon\}$. Working with $\epsilon \le a_0(\delta)$ for some sufficiently small $a_0 = a_0(\delta) > 0$ ensures that the preceding choice of $w$ is possible on account of Lemma \ref{LMinEst}(b). 

After increasing $N_0$ if necessary, choose $k \in [m-1]$ as in Lemma \ref{LShiftProp}(a): (\romannumeral1) $\ShpMin_{m-k, n} < w \le \ShpMin_{m-k+1, n}$, (\romannumeral2) $\wt{\bfv} = (m-k, n) \in S_{\eta_0}$ for some constant $\eta_0 = \eta_0(\delta, \epsilon) > 0$, and (\romannumeral3) $r_0 (k-i) \norm{\bfv}^{-1}$ $\le \ShpMin_{m-i, n}-\ShpMin_{m-k, n} \le$ $R_0 (k-i) \norm{\bfv}^{-1}$ for $i \in [k] \cup \{0\}$ and some constants $r_0 = r_0(\delta, \epsilon) > 0$ and $R_0 = R_0(\delta, \epsilon) > 0$. It follows from (\romannumeral1)-(\romannumeral3) and the choice of $w$ that 
\begin{align}
k &\ge (\ShpMin_{m, n}-\ShpMin_{m-k, n}) \cdot R_0^{-1}\norm{\bfv} \ge (\ShpMin_\bfv-w) \cdot R_0^{-1}\norm{\bfv} \label{Eq-76}\\
&\ge R_0^{-1} \min \{K_0 p^{1/3}\norm{\bfv}^{2/3}, \epsilon \norm{\bfv}\}. \label{Eq-76b}
\end{align} 

By the triangle inequality and the linearity of $\M^w$, 
\begin{align}
\label{Eq-55}
\begin{split}
\E[|\G_{\bfv}^{w, \hor}-\M_{\bfv}^w|^p] &\ge 3^{-p}\E[|\G^{w}_{(k, 0)}-kw^{-1}|^p] \\
&-\E[|\G^{w, \hor}_{\bfv}-\G^w_{(k, 0)}-\shp_{\wt{\bfv}}|^p]-(\M^w_{\wt{\bfv}}-\shp_{\wt{\bfv}})^p. 
\end{split}
\end{align}
The following bound for the last two terms in \eqref{Eq-55} is immediate from the bounds \eqref{Eq-45} and \eqref{Eq-47} obtained in the proof of Theorem \ref{TGausUB}, and the assumption that $p \le T\norm{\bfv}$
\begin{align}
\label{Eq-57}
\E[|\G^{w, \hor}_{\bfv}-\G^w_{(k, 0)}-\shp_{\wt{\bfv}}|^p]+(\M^w_{\wt{\bfv}}-\shp_{\wt{\bfv}})^p \le C_0^p p^{2p/3}\norm{\bfv}^{p/3}. 
\end{align} 
for some constant $C_0 = C_0(\delta, \epsilon, T) > 0$. 

Let $T_0 = T_0(\delta, \epsilon, T) > 1$ denote a constant to be specified later in the proof. Let $s_0 = s_0(T_0) > 0$ and $A_0 = A_0(T_0) > 0$ refer to the constants in Lemma \ref{LExpTLB}. Note that $A_0$ can be chosen as increasing in $T_0$. Now bound from below the first term on the right-hand side of \eqref{Eq-55} as follows. 
\begin{align}
\label{Eq-56}
\begin{split}
&\E[|\G^{w}_{(k, 0)}-kw^{-1}|^p] \ge \E[(\G^w_{(k, 0)}-kw^{-1})_+^p] = w^{-p}\E[(w\G^w_{(k, 0)}-k)_+^p] \\
&\ge pk^{p/2}\int_{0}^{\infty} \P\{w\G^w_{(k, 0)} \ge k + k^{1/2}s\} s^{p-1} \dd s \ge p k^{p/2} \int_{s_0}^{T_0 k^{1/2}} \exp\{-A_0s^2\}s^{p-1} \dd s \\
&= pk^{p/2}A_0^{-p/2}\int_{A_0^{1/2}s_0}^{A_0^{1/2}T_0 k^{1/2}} \exp\{-t^2\}t^{p-1}\dd t.
\end{split}
\end{align}
The second inequality above drops the factor $w^{-p} \ge 1$. The last inequality in \eqref{Eq-56} invokes Lemma \ref{LExpTLB}(a), and holds for large enough $N_0 = N_0(T_0)$. 

For the moment, continue with the assumption that $p \ge p_0$ for some constant $p_0 = p_0(\delta, \epsilon, T) \ge 2$ to be chosen shortly. Increase $N_0 = N_0(p_0, T)$ if necessary to have $p_0 \le T\norm{\bfv}$ so that the inequality $p \ge p_0$ can hold. The first and last integrals in the next display are bounded via Lemma \ref{LIntBd}(b)-(c).  
\begin{align}
\label{Eq-77}
\begin{split}
&\int_0^\infty \exp\{-t^2\}t^{p-1} \dd t \ge b_0^p p^{p/2}, \\
&\int_0^{A_0^{1/2} s_0} \exp\{-t^2\} t^{p-1} \dd t \le \frac{1}{p}A_0^{p/2}s_0^p \le \frac{1}{4}b_0^p p^{p/2}, \quad \text{ and }\\
&\int_{A_0^{1/2}T_0k^{1/2}}^\infty \exp\{-t^2\}t^{p-1} \dd t \\
&\le B_0^p p^{p/2} \exp\left\{-d_0 (A_0 T_0^2k-\lc p/2 \rc)_+\min \left\{\frac{A_0 T_0^2k}{\lc p/2 \rc} -1, 1\right\}\right\} \le \frac{1}{4}b_0^p p^{p/2} 
\end{split}
\end{align}
for some absolute positive constants $b_0$, $B_0$ and $d_0$. The final inequality of the middle line in \eqref{Eq-77} holds provided that $p_0 = p_0(A_0, b_0, s_0)$ is chosen sufficiently large. To reach the very last inequality in \eqref{Eq-77}, first pick a large absolute constant $L_0 \ge 2$ such that $4B_0 \exp\{-d_0 L_0\} \le b_0$. Then, for sufficiently large $T_0 = T_0(\epsilon, K_0, L_0, R_0, T)$,  
\begin{align}
\label{Eq-78}
\begin{split}
A_0 T_0^2 k \ge A_0 T_0^2 R_0^{-1} \min \{K_0 p^{1/3}\norm{\bfv}^{2/3}, \epsilon \norm{\bfv}\} \ge 2L_0 T \norm{\bfv} \ge 2L_0 p
\end{split}
\end{align}
in view of \eqref{Eq-76b}, and because $p \le T\norm{\bfv}$ and $A_0$ is nondecreasing in $T_0$. The final inequality in \eqref{Eq-77} now holds due to \eqref{Eq-78} and the choice of $L_0$. Putting together the bounds in \eqref{Eq-77} and \eqref{Eq-78}, and appealing to \eqref{Eq-76} lead to  
\begin{align}
\label{Eq-79}
\E[|\G^{w}_{(k, 0)}-kw^{-1}|^p] \ge \frac{1}{2}k^{p/2}A_0^{-p/2}b_0^p p^{p/2 + 1} \ge c_0^p (\ShpMin_{\bfv}-w)^{p/2} p^{p/2}\norm{\bfv}^{p/2}
\end{align}
for $p \in [p_0, T\norm{\bfv}]$ and some constant $c_0 = c_0(\delta, \epsilon, T) > 0$. 

The case $p < p_0$ is handled similarly to \eqref{Eq-35} as follows. After possibly reducing $c_0$, 
\begin{align}
\label{Eq-80}
\begin{split}
&\E[|\G^{w}_{(k, 0)}-kw^{-1}|^p] \ge \E[(w\G^w_{(k, 0)}-k)_+^p] \ge s_0^{p}k^{p/2}\P\{w\G^{w}_{(k, 0)} \ge k +  k^{1/2}s_0\} \\
&\ge s_0^p k^{p/2} \exp\{-A_0 s_0^2\} \ge s_0^p R_0^{-p/2}\exp\{-A_0 s_0^2\}p_0^{-p/2} (\ShpMin_{\bfv}-w)^{p/2}p^{p/2} \norm{\bfv}^{p/2} \\
&\ge c_0^p (\ShpMin_{\bfv}-w)^{p/2}p^{p/2} \norm{\bfv}^{p/2}, 
\end{split}
\end{align}
which extends \eqref{Eq-79} to all $p \in [1, T\norm{\bfv}]$. The penultimate inequality in \eqref{Eq-80} relies on \eqref{Eq-76} and that $p \le p_0$. 

The next task is to show that the lower bound in \eqref{Eq-80} dominates the upper bound \eqref{Eq-57}. Pick $\epsilon_0 = \epsilon_0(\delta, \epsilon, T) \in (0, T)$ such that $\epsilon_0 \le \epsilon^3 K_0^{-3}$ and work with $N_0 = N_0(\epsilon_0) \ge \epsilon_0^{-1}$. Consider the case $1 \le p \le \epsilon_0 \norm{\bfv}$, which is not vacuous $N_0$ being sufficiently large. Then, due to the choices of $w$ and $\epsilon_0$, 
\begin{align}
c_0^p (\ShpMin_{\bfv}-w)^{p/2}p^{p/2}\norm{\bfv}^{p/2} \ge c_0^p K_0^{p/2}p^{2p/3}\norm{\bfv}^{p/3} \ge 2C_0^p p^{2p/3}\norm{\bfv}^{p/3} \label{Eq-81}
\end{align}
provided that $K_0 = K_0(C_0, c_0)$ is sufficiently large. 

Now it follows from \eqref{Eq-55}, \eqref{Eq-57}, \eqref{Eq-80} and \eqref{Eq-81} that 
\begin{align}
\label{Eq-82}
\begin{split}
\E[|\G_{\bfv}^{w, \hor}-\M_{\bfv}^w|^p] &\ge c_0^p (\ShpMin_{\bfv}-w)^{p/2}p^{p/2} \norm{\bfv}^{p/2}-C_0^p p^{2p/3}\norm{\bfv}^{p/3} \\ 
&\ge \frac{1}{2}c_0^p (\ShpMin_{\bfv}-w)^{p/2} p^{p/2}\norm{\bfv}^{p/2} \quad \text{ for } p \in [1, \epsilon_0 \norm{\bfv}]. 
\end{split}
\end{align} 
When $p \in [\epsilon_0\norm{\bfv}, T\norm{\bfv}]$, Jensen's inequality together with \eqref{Eq-82} also gives 
\begin{align}
\label{Eq-83}
\begin{split}
&\E[|\G_{\bfv}^{w, \hor}-\M_{\bfv}^w|^p] \ge \E[|\G_{\bfv}^{w, \hor}-\M_{\bfv}^w|^{\epsilon_0 \norm{\bfv}}]^{p\epsilon_0^{-1}\norm{\bfv}^{-1}} \\
&\ge 2^{-p}c_0^{-p} (\ShpMin_{\bfv}-w)^{p/2}(\epsilon_0\norm{\bfv})^{p/2}\norm{\bfv}^{p/2} \ge 2^{-p} c_0^{-p} \epsilon_0^{p/2}T^{-p/2} (\ShpMin_{\bfv}-w)^{p/2} p^{p/2}\norm{\bfv}^{p/2}. 
\end{split}
\end{align}
To complete the proof, combine \eqref{Eq-82} and \eqref{Eq-83}, and redefine the constant $c_0$ suitably. 
\end{proof}

\appendix

\renewcommand{\thesubsection}{\thesection.\arabic{subsection}}

\section{}
\label{AppA}
\subsection{A few basic bounds}

Some frequently used elementary bounds are collected in the next lemma. These bounds are immediate from formulas \eqref{EShp}, \eqref{EShpMin} and \eqref{ECurv}. 
\begin{lem}
\label{LMinEst}
Fix $\delta > 0$. The following bounds hold. 
\begin{enumerate}[\normalfont (a)]
\item $\norm{\bfx} \le \shp_{\bfx} \le 2 \norm{\bfx}$ for $\bfx \in \bbR_{>0}^2$.  
\item $a_0 \le \ShpMin_{\bfx} \le 1-a_0$ for $\bfx \in S_\delta$ for some constant $a_0 = a_0(\delta) > 0$. 
\item $\norm{\bfx}^{1/3} \le \curv_\bfx \le A_0 \norm{\bfx}^{1/3}$ for $\bfx \in S_\delta$ for some constant $A_0 = A_0(\delta) > 0$. 
\end{enumerate}
\end{lem}

\subsection{Comparison lemma}

The following standard lemma states a deterministic monotonicity property for the last-passage increments. See \cite[Lemma 6.2]{Rass_18} and \cite[Lemma 4.6]{Sepp_18_CGM} for different proofs. 

\begin{lem}
\label{LComp}
Let the LPP values $\{G_{\bfu, \bfv}: \bfu, \bfv \in \bbZ^2\}$ be defined as in \eqref{EBlkLPP}, over arbitrary real weights $\{ w_\bfv\}_{\bfv \in \bbZ^2}$.   Then the following inequalities hold for $\bfu, \bfv \in \bbZ^2$ with $\bfu \le \bfv$. 
\begin{enumerate}[\normalfont (a)]
\item $G_{\bfu, \bfv}-G_{\bfu+(1, 0), \bfv} \le G_{\bfu, \bfv+(0, 1)}-G_{\bfu+(1, 0), \bfv+(0, 1)}$. 
\item $G_{\bfu, \bfv}-G_{\bfu+(1, 0), \bfv} \ge G_{\bfu, \bfv+(1, 0)}-G_{\bfu+(1, 0), \bfv+(1, 0)}$. 
\item $G_{\bfu, \bfv}-G_{\bfu+(0, 1), \bfv} \le G_{\bfu, \bfv+(1, 0)}-G_{\bfu+(0, 1), \bfv+(1, 0)}$.
\item $G_{\bfu, \bfv}-G_{\bfu+(0, 1), \bfv} \ge G_{\bfu, \bfv+(0, 1)}-G_{\bfu+(0, 1), \bfv+(0, 1)}$. 
\end{enumerate}
\end{lem}

\subsection{Tail bounds for sums of i.i.d.\ exponentials}

Let $(X_i)_{i \in \bbZ_{>0}}$ be a sequence of independent $\Exp(1)$-distributed random real numbers. Our aim in this part is to record upper and lower bounds for the tail probabilities of the sums\footnote{Not to be confused with the cone $S_\delta$ defined at \eqref{ECone}.} $S_n = \sum_{i = 1}^n X_i$ for $n \in \bbZ_{>0}$. These bounds are obtained using some ideas within the proof of Cram\'{e}r's theorem presented, for example, in \cite[Section 2.2]{Demb_Zeit_10} and \cite[Section 2.4]{Sepp_Rass_15}. Although this argument is standard, we include a self-contained derivation of our tail bounds since we could not locate them in entirety in a reference. 

We begin with recalling that the l.m.g.f. of $X \sim \Exp(1)$ is given by 
\begin{align*}
\log \E[\exp\{\lambda X\}] = -\one_{\{\lambda < 1\}} \cdot \log (1-\lambda) + \one_{\{\lambda \ge 1\}} \cdot \infty \quad \text{ for } \lambda \in \bbR. 
\end{align*}
Define the associated \emph{Cram\'{e}r rate function} $I$ via convex conjugation: 
\begin{align}
\label{Eq-58}
\begin{split}
I(x) &= \sup_{\lambda \in \bbR} \{\lambda x - \log \E[\exp\{\lambda X\}]\} = \sup_{\lambda < 1} \{\lambda x + \log (1-\lambda)\} \\
&= \one_{\{x > 0\}} \cdot (x-1-\log x) + \one_{\{x \le 0\}} \cdot \infty \quad \text{ for } x \in \bbR. 
\end{split}
\end{align}
The next lemma provides some useful estimates for the rate function. 

\begin{lem}
\label{LRateEst}
Let $\epsilon \in (0, 1)$. There exist positive constants $a_0$, $b_0$, $A_0$ and $B_0 = B_0(\epsilon)$ such that the following inequalities hold. 
\begin{enumerate}[\normalfont (a)]
\item $I(1+x) \ge a_0 x\min \{x, 1\}$ for $x \ge 0$. 
\item $I(1-x) \ge b_0 x^2$ for $x \ge 0$. 
\item $I(1+x) \le A_0 x\min \{x, 1\}$ for $x \ge 0$. 
\item $I(1-x) \le B_0 x^2$ for $x \in [0, 1-\epsilon]$. 
\end{enumerate}
\end{lem}
\begin{proof}
Definition \eqref{Eq-58} shows that $\displaystyle I(1+x) = \int_0^x \frac{t \dd t}{1+t}$ for $x > -1$. Therefore, 
$$I(1-x) = \int_0^{-x} \frac{t \dd t}{1+t} \dd t = \int_{0}^x \frac{t \dd t }{1-t}\ge \int_0^x t \dd t = \frac{x^2}{2} \quad \text{ for } x \in [0, 1).$$
The preceding bound also extends to $x > 1$ since $I = \infty$ on $\bbR_{\le 0}$. Hence, (b). The second equality above also gives 
$$I(1-x) \le \frac{1}{\epsilon} \int_0^x t \dd t \le \frac{x^2}{2\epsilon} \quad \text{ for } x \in [0, 1-\epsilon],$$
proving (d). Next note that 
$$I(1+x) \ge \int_0^x \frac{t \dd t}{1 + 1/2} = \frac{x^2}{3} \quad \text{ for } x \in [0, 1/2]. $$
Since the function $t \mapsto t(1+t)^{-1}$ is increasing on $[0, \infty)$, one also has 
$$I(1+x) \ge \int_{1/3}^x \frac{t \dd t}{1+t}\dd t \ge \left(x-\frac{1}{3}\right)\cdot \frac{1}{4} \ge \frac{x}{12} \quad \text{ when } x > \frac{1}{2}.$$  Combining this with the previous display proves (a). Finally, to obtain (c), note that 
\begin{equation*}
I(1+x) \le \int_0^x \min \{t, 1\} \dd t = \min \left\{\frac{x^2}{2}, x\right\} = \frac{1}{2}x\min \{x, 1\}\quad \text{ for } x \ge 0. \qedhere
\end{equation*} 
\end{proof}

\begin{lem}
\label{LExpTUB}
There exist absolute positive constants $a_0$ and $b_0$ such that the following bounds hold for $n \in \bbZ_{>0}$ and $s \ge 0$. 
\begin{enumerate}[\normalfont (a)]
\item $\bfP\{S_n \ge n + sn^{1/2}\} \le \exp\{-a_0\min \{s^2, sn^{1/2}\}\}$.  
\item $\bfP\{S_n \le n - sn^{1/2}\} \le \exp\{-b_0s^2\}$. 
\end{enumerate}
\end{lem}
\begin{proof}
Let $n \in \bbZ_{>0}$ and $s \ge 0$. From the Chernoff-Cram\'{e}r bounds, one obtains that  
\begin{align*}
\P\{S_n \ge n + sn^{1/2}\} &\le \exp\{-nI(1 + sn^{-1/2})\}, \\
\P\{S_n \le n - sn^{1/2}\} &\le \exp\{-nI(1-sn^{-1/2})\}.  
\end{align*}
(Alternatively, these bounds can be taken from \cite[Theorem 5.1]{Jans_18}). The result is then immediate from Lemma \ref{LExpTUB}(a)-(b). 
\end{proof}

We now derive complementary lower bounds for the probabilities in Lemma \ref{LExpTUB} via the classical change-of-measure argument. For $n \in \bbZ_{>0}$ and $\mu < 1$, let $\bfQ_{n, \mu}$ denote the \emph{tilted} probability measure given by
\begin{align}
\bfQ_{n, \mu}\{\cdot\} = \frac{\E[\exp\{\mu S_n\}\one\{\cdot\}]}{\E[\exp\{\mu S_n\}]}. \label{Eq-63}
\end{align}
The next calculation verifies that the sequence $(X_i)_{i \in [n]}$ is independent with $\Exp(1-\mu)$ marginals under $\bfQ_{n, \mu}$. Since $X_i \sim \Exp(1)$ and are independent for $i \in [n]$ under $\bfP$, 
\begin{align}
\label{Eq-90}
\begin{split}
&\bfQ_{n, \mu}\{X_i > x_i \text{ for } i \in [n]\} = \frac{\E[\exp\{\mu S_n\}\prod_{i \in [n]}\one\{X_i > x_i\}]}{\E[\exp\{\mu S_n\}]} \\
&= \prod_{i \in [n]} \{(1-\mu)\E[\exp\{\mu X_i\}\one\{X_i > x\}]\} = \prod_{i \in [n]} \left\{(1-\mu) \int_0^\infty \one\{t > x_i\}e^{-(1-\mu)t} \dd t\right\} \\
&= \prod_{i \in [n]} \P\{X_i > (1-\mu) x_i\} = \P\{X_i > (1-\mu)x_i \text{ for } i \in [n]\}. 
\end{split}
\end{align}
The last line uses that $X_i (1-\mu)^{-1} \sim \Exp(1-\mu)$ and are independent for $i \in [n]$ under $\bfP$. 

Through the tilted measure, we obtain the following lemma. In fact, part (b) is not needed in the main text but included for completeness. 

\begin{lem}
\label{LExpTLB}
Let $T > 1$. There exist positive constants $A_0$, $B_0 = B_0(T)$, $N_0 = N_0(T)$ and $s_0 = s_0(T)$ such that the following bounds hold for $n \ge N_0$ and $s \ge s_0$.  
\begin{enumerate}[\normalfont (a)]
\item $\bfP\{S_n \ge n+sn^{1/2}\} \ge \exp\{-A_0\min \{s^2, sn^{1/2}\}\}$ when $s \le Tn^{1/2}$. 
\item $\bfP\{S_n \le n-sn^{1/2}\} \ge \exp\{-B_0s^2\}$ when $s \le (1-T^{-1})n^{1/2}$. 
\end{enumerate}
\end{lem}
\begin{proof}
Let $s_0  = s_0(T) > 0$ denote a constant to be chosen. Pick $N_0 = N_0(s_0, T) > 0$ sufficiently large such that $s_0 \le TN_0^{1/2}$. Let $n \ge N_0$ and $s \in [s_0, Tn^{1/2}]$. To obtain (a), 
choose $\mu \in (0, 1)$ such that $\mu \cdot (n^{1/2} + s+s_0) = s+s_0$, and consider the event 
\begin{align}
E_{n, s} = \{|S_n-n(1-\mu)^{-1}| \le s_0n^{1/2}\}  = \{S_n \in [n+sn^{1/2}, n+(s+2s_0) n^{1/2}]\}. \label{Eq-62}
\end{align}
Using the choice of $\mu$ along with definitions \eqref{Eq-63}, \eqref{Eq-62} and \eqref{Eq-58}, one develops
\begin{align}
\label{Eq-65}
\begin{split}
&\P\{S_n \ge n + sn^{1/2}\} \ge  \P\{E_{n, s}\} \ge \exp\{-\mu n - \mu (s+2s_0) n^{1/2}\} \E[\exp\{\mu S_n\}\one_{E_{n, s}}] \\
&= \exp\{-(s+s_0+\mu s_0)n^{1/2}\} \E[\exp\{\mu S_n\}]\bfQ_{n, \mu}\{E_{n, s}\} \\
&= \exp\{-(s+s_0+\mu s_0)n^{1/2} - n \log (1-\mu)] \bfQ_{n, \mu}\{E_{n, s}\} \\
&= \exp\{-\mu s_0 n^{1/2}-nI(1+(s+s_0)n^{-1/2})\}\bfQ_{n, \mu}\{E_{n, s}\}. 
\end{split}
\end{align}
We continue with bounding the terms on the last line of \eqref{Eq-65}. First, note that  
\begin{align}
\mu n^{1/2} = \frac{(s + s_0)n^{1/2}}{n^{1/2}+s+s_0} \le 2s. \label{Eq-69}
\end{align}
By Lemma \ref{LRateEst}(c) and since $s \ge s_0$, one also has 
\begin{align}
\label{Eq-66}
\begin{split}
I(1+(s+s_0) n^{-1/2}) \le A_0  s n^{-1/2} \min \{s n^{-1/2}, 1\}
\end{split}
\end{align}
for some absolute constant $A_0 > 0$. Finally, using \eqref{Eq-90} and Lemma \ref{LExpTUB}, bound the probability of the complement $E_{n, s}^c$ as follows. 
\begin{align}
\label{Eq-64}
\begin{split}
&\bfQ_{n, \mu}\{E_{n, s}^c\} = \bfQ_{n, \mu}\{|S_n-n(1-\mu)^{-1}| > s_0n^{1/2}\} = \bfP\{|S_n-n| > s_0(1-\mu) n^{1/2}\} \\
&\le 2\exp\{-a_0 \min \{s_0^2 (1-\mu)^2, s_0(1-\mu)n^{1/2}\}\} \le \frac{1}{2}
\end{split}
\end{align}
for some absolute constant $a_0 > 0$. For the last inequality in \eqref{Eq-64}, observe that $1-\mu = n^{1/2} \{n^{1/2} + s+s_0\}^{-1} \ge \{1+2T\}^{-1}$ and pick $s_0 = s_0(a_0, T)$ sufficiently large. Now collect the bounds \eqref{Eq-69}, \eqref{Eq-66} and \eqref{Eq-64} to conclude from \eqref{Eq-65} that 
\begin{align}
\label{Eq-70}
\begin{split}
\bfP\{S_n \ge n + sn^{1/2}\} &\ge \frac{1}{2} \exp\{-A_0 \min \{s^2, sn^{1/2}\}-2s_0 s\} \\
&\ge \exp\{-2A_0 \min \{s^2, sn^{1/2}\}\}, 
\end{split}
\end{align}
which proves (a). The last step of \eqref{Eq-70} holds provided that $s_0 = s_0(A_0)$ and $N_0 = N_0(A_0, s_0)$ are both sufficiently large. 

A few modifications are needed in the preceding argument to also derive (b). Assume now that $s \in [s_0, (1-T^{-1})n^{1/2}]$ and $n \ge N_0$, working with sufficiently large $N_0 = N_0(s_0, T)$ such that $s_0 \le (1-T^{-1})N_0^{1/2}$. After possibly increasing $N_0$ further, $2Ts_0 \le N_0^{1/2}$. Pick $\nu < 0$ such that $\nu \cdot (n^{1/2} - s-s_0) = -s-s_0$ noting that $s+s_0 \le (1-T^{-1})n^{1/2} + 2^{-1}T^{-1}N_0^{1/2}\le (1-2^{-1}T^{-1}) n^{1/2}$. As in \eqref{Eq-62}, define the event 
\begin{align*}
F_{n, s} = \{|S_n - n(1-\nu)^{-1}| \le s_0n^{1/2}\} = \{S_n \in [n - (s+2s_0) n^{1/2}, n-sn^{1/2}]\}. 
\end{align*}
Proceeding similarly to \eqref{Eq-65}, one obtains that 
\begin{align}
\label{Eq-75}
\begin{split}
\P\{S_n \le n - sn^{1/2}\} &\ge \P\{F_{n, s}\} \ge \exp\{-\nu (n - (s+2s_0)n^{1/2})\} \E[\exp\{\nu S_n\} \one_{F_{n, s}}]\\
&= \exp\{\nu s_0 n^{1/2} + (s+s_0)n^{1/2} - n \log (1-\nu)\}\bfQ_{n, \nu}\{F_{n, s}\}\\
&= \exp\{\nu s_0 n^{1/2} - n I(1-(s+s_0) n^{-1/2})\}\bfQ_{n, \nu}\{F_{n, s}\}. 
\end{split}
\end{align}
Due to the choices of $\nu$ and $N_0$, 
\begin{align}
\nu n^{1/2} =  - \frac{(s+s_0)n^{1/2}}{n^{1/2}-s-s_0} \ge -4Ts. \label{Eq-67}
\end{align}
By Lemma \eqref{LRateEst}(d) and since $s \ge s_0$, 
\begin{align}
\label{Eq-68}
I(1-(s+s_0)n^{-1/2}) \le B_0 s^2 n^{-1} 
\end{align}
for some constant $B_0 = B_0(T) > 0$. Appealing to Lemma \ref{LExpTUB} also gives 
\begin{align}
\label{Eq-74}
\begin{split}
\bfQ_{n, \nu}\{F_{n, s}^c\} &\le \exp\{-a_0 \min \{s_0^2(1-\nu)^2, s_0 (1-\nu)n^{1/2}\}\}\} \\
&\le \exp\{-a_0 \min \{s_0^2, s_0 n^{1/2}\}\} \le \frac{1}{2}
\end{split}
\end{align}
for sufficiently large $s_0 = s_0(a_0)$. Now (b) follows from \eqref{Eq-75}, \eqref{Eq-67}, \eqref{Eq-68} and \eqref{Eq-74} upon choosing $s_0 = s_0(B_0, T)$ sufficiently large. 
\end{proof}

For each $n \in \bbZ_{>0}$, the sum $S_n$ has the gamma distribution with density $s \mapsto \one_{\{s > 0\}}\dfrac{s^{n-1}e^{-s}}{(n-1)!}$. Therefore, 
\begin{align}
\label{EIntMon}
\int_{x}^\infty t^{n-1}\exp\{-t\} \dd t = (n-1)! \P\{S_n \ge x\} \quad \text{ for } n \in \bbZ_{>0} \text{ and } x \in \bbR_{\ge 0}.  
\end{align}
The following integral bounds are immediate from \eqref{EIntMon} and Lemmas \ref{LExpTUB} and \ref{LExpTLB}.  

\begin{lem}
\label{LGamTB}
Let $T > 1$. There exist positive constants $a_0$, $b_0$, $A_0 = A_0$, $B_0 = B_0(T)$, $s_0 = s_0(T)$ and $N_0 = N_0(T)$ such that the following bounds hold. 
\begin{enumerate}[\normalfont (a)]
\item 
$\displaystyle \int_{n + sn^{1/2}}^\infty t^{n-1}\exp\{-t\} \dd t \le (n-1)!\exp\{-a_0 \min \{s^2, sn^{1/2}\}\}$ for $n \in \bbZ_{>0}$ and $s \ge 0$. 
\item 
$\displaystyle \int_{0}^{n-sn^{1/2}} t^{n-1}\exp\{-t\} \dd t \le (n-1)! \exp\{-b_0 s^2\}$ for $n \in \bbZ_{>0}$ and $s \ge 0$. 
\item 
$\displaystyle \int_{n + sn^{1/2}}^\infty t^{n-1}\exp\{-t\} \dd t \ge (n-1)!\exp\{-A_0 \min \{s^2, sn^{1/2}\}\}$ for $n \in \bbZ_{\ge N_0}$ and $s \in [s_0, Tn^{1/2}]$. 
\item 
$\displaystyle \int_{0}^{n-sn^{1/2}} t^{n-1}\exp\{-t\} \dd t \ge (n-1)! \exp\{-B_0 s^2\}$ for $n \in \bbZ_{\ge N_0}$ and $s \in [s_0,  (1-T^{-1})n^{1/2}]$. 
\end{enumerate}
\end{lem}

\subsection{Integral bounds for stretched exponentials}

\begin{lem}
\label{LFacBnd}
$n^{n}e^{-n} \le n! \le (2n+1) n^{n}e^{-n}$ for  $n \in \bbZ_{\ge 0}$. 
\end{lem}

\begin{lem}
\label{LIntBd}
Let $x \ge 0$, $p \ge 1$, $q \in (0, p]$ and $r = p/q$. The following bounds hold for some absolute constants $a_0 > 0$ and $b_0 > 0$, and some $C_0 = C_0(q) > 0$ and $c_0 = c_0(q) > 0$ depending only on $q$. 
\begin{enumerate}[\normalfont (a)]
\item $\displaystyle \int_0^\infty t^{p-1} \exp\{-t^q\} \dd t \le C_0^p p^r$. 
\item $\displaystyle \int_0^\infty t^{p-1} \exp\{-t^q\} \dd t \ge c_0^p p^r$.
\item $\displaystyle \int_x^\infty t^{p-1} \exp\{-t^q\} \dd t \le C_0^p p^{r}\exp\left\{-a_0(x^q-\lc r \rc)_+ \min \left\{\frac{x^q}{\lc r \rc}-1, 1\right\}\right\}.$
\item $\displaystyle \int_0^x t^{p-1} \exp\{-t^q\} \dd t \le C_0^p p^{r}\exp\left\{-\frac{b_0(x^q-\lf r \rf)_-^2}{\lf r \rf}\right\}.$
\end{enumerate}
\end{lem}
\begin{proof}
Abbreviate $k = \lc r \rc$ and $l = \lf r \rf \ge 1$. Starting with the substitution $s = t^q$ and then appealing to Lemma \ref{LGamTB}(a), one obtains that
\begin{align}
\label{Eq-28}
\begin{split}
&\int_x^\infty t^p \exp\{-t^q\} \dd t = \frac{1}{q}\int_{x^q}^\infty s^{r-1} \exp\{-s\} \dd s \le \frac{1}{2q} \int_{x^q}^\infty (s^{k - 1} + s^{l-1}) \exp\{-s\} \dd s \\
&\le \frac{1}{q}\int_{x^q}^\infty s^{k - 1} \exp\{-s\} \dd s \le \frac{(k-1)!}{q} \cdot \exp\left\{-a_0 (x^q-k)_+ \min \left\{\frac{x^q}{k}-1, 1\right\}\right\}. 
\end{split}
\end{align}
for some absolute constant $a_0 > 0$. The second inequality in \eqref{Eq-28} uses that $k \ge l$ and that the right-hand side in \eqref{EIntMon} is nondecreasing in $n$. In the same vein,
\begin{align*}
\begin{split}
&\int_0^x t^p \exp\{-t^q\} \dd t \le \frac{1}{2q} \int_0^{x^q} (s^{k - 1} + s^{l-1}) \exp\{-s\} \dd s \\
&\le \frac{(k-1)!}{q(l-1)!}\int_{0}^{x^q} s^{l - 1} \exp\{-s\} \dd s \le \frac{(k-1)!}{q} \cdot \exp\left\{-\frac{b_0(x^q-l)_-^2}{l} \right\}. 
\end{split}
\end{align*}
for some absolute constant $b_0 > 0$. To complete the proofs of (c) and (d), it suffices to note from Lemma \ref{LFacBnd} that if $k > 1$ then 
\begin{align*}
(k-1)! \le (2k-1)(k-1)^{k-1}e^{-k+1} \le (2r+1)r^r e \le C_0^p p^{r}
\end{align*}
for sufficiently large $C_0 = C_0(q) > 0$. Part (a) is the special case $x = 0$ of both (c) and (d).  To deduce (b), use the first equality in \eqref{Eq-28} with $x = 0$ and then \eqref{EIntMon} to obtain 
\begin{align}
\label{Eq-29}
\begin{split}
\int_0^\infty t^p \exp\{-t^q\} \dd t &= \frac{1}{q}\int_{0}^\infty s^{r-1} \exp\{-s\} \dd s \ge \frac{1}{q} \int_1^\infty s^{l-1} \exp\{-s\}\dd s \\
&\ge \frac{1}{q} \int_0^\infty s^{l-1} \exp\{-s\}\dd s - \frac{1}{q} \int_0^1 \exp\{-s\} \dd s = \frac{(l-1)!-1+e^{-1}}{q}.  
\end{split}
\end{align}
When $l > 1$, the last expression in \eqref{Eq-29} is at least 
\begin{align*}
\frac{(l-1)!}{2q} \ge \frac{(l-1)^{l-1}e^{-l+1}}{2q} \ge \frac{r^{r-2}e^{-r}}{2q4^{r-2}} \ge c_0^p p^r 
\end{align*}
for some sufficiently small $c_0 = c_0(q) > 0$. The first step above relies on Lemma \ref{LFacBnd}, while the second step uses the inequalities $l \ge r-1$ and $l-1 \ge r/4$, which hold since $r \ge l \ge 2$ in the present case. This completes the proof of (b). 
\end{proof}

\subsection{A maximal inequality}

\begin{lem}
\label{LExpMaxIneq}
Let $n \in \bbZ_{>0}$ and $a, b \in \bbR_{>0}$. Let $\{X_i, Y_i: i \in [n]\}$ be a collection of independent random reals such that $X_i \sim \Exp\{a\}$ and $Y_i \sim \Exp\{b\}$ for $i \in [n]$. Let 
\begin{align}
M_k = \sum_{i \in [k]}\left(X_i-Y_i - \E[X_i-Y_i]\right) = \sum_{i \in [k]}\left(X_i-Y_i-\frac{1}{a}+\frac{1}{b}\right) \quad \text{ for } k \in [n]. \label{Eq-34}
\end{align}
Then the following bound hold with $C = \min \bigg\{\dfrac{a}{4}, \dfrac{1}{a^2} + \dfrac{1}{b^2}\bigg\}$. 
\begin{align*}
\P\left\{\max \limits_{k \in [n]} M_k\ge x\right\} &\le \exp\left\{-Cx \min\left\{\dfrac{x}{n}, 1\right\}\right\} \quad \text{ for } x \in \bbR_{\ge 0} 
\end{align*}
\end{lem}
\begin{proof}
The sequence $(M_k)_{k \in [n]}$ is a martingale. Let $\mu \in [0, c]$ where $c = \dfrac{a}{2}$. Since the function $t \mapsto \exp\{\mu t\}$ is convex on $\bbR$ and 
\begin{align*}
\E[\exp\{\mu M_k\}] = \exp\bigg\{k \log \bigg(\frac{a}{a-\mu}\bigg)-\frac{\mu k}{a} + k\log\bigg(\frac{b}{b+\mu}\bigg)+\frac{\mu k}{b}\bigg\} < \infty, 
\end{align*}
the sequence $(\exp\{\mu M_k\})_{k \in [n]}$ is a submartingale. Then applying Doob's submartingale inequality leads to 
\begin{align}
\label{Eq-24}
\begin{split}
\P\left\{\max_{k \in [n]} M_k \ge x\right\} &= \P\left\{\max_{k \in [n]} \exp\{\mu M_k\} \ge \exp\{\mu x\}\right\} \le \E[\exp\{\mu M_n\}] \exp\{-\mu x\} \\
&= \exp\bigg\{-n \log \bigg(1-\frac{\mu}{a}\bigg)-\frac{\mu n}{a}- n\log\bigg(1 + \frac{\mu}{b}\bigg)+\frac{\mu n}{b}-\mu x\bigg\} \\
&\le \exp\bigg\{\frac{1}{2}\mu^2 dn- \mu x\bigg\}, 
\end{split}
\end{align}
where $d = \dfrac{2}{a^2} + \dfrac{2}{b^2}$. The final step uses the inequality $-\log (1+t)+t \le t^2$ for $t \in [-1/2, \infty)$. With $\mu = \min \bigg\{\dfrac{x}{dn}, c\bigg\}$, the last exponent in \eqref{Eq-24} attains its optimal value 
\begin{align*}
&-\frac{x^2}{2dn}\one\{x \le cdn\}  + \bigg(\frac{1}{2}c^2 d n - cx\bigg)\one\{x > cdn\} \\
&\le -\frac{x^2}{2dn}\one\{x \le cdn\}  - \frac{cx}{2}\one\{x > cdn\} = - \frac{1}{2} \min \bigg\{\frac{x^2}{dn}, cx\bigg\} \le - Cx \min \bigg\{\frac{x}{n}, 1\bigg\}, 
\end{align*}
completing the proof. 
\end{proof}


\bibliographystyle{habbrv}
\bibliography{Refs} 

\end{document}